\documentclass[12pt]{amsart}
\usepackage{graphicx}
\usepackage{mathrsfs}
\usepackage{amsmath, amssymb, amsthm, latexsym}
\usepackage{amssymb,amscd,amsmath}
\usepackage{latexsym}
\usepackage{MnSymbol}
\usepackage[center]{caption}
\usepackage{tikz}
\usepackage{float}
\newcounter{braid}
\newcounter{strands}

\DeclareMathAlphabet{\bsf}{OT1}{cmss}{bx}{n}

\pgfkeyssetvalue{/tikz/braid height}{1cm}
\pgfkeyssetvalue{/tikz/braid width}{1cm}
\pgfkeyssetvalue{/tikz/braid start}{(0,0)}
\pgfkeyssetvalue{/tikz/braid colour}{black}
\pgfkeys{/tikz/strands/.code={\setcounter{strands}{#1}}}

\makeatletter
\def\cross{%
  \@ifnextchar^{\message{Got sup}\cross@sup}{\cross@sub}}

\def\cross@sup^#1_#2{\render@cross{#2}{#1}}

\def\cross@sub_#1{\@ifnextchar^{\cross@@sub{#1}}{\render@cross{#1}{1}}}

\def\cross@@sub#1^#2{\render@cross{#1}{#2}}

\def\render@cross#1#2{
  \def\strand{#1}
  \def\crossing{#2}
  \pgfmathsetmacro{\cross@y}{-\value{braid}*\braid@h}
  \pgfmathtruncatemacro{\nextstrand}{#1+1}
  \foreach \thread in {1,...,\value{strands}}
  {
    \pgfmathsetmacro{\strand@x}{\thread * \braid@w}
    \ifnum\thread=\strand
    \pgfmathsetmacro{\over@x}{\strand * \braid@w + .5*(1 - \crossing) * \braid@w}
    \pgfmathsetmacro{\under@x}{\strand * \braid@w + .5*(1 + \crossing) * \braid@w}
    \draw[braid] \pgfkeysvalueof{/tikz/braid start} +(\under@x pt,\cross@y pt) to[out=-90,in=90] +(\over@x pt,\cross@y pt -\braid@h);
    \draw[braid] \pgfkeysvalueof{/tikz/braid start} +(\over@x pt,\cross@y pt) to[out=-90,in=90] +(\under@x pt,\cross@y pt -\braid@h);
    \else
    \ifnum\thread=\nextstrand
    \else
     \draw[braid] \pgfkeysvalueof{/tikz/braid start} ++(\strand@x pt,\cross@y pt) -- ++(0,-\braid@h);
    \fi
   \fi
  }
  \stepcounter{braid}
}

\tikzset{braid/.style={double=\pgfkeysvalueof{/tikz/braid colour},double distance=1pt,line width=2pt,white}}

\newcommand{\braid}[2][]{%
  \begingroup
  \pgfkeys{/tikz/strands=2}
  \tikzset{#1}
  \pgfkeysgetvalue{/tikz/braid width}{\braid@w}
  \pgfkeysgetvalue{/tikz/braid height}{\braid@h}
  \setcounter{braid}{0}
  \let\sigma=\cross
  #2
  \endgroup
}
\makeatother

\input xypic
\newtheorem{theorem}{Theorem}
\newtheorem{proposition}[theorem]{Proposition}

\newtheorem{lemma}[theorem]{Lemma}

\newtheorem{corollary}[theorem]{Corollary}
\newtheorem{definition}[theorem]{Definition}


\def\dash{\text{-}}

\def\Z{\mathbb{Z}}

\def\C{\mathbb{C}}

\def\Q{\mathbb{Q}}

\def\R{\mathbb{R}}
\def\C{\mathbb{C}}

\def\N{\mathbb{N}}

\def\Pi{\mathbb{P}^{\infty}}

\def\qed{\hfill$\square$\medskip}

\def\Zpk{\mathbb{Z}/p^{k}}
\def\Zpk1{\mathbb{Z}/p^{k-1}}

\newcommand{\rref}[1]{(\ref{#1})}

\newcommand{\beg}[2]{\begin{equation}\label{#1}#2\end{equation}}
\def\r{\rightarrow}
\def\mc{\mathscr{C}}

\def\sl2{\widetilde{SL_{2}(\Z)}}

\title[Derived representation theory]{Derived representation theory of Lie algebras and stable homotopy categorification 
of $sl_k$}
\author{Po Hu, Igor Kriz and Petr Somberg}
\thanks{The authors acknowledge support by grant GA\,CR P201/12/G028
and NSF grants DMS 1102614, DMS 110434. The second author was also supported by 
Simons Collaboration Grant 403297}

\begin{document}

\begin{abstract}
We set up foundations of representation theory over $S$, the sphere spectrum, which is 
the ``initial ring'' of stable homotopy theory. 
In particular, we treat $S$-Lie algebras and their representations, characters, $gl_n(S)$-Verma modules
and their duals, Harish-Chandra pairs and Zuckermann functors. 
As an application, we construct a Khovanov $sl_k$-stable homotopy
type with a large prime hypothesis, which is a new link invariant, using a stable
homotopy analogue of the method of J.Sussan. 
\end{abstract}

\maketitle

\tableofcontents

\section{Introduction}

The primary result of this paper is to set up the foundations, and do some basic computations in representation
theory over the sphere spectrum $S$. The motivation which guided the present investigation originated in 
knot theory. An important new direction of homotopy theory called {\em categorification} was started around
the year 2000 by
Khovanov \cite{khovanov} with a categorification of the Jones polynomial, which became known as
Khovanov homology. This invariant is easily defined directly (for a particularly neat definition, see D.Bar-Natan
\cite{barnatan}), but actually comes from the categorification of representations of $sl_2$ (see Bernstein,
Frenkel, Khovanov \cite{bfk}). Here ``categorification'' means the introduction of a certain categories 
of chain complexes whose
Grothendieck groups, tensored with $\C$, are the representations in question. Morphisms of representations
correspond to functors, and as a result, instead of a ``knot polynomial'', we obtain a chain complex whose 
homology is Khovanov homology of the given type. Many more categorifications appeared since. From the point
of view of the present paper, the most important one is the paper \cite{sussan} by Joshua Sussan, which
gave, by analogy with \cite{bfk}, a categorification of a certain part of the representation theory of $sl_k$,
thus defining what became known as {\em $sl_k$-Khovanov homology}.

\vspace{3mm}
A major development in knot theory, and a major connection between that field and homotopy theory, 
was started in 2012 in the paper \cite{lsarkar} by Robert Lipshitz and 
Sucharit Sarkar, in which the authors defined what they termed an {\em $sl_2$-Khovanov stable homotopy type}, i. e. a 
spectrum $\mathcal{X}(L)$ for a link diagram $L$, whose homology is Khovanov homology of $L$. A different construction,
using an abstract approach to topological quantum field theory, was obtained shortly afterward
by Hu, I.Kriz and D.Kriz
in \cite{hkk}. 
Lipshitz and Sarkar \cite{lsteenrod} also showed that the $sl_2$-Khovanov stable homotopy type produces non-trivial
spectra in the sense of stable homotopy theory.
The equivalence of the constructions of \cite{lsarkar} and \cite{hkk} was recently proved by T. Lawson, R.Lipshitz and S.Sarkar in \cite{lls}, see also \cite{llsc}, with further applications to knot theory. 

\vspace{3mm}
The constructions \cite{lsarkar, hkk, lls, llsc} of the $sl_2$-Khovanov stable homotopy type all substantially
use the fact that the construction of $sl_2$-Khovanov homology is completely elementary, and uses no 
non-trivial representation theory. This is due to the fact that the representation theory of $sl_2$ is very simple. 
The authors of the present paper set out to define an $sl_k$-Khovanov stable
homotopy type by finding a stable homotopy analogue of J.Sussan's method \cite{sussan}. 
This is the main result of the present paper. There are several limitations of this result, which we need to mention. 
First of all, in this paper, we do not make any knot computations: we felt that the very deep connection
between concepts of stable homotopy theory and representation theory we needed to probe were enough for
the subject of one paper, and we postponed any knot computations to future work. The other limitation
which should be mentioned concerns any knot stable homotopy type, including the $sl_2$-Khovanov homotopy type. 
Khovanov homology actually has another grading, making it more precisely a categorification of the
representation theory of a quantum group. This structure, which is also established in \cite{sussan}, appears to be lost
by lifting to stable homotopy theory. Accordingly, also the present paper only concerns the stable 
homotopy categorification 
of $sl_k$-representations, and not the corresponding quantum group. Finally, and this is perhaps the
most material restriction is that we work at a (linearly) large prime with respect to 
$k$. This is needed to 
remove some difficulties specific to modular representation theory, which at present we do not know how to address 
in stable homotopy theory, although we feel that, of course, the modular story will ultimately be the most interesting.
Still, the large prime story is non-trivial to set up, and the case $k=2$ shows that it is also non-trivial, as there
are knots with odd torsion in their Khovanov homology (the smallest one being the
$(5,6)$-torus knot). 

\vspace{3mm}
Finally, it should be mentioned that there are other possible approaches to an $sl_k$-homotopy type. For example,
$sl_k$-Khovanov homology was also constructed by Khovanov and Rozansky \cite{khovr,khovr2} using matrix 
factorization. Efforts to construct an $sl_k$-stable homotopy type, i. e. a spectrum whose homology is 
$sl_k$-Khovanov homology, using this method were 
made in \cite{schutz}. Our present program of a construction using representation theory over $S$ took off in 2014
after conversations with Jack Morava. The concept of representation theory over $S$ is of independent interest
(see for example Lurie \cite{lurie}; there is also more recent unpublished joint work of Gaitsgory and Lurie).

\vspace{3mm}
To describe what type of mathematics we get into when constructing a Khovanov $sl_k$-homotopy type, 
we first describe the approach of Sussan \cite{sussan}, which follows the method of Bernstein, Frenkel and
Khovanov \cite{bfk} for the case of
$k=2$. Their approach is to consider the derived category of the Bernstein-Gelfand-Gelfand (BGG) category 
$\mathcal{O}$ of representations of $gl_n$,
cf. \cite{humphreys}. Here $n$ is another, typically larger natural number without direct relationship to $k$. In
the context of Khovanov homology, it represents, roughly, the number of strands of a horizontal slice of a 
link projection into the plane. The main feature of the BGG category $\mathcal{O}$
is that it is really a category of complex
representations of {\em Harish-Chandra pairs} $(H,gl_n)$ where $H$ is the exponentiation of a Cartan subalgebra,
say, consisting of the diagonal matrices. In the context of (derived) BGG, we further restrict to certain finite 
complexes of special objects called {\em Verma and co-Verma modules}, which is the right approach from the point
of view of categorification. This restriction, however, is actually less important, since it does not change $Ext$-groups (while passing
from representations of Lie algebras to Harish-Chandra pairs does change $Ext$-groups by ``rigidifying'' the action
of $H$). 

To get link invariants, we must actually go somewhat further, considering {\em parabolic BGG categories}
$\mathcal{O}_{\mathbf{p}}$
with respect to a parabolic Lie algebra $\mathbf{p}\subseteq gl_n$.
These are certain full subcategories of the categories of representations of Harish-Chandra pairs $(L,gl_n)$
where $L$ is the exponentiation of the Levi factor of the parabolic $\mathbf{p}$. The parabolic BGG categories 
decompose into {\em blocks}, and taking blocks with certain prescribed weights gives a categorification
of tensor products of exterior powers of the standard $k$-dimensional representation of $gl_k$, with 
morphisms of representations categorified by {\em left and right Zuckermann functors}, which are adjoints
to forgetful functors between Harish-Chandra pairs. To get a link complex of a tangle, one takes a {\em Khovanov cube}
of functors associated with a link projection (where, like in the original Khovanov homology, each crossing
is a cone, thus responsible for one additional dimension of the cube). If we have a link instead of a tangle,
we just get a chain complex, and its cohomology is $sl_k$-Khovanov homology. (Again, to be completely precise,
Sussan \cite{sussan} works with a quantum group instead of $gl_n$, thereby producing an additional grading,
but this structure is lost in stable homotopy.)

\vspace{3mm}
To obtain a stable homotopy version of the construction of \cite{sussan} we just described, one needs
to develop an analogue of the above described $\C$-valued representation theory over $S$, the sphere spectrum. 
Of course, even over $\Z$, the representation theory does not work nearly as nicely as over $\C$ (see \cite{Zforms}).
We circumvent this difficulty by working over {\em a large prime}. This means a prime linearly larger than 
the number $k$. It is important to note that we cannot assume to be at a prime larger than $n$, since $n$ depends
on the size of the knot or link. While working at a large prime is of course a major and undesirable restriction,
it allows non-trivial results over $\Z$, as seen in the examples of odd torsion in Khovanov homology. 

\vspace{3mm}
Over $S$, on the other hand, even at a large prime, we face a formidable array of technical and conceptual 
challenges, and in some sense, resolving them is the main contribution of the present paper. First, we must
recall what $S$ is - the stable sphere, which can also be thought of as the ``absolute generalized homology theory'',
or, as a homotopy theorist would say, {\em the sphere spectrum}.
It became clear in the 1990's that one can do algebra over $S$ by developing an analogue of the tensor product
of chain complexes which is commutative, associative and unital. The resulting field that began to open up
was called by Peter May {\em spectral algebra}, and later by Jacob Lurie {\em derived algebraic geometry}. 

\vspace{3mm}
The difficulty with spectra (i.e., a rigid category of generalized (co) homology theories) is that the geometric
model of the shift, which is compatible with the tensor product, is the {\em suspension},
which involves a choice of an intrinsically non-canonical coordinate. Because of this, the naive approach 
of the analogue of
the tensor product in spectra (which realizes products
in (co)homology theories, and was dubbed, by Frank Adams, the {\em smash product}),  lacks strict commutativity
and associativity, thereby restricting severely the type of algebraic constructions one could do. This is why
a highly technical procedure, developing a symmetric monoidal smash product, is needed before one can
do spectral algebra, derived algebraic geometry,
or representation theory over $S$. One such construction was given in \cite{ekmm},
but as is often the case in similar situations, other approaches emerged as well, and their equivalence was later
proved in \cite{scompare}.

\vspace{3mm}
In this paper, we use the foundations of \cite{ekmm}, which have certain technical advantages from
our point of view (for cognoscenti, the advantage is that in the Quillen model structure,
{\em every object is fibrant}). The first step toward representation theory is to define Lie algebras over $S$. That
was actually done by M.Ching \cite{ching} who constructed a stable homotopy {\em Lie operad}, using 
the Goodwillie derivative of the identity. It turns out that the concept of a Lie algebra representation that we need
is provided by the concept of {\em operad module}. This requires some thought because an analogous statement
is true for some operads and not others: the right concept of a module over an $E_\infty$-algebra
is an operad module,  while the operad module for associative algebra is a {\em bimodule}. A certain 
modification is needed to get left and right modules, and this is something we must come to terms with in
the present paper as well. 

\vspace{3mm}
Now that we know what a Lie algebra and a representation is, we need examples. Lie algebras, at least
from the point of view of what we need to model \cite{sussan}, are not a problem: $gl_n$, as well
as all its parabolic subalgebras, turn out to be easy, since they arise from a structure of {\em associative algebras}. 
The forgetful functor from associative algebras to Lie algebras has an analogue in \cite{ching}. (Therefore,
we also have its left adjoint, which is the $S$-analogue of the {\em universal enveloping algebra}.)
Examples of representations are less easy. While we have the ``standard'' representations, 
and they are important, to define Verma and co-Verma modules, we also need to model the concept of a {\em character}
of an ``abelian'' Lie algebra (namely the Cartan subalgebra of $gl_n$). An abelian Lie algebra is one which
comes not only by forgetting structure from an associative algebra, but from a {\em commutative} algebra.
Classically, it is obvious that the Lie bracket is then $0$, which is needed in defining
characters. 

\vspace{3mm}
Over $S$, the ``vanishing of the Lie bracket'', in the appropriate operadic sense, for an abelian Lie algebra,
is a non-trivial theorem. To prove it, we in fact make another construction of the Lie operad, modeling, 
over $S$, the {\em Koszul dual} of the infinite little cube operad. This brings out another issue: classically,
Koszul duality involves {\em shifts}. The Koszul dual of the little cube operad is its shift. (Unless we shift back,
we obtain a graded Lie bracket which itself has a non-zero degree.) In chain complexes, we may, of course,
shift at will, but can we do that with operads over $S$, given that the shift is replaced by the much less
well behaved suspension?
Luckily, for this, too, a nascent technique was developed by Arone and Kankaanrinta \cite{ak}. In this
paper, we develop it fully, showing, in fact, that at least on the level of derived categories, operads over $S$
are {\em stable} in the sense that shifts can be modeled by equivalences in categories.

\vspace{3mm}
Having constructed characters of abelian Lie algebras over $S$, we now have Verma and co-Verma modules. Unrestricted
groups of derived morphisms of Verma and co-Verma modules over $S$, however, are quite ``wild'', involving
Spanier-Whitehead duals of very infinite spectra, and we do not have 
computations. The situation is somewhat better in a variant of our construction, which is {\em graded by weights}.
There, some computations are possible using Carlsson's solution \cite{carlsson} of the Segal conjecture.
Still, those are not the groups we need for knot theory.

\vspace{3mm}
Next, Harish-Chandra pair enter. This concept of pairs of a Lie algebra, and compatible group in some sense (Lie
or algebraic), occurs in several areas of representation theory. Although 
we do not have Lie groups over $S$, but it turns out that we do have algebraic groups, at least in
a suitably weak sense: We can define a meaningful concept of a commutative Hopf algebra, and we can
invert a homotopy class. From this point of view, we have a model of $\mathscr{O}_{GL_nS}$.
To have Harish-Chandra pairs, however, we need to model a co-action of such ``$S$-algebraic groups''
on a Lie algebra. This is, at present, a problem, since we do not know of a way of rigidifying sufficiently the 
Hopf algebra {\em conjugation} (also called {\em antipode}), 
which, for $GL_n$, models the inverse matrix. We have a workaround
with a large prime hypothesis (which is good enough, since the $n$ here is actually not the $n$ mentioned earlier -
we only need Harish-Chandra pairs where the group is a product of $GL_\ell$'s for $\ell\leq k$).

\vspace{3mm}
Once the large prime hypothesis is adopted, the theory of Harish-Chandra pairs, at least
on blocks of limited weights, behaves well, as one may basically refer to characteristic $0$ for 
calculations. We managed to define analogues of Zuckermann functors, model a stable homotopy analogue of
the Khovanov cube, and also formalize Sussan's {\em diagram relations} to a point where they can be reduced
to existence of dualizing objects, which, again, follows from the characteristic $0$ case. Thus, we have 
constructed an $sl_k$ Khovanov homology stable homotopy type, at a prime (linearly) large with respect to $k$,
and prove its link invariance.

\vspace{3mm}
The present paper is organized as follows: In Section \ref{sbs}, we introduce the stable homotopy foundations
we work with. In Section \ref{slierep}, we introduce $S$-Lie algebras and their representations. In Section 
\ref{sgln}, we focus on the example of $gl_nS$, and define Verma modules and some other examples
of interest. In Section \ref{salgghc}, we define Harish-Chandra pairs and Zuckermann functors. 
In Section \ref{skhovanov}, we apply this to the construction of the $sl_k$ Khovanov homotopy type, and proof of
its link invariance.

\vspace{3mm}

\section{Operads in based spaces and spectra}
\label{sbs}

\subsection{Basic setup and model structure}

Let $\mathscr{B}$ be the category of based spaces with the usual smash product $\wedge$ with
the usual model structure where equivalences are weak equivalences and cofibrations are retracts of 
relative cell complexes. We need a category $\mathscr{S}$ of
{\em spectra} with the following properties:
\begin{enumerate}
\item There is a model structure for which every object is fibrant, and a suspension spectrum functor 
$\Sigma^{\infty}: \mathscr{B}\r \mathscr{S}$ which
is a left Quillen adjoint. (We will often tend to omit this functor from the notation.) We want this model 
structure to be defined via the ``small object" method that appears commonly in stable homotopy theory, by 
which we mean that cofibrantions are obtained as retracts of relative cell complexes. 
Further, there is a notion
of geometric homotopy of morphisms such that cofibrant objects are co-local with respect to the geometric 
homotopy category (given by the same objects and geometric homotopy classes of morphisms).
\item There is a commutative, associative and unital smash product $\wedge$ in $\mathscr{S}$ 
such that the suspension
spectrum functor preserves the smash product.
\end{enumerate}
One example of such a category $\mathscr{S}$ is constructed in \cite{ekmm}, which we summarize here. 
Let $\mathcal{S}$ denote the category of May spectra over a given universe.  
Denoting by $\mathscr{L}\!\mathscr{I}$ the linear isometries 
operad on a given universe, one first constructs the category $\mathcal{S}[\mathscr{L}\!\mathscr{I}]$ of 
$\mathscr{L}\!\mathscr{I}(1)$-modules, 
whose objects are spectra $X$ together with an action
$$\mathscr{L}\!\mathscr{I}(1)\rtimes X\r X,$$
and morphisms are morphisms of spectra which preserve this structure.
Denoting by $LI(1)$ the monad defining $\mathscr{L}\!\mathscr{I}(1)$-modules, we can consider the monad 
$(LI(1),LI(1))$ on 
pairs of spectra. Then there is a right $(LI(1),LI(1))$-module $LI(2)$ with 
$$LI(2)(X,Y)= \mathscr{L}\!\mathscr{I}(2)\rtimes (X\wedge Y)$$
(where on the right hand side, $\wedge$ denotes the external smash product). We then define the smash product
in $\mathcal{S}[\mathscr{L}\!\mathscr{I}]$ as
$$X\wedge_{\mathscr{L}\!\mathscr{I}} Y= LI(2)\otimes_{(LI(1),LI(1))} (X,Y)$$
where $\otimes_{(LI(1),LI(1))}$ denotes the coend. This operation is commutative and associative but not
strictly unital. It is important, however, that $S$, the suspension spectrum of $S^0$, is a commutative semigroup with respect
to the operation $\wedge_{\mathscr{L}\!\mathscr{I}}$ in $\mathcal{S}[\mathscr{L\!\mathscr{I}}]$. 

The model structure on $\mathcal{S}[\mathscr{L}\!\mathscr{I}]$ has the property that fibrations resp. equivalences are those morphisms which forget
to fibrations resp. equivalences in spectra. The forgetful functor from $\mathcal{S}[\mathscr{L}\!\mathscr{I}]$ to 
$\mathcal{S}$ is then right Quillen adjoint to $LI(1)$ (thought of as a functor from spectra to 
$\mathcal{S}[\mathscr{L}\!\mathscr{I}]$). The desired suspension
spectrum functor is the composition of the May suspension spectrum functor with $LI(1)$. 

\vspace{3mm}
As remarked, the smash product $\wedge_{\mathscr{L}\!\mathscr{I}}$ is not strictly unital, although 
$$S\wedge_{\mathscr{L}\!\mathscr{I}}S\cong S.$$
This can be remedied by the following further trick: Denote by $\mathscr{S}$ the full subcategory
of $\mathcal{S}[\mathscr{L}\!\mathscr{I}]$ on objects $X$ for which the canonical morphism
$$S\wedge_{\mathscr{L}\!\mathscr{I}} X\r X$$
is an isomorphism. Then we have a functor
\beg{esmodl}{S\wedge_{\mathscr{L}\!\mathscr{I}} (?):\mathcal{S}[\mathscr{L}\!\mathscr{I}]\r\mathscr{S},
}
left adjoint to the functor
\beg{esmodr}{F_{\mathscr{L}\!\mathscr{I}} (S,?):\mathscr{S}\r \mathcal{S}[\mathscr{L}\!\mathscr{I}].
}
(Both functors are restrictions of self-functors on $\mathcal{S}[\mathscr{L}\!\mathscr{I}]$. It is proved in \cite{ekmm},
Chapter VII, that
there is a model structure on $\mathscr{S}$ such that the pair of adjoint functors \rref{esmodl}, \rref{esmodr}
is a Quillen equivalence.

\vspace{3mm}
We will work in the category $\mathscr{S}$. From now on, we will abuse terminology 
by referring to $\mathscr{S}$ as the category of {\em spectra}, and denoting the smash product 
$\wedge_{\mathscr{L}\!\mathscr{I}}$ restricted to $\mathscr{S}$ simply by $\wedge$.

\vspace{3mm}

We briefly recall the definition of operads and their algebras, as well as modules over 
an operad algebra, in the categories of based spaces and of spectra. For the full precise statements of 
the axioms that must be satisfied by the structure morphisms, we ask the reader to look to the references given
below. 
\begin{definition}
\label{d1}
An {\em operad} in $\mathscr{B}$ is a sequence of objects $\mathscr{C}(n)$, $n=0,1,2,\dots$,
and morphisms
$$\gamma:\mc(k)\wedge \mc(n_1)\wedge\dots\wedge \mc(n_k)\r\mc(n_1+\dots+n_k),$$
$$\iota:S^0\r \mathscr{C}(1)$$
and actions of the symmetric group $\Sigma_k$ on $\mc(k)$. These structure morphisms and the 
$\Sigma_k$ action satisfy the usual associativity, unit and equivariance axioms
(\cite{mayg}, Definition 1.1). An operad in $\mathscr{S}$ is defined the same way except that $S^0$ is replaced by $S$.
A {\em $1$-operad} is defined the same way, except that we have $n=1,2,\dots$.

An {\em algebra} over an operad $\mathscr{C}$ in $\mathscr{B}$ or $\mathscr{S}$ is an object $X$ together
with structure maps
$$\mathscr{C}(n)\wedge\underbrace{X\wedge\dots\wedge X}_{n}\r X$$
which satisfy the usual associativity, unit and equivariance axioms (\cite{mayg}, Definition 1.2). For an algebra $R$
over an operad (or $1$-operad) $\mathscr{C}$), a {\em $(\mathscr{C},R)$-module $M$} has structure maps
$$\mathscr{C}(n)\wedge\underbrace{X\wedge\dots\wedge X}_{n-1}\wedge M\r M$$
which satisfy the usual associativity, equivariance and unitality axioms (see e.g. \cite{ekmm}). In the case
of an operad, there is no structure map for $n=0$.

Morphisms of operads, $1$-operads, algebras and modules are (sequences of) morphisms in the underlying category
which preserve the operations.
\end{definition}

\vspace{3mm}
The category of $1$-operads is (canonically equivalent to)
a coreflexive subcategory of the category of operads. The inclusion functor takes
a $1$-operad $\mathscr{C}$ to the operad where we additionally define $\mathscr{C}(0)=*$
(the zero object). The coreflection functor (right adjoint to the inclusion) is given on an operad $\mathscr{C}$
by forgetting $\mathscr{C}(0)$. The category of algebras over a $1$-operad $\mathscr{C}$ is canonically equivalent
to the category of algebras over $\mathscr{C}$ considered as an operad, and similarly for modules.

The category of associative unital algebras (in $\mathscr{B}$ or $\mathscr{S}$) is additionally a
reflexive subcategory of the category of $1$-operads. The inclusion functor is defined, on an associative algebra
$\mathscr{A}$, by putting $\mathscr{A}(1)=\mathscr{A}$, $\mathscr{A}(n)=*$ for $n\geq 2$. The reflection
(left adjoint to inclusion) is defined by sending an operad $\mathscr{C}$ to the associative algebra $\mathscr{C}(1)$.

The category of operads (resp. $1$-operads) has a terminal object $*$ where $*(n)=*$ for all $n$, and
an initial object $\mathscr{E}$ where $\mathscr{E}(1)=S^0$ and $\mathscr{E}(n)=*$ for $n\neq 1$. The
category of $\mathscr{E}$-algebras is canonically equivalent to the underlying category ($\mathscr{B}$
or $\mathscr{S}$).

Additionally, $\Sigma^\infty$ extends to
a functor 
$$\Sigma^\infty:\mathscr{B}\text{-operads}\r\mathscr{S}\text{-operads}.$$
On $\mathscr{B}$ and $\mathscr{S}$, there are model structures where cofibrations are retracts of relative
cell complexes where cells are of the form
\beg{esbs1}{\begin{array}{ll}
S^{n-1}_{+}\r D^{n}_{+},\; n\geq 0 & \text{in $\mathscr{B}$}\\
S^{n-1}_{+}\r D^{n}_{+},\; n\in \Z & \text{in $\mathscr{S}$},
\end{array}
}
equivalences are maps inducing an isomorphism in $\pi_*$. Then $\Sigma^\infty$ is a left Quillen adjoint. Note also that
the smash product of two cofibrant objects in $\mathscr{B}$ or $\mathscr{S}$ is cofibrant. 

Similar comments also apply to $1$-operads, and the pairs of adjunct functors between $1$-operads and operads,
and $1$-operads and associative algebras, are Quillen adjunctions, as usual.

The reason why we consider both operads and $1$-operads is subtle. Concepts such as stability and Koszul duality
work better for $1$-operads. On the other hand, the tensor product of modules over an operad algebra 
works better for operads.

Now we have forgetful functors 
\beg{esbs2}{\begin{array}{l}
?(n):\mathscr{B}\text{-operads}\r\mathscr{B}\\
?(n):\mathscr{S}\text{-operads}\r\mathscr{S}
\end{array}
}
for $n\in \mathbb{N}_0$, with left adjoint $\mathscr{O}(n)$. Then we have model structures on 
$\mathscr{B}$-operads and $\mathscr{S}$-operads where cofibrations are relative cell complexes where cells
are $\mathscr{O}(n)$ of cells of the form \rref{esbs1}, and equivalences are morphisms of operads which 
become equivalences upon applying \rref{esbs2} for every $n$. Then $\Sigma^\infty$ extends to a left 
Quillen adjoint
$$\mathscr{B}\text{-operads}\r\mathscr{S}\text{-operads}.$$
For the purposes of homotopy theory, we will only consider algebras over cofibrant operads. For a cofibrant operad $\mathscr{C}$
(in $\mathscr{B}$ or $\mathscr{S}$),
there is a model structure on $\mathscr{C}$-algebras where equivalences and fibrations are the morphisms which
have the same property in the underlying category. The functor $\Sigma^\infty$ then takes $\mathscr{C}$-algebras
to $\Sigma^\infty\mathscr{C}$-algebras, and is a left Quillen adjoint.

Suppose $f:\mathscr{C}\r \mathscr{D}$ is a morphism of operads in $\mathscr{B}$ or $\mathscr{S}$,
where $\mathscr{C}$ and $\mathscr{D}$ are cofibrant. Then every $\mathscr{D}$-algebra $X$ is automatically 
a $\mathscr{C}$-algebra. We will denote this functor from $\mathscr{D}$-algebras to $\mathscr{C}$-algebras by
$f^*$. It is a right Quillen adjoint to a functor which we will denote by $f_\sharp$.

\vspace{3mm}
\subsection{Stability}

We will need to use suspensions and desuspensions of operads in the category of spectra. Recall that 
in the category of chain complexes, there is a suspension functor $?[1]$ on $1$-operads given by
$$\mathscr{C}[1](n)=\mathscr{C}(n)[n-1]$$
where the brackets on the right hand side mean ordinary suspension in the category of chain complexes:
$$C[k]_n = C_{n-k}.$$
For an $1$-operad $\mathscr{C}$, there is then an equivalence of categories
$$\mathscr{C}\text{-algebras}\r\mathscr{C}[1]\text{-algebras},$$
$$X\mapsto X[-1].$$
We shall prove an analogous result for spectra, at least on the level of derived (i.e. Quillen homotopy) categories
of algebras over cofibrant $1$-operads.

In order to define suspension of operads in $\mathscr{S}$,
we use the Arone-Kankaanrantha (AK) $1$-operad $\mathscr{J}$ in $\mathscr{B}$ (see \cite{ak}) as a 
model of the sphere. This operad has
$$\mathscr{J}(n)=\Delta^{n-1}/\partial\Delta^{n-1}$$
where $\Delta^{n-1} $ is the standard $(n-1)$-simplex. The barycentric coordinates are written as
$[t_1,\dots,t_n]$, the symmetric group $\Sigma_n$ acts by permutation of coordinates, and composition is given by
$$\begin{array}{l}
[s_1,\dots,s_k]\times [t_{1,1},\dots,t_{1,n_1}]\times\dots
\times [t_{k,1},\dots,t_{k,n_k}]\\
\mapsto
[s_1t_{1,1},\dots,s_{1}t_{1,n_1},\dots,s_k t_{k,1},\dots, s_k t_{k,n_k}].
\end{array}
$$
It is also a co-$1$-operad, with structure map
$$\mathscr{J}(n_1+\dots+n_k)\r\mathscr{J}(k)\wedge\mathscr{J}(n_1)\wedge\dots\wedge \mathscr{J}(n_k)$$
given by
$$
\begin{array}{l}
[s_{1,1},\dots,s_{1,n_1},\dots,s_{k,1},\dots,s_{k,n_k}]\mapsto\\ \protect
[s_{1,1}+\dots+s_{1,n_1},\dots,s_{k,1}+\dots+s_{k,n_k}]\times\\[2ex]\protect
\displaystyle
\left[
\frac{s_{1,1}}{s_{1,1}+\dots+s_{1,n_1}},\dots,\frac{s_{1,n_1}}{s_{1,1}+\dots+s_{1,n_1}}\right]\times\\[2ex]
\displaystyle
\displaystyle\dots\times
\left[
\frac{s_{k,1}}{s_{k,1}+\dots+s_{k,n_k}},\dots,\frac{s_{k,n_k}}{s_{k,1}+\dots+s_{k,n_k}}\right].
\end{array}
$$
This means that
$$\mathscr{J}=\Sigma^\infty \mathscr{J},$$
$$\mathscr{T}=F(\Sigma^\infty\mathscr{J},S^0)$$
are operads in $\mathscr{S}$. Let $\widetilde{\mathscr{J}}$, $\widetilde{\mathscr{T}}$ be their
cofibrant replacements.

\begin{definition}\label{d2}
An {\em $E_\infty$-operad} (resp. {\em $E_\infty$-$1$-operad})
is a cofibrant operad (resp. $1$-operad) in $\mathscr{S}$ equivalent to $\Sigma^\infty S^0$. (Here $S^0$ stands
for the unique operad (resp.  $E_\infty$-$1$-operad)
in $\mathscr{B}$ whose $n$'th term is $S^0$.) An algebra over an $E_\infty$-operad will be 
also called an {\em $E_\infty$}-algebra.
\end{definition}

\vspace{3mm}

The next lemma shows that $\widetilde{\mathscr{T}}$ is an ``inverse" to $\widetilde{\mathscr{J}}$ as operads 
in $\mathscr{S}$. 

\begin{lemma}
\label{l1}
$\widetilde{\mathscr{J}}\wedge\widetilde{\mathscr{T}} $ is an $E_\infty$-$1$-operad. (The smash product is performed
one term at a time.)
\end{lemma}

\begin{proof}
We have the commutative diagram
\beg{esbs+}{\diagram
\mathscr{J}(k)\wedge\mathscr{J}(n_1)\wedge\dots\wedge \mathscr{J}(n_k)\rto\dto &
\mathscr{J}(k)\wedge\mathscr{J}(n_1)\wedge\dots\wedge\mathscr{J}(n_k)\\
\mathscr{J}(n)\rto & \mathscr{J}(n).\uto
\enddiagram
}
(Here, $n = n_1 + \cdots n_k$.)
Tthe horizontal maps
are identities. The left hand side map is the operad structure map: 
$$[s_1,\dots,s_k]\times
[t_{1,1},\dots,t_{1,n_1}]\times\dots\times\protect
[t_{k,1},\dots,t_{k,n_k}]$$
is mapped to
$$\protect[s_1t_{1,1},\dots,s_1 t_{1,n_1},\dots,s_k t_{k,1},\dots, s_k t_{k,n_k}].$$
The right vertical arrow is the co-operad structure map. It is easy to check that this diagram commutes. 

The diagram \rref{esbs+} gives an evaluation map 
\[ \widetilde{\mathscr{J}} \wedge \widetilde{F(\mathscr{\mathscr{J}}, S^0)} \rightarrow \Sigma^{\infty} S^0 \]
which is an equivalence. 
\end{proof}

We remark that actually, for any co-$1$-operad $\mathscr{R}$ in $\mathscr{B}$, 
a commutative diagram of the form
\beg{esbsgen+}{\diagram
\mathscr{J}(k)\wedge\mathscr{J}(n_1)\wedge\dots\wedge \mathscr{J}(n_k)\rto\dto &
\mathscr{R}(k)\wedge\mathscr{R}(n_1)\wedge\dots\wedge\mathscr{R}(n_k)\\
\mathscr{J}(n)\rto & \mathscr{R}(n).\uto
\enddiagram
}
always gives rise to an evaluation map 
\[ \widetilde{\mathscr{J}} \wedge \widetilde{F(\mathscr{R}, S^0)} \rightarrow \Sigma^{\infty} S^0 . \]
\vspace{3mm}

\begin{lemma}
\label{l2}
$\Sigma^\infty S^1$ and the cofibrant replacement $\widetilde{F(\Sigma^\infty S^1, S^0)}$ of

\noindent
$F(\Sigma^\infty S^1, S^0)$ are algebras over $\widetilde{\mathscr{J}}$, $\widetilde{\mathscr{T}}$,
respectively, and
\beg{esbs*}{\Sigma^\infty S^1\wedge \widetilde{F(\Sigma^\infty S^1, S^0)}}
is equivalent to the $E_\infty$-algebra $S^0$.
\end{lemma}

\begin{proof}
Consider the $1$-operad $\mathscr{J}$ and any co-$1$-operad $\mathscr{R}$ satisfying \rref{esbsgen+}.
Purely on the level of spaces, for a space $X$, for $\Sigma^\infty X$ and $F(\widetilde{\Sigma^\infty X},S)$ to have
the structure described, along with a map of algebras
\beg{esbs**}{ev:\Sigma^\infty X\wedge F(\widetilde{\Sigma^\infty X}, S)\r S,
}
it suffices to have an $\mathscr{R}$-algebra structure and an $\mathscr{J}$-coalgebra structure on $X$ together
with a commutative diagram
\beg{esbs***}{\diagram
X\wedge\dots\wedge X\ddrto_{Id}\rto & X\wedge \mathscr{J}(n)\dto\\
&X\wedge \mathscr{R}(n)\dto\\
&X\wedge\dots \wedge X.
\enddiagram
}
For $\mathscr{R}= \mathscr{J}$, we want to construct maps that would satisfy \rref{esbs***}. A first 
attempt would be to use the formulas 
\beg{esbss1}{\begin{array}{c}
\displaystyle
S^1\wedge\dots\wedge S^1 \r S^1\wedge \mathscr{R}(m)\\[2ex]
\protect
\displaystyle
(t_1,\dots,t_m)\mapsto
(t_1,\dots, t_m)\times \left[
\frac{t_1}{t_1+\dots+t_m},\dots,\frac{t_m}{t_1+\dots+t_m}\right]
\end{array}
}
and
\beg{esbss2}{\begin{array}{c}
\displaystyle\protect
S^1\wedge\mathscr{J}(n)\r S^1\wedge\dots\wedge S^1\\[2ex]
\protect
\displaystyle
t\times [s_1,\dots,s_m]\mapsto (s_1t,\dots,s_mt) .
\end{array}
}
Unfortunately, 
$t=1$ in \rref{esbss2} does not necessarily imply that the element goes to the base point. Note that
\rref{esbss2} is a partial structure in the sense that it is correct if we let the model of $S^1$ on the left be
$[0,N]/0\sim N$ where $N\geq m$. 

A better approach, however, is replacing $S^1$ with 
$$\underline{S}^1=[0,1]\times (0,1]/(0,\epsilon)\sim *, (s,\epsilon)\sim *$$ where $s\geq \epsilon$. 
One can define a suitable  $E_\infty$-$1$-operad $\mathscr{C}$ by ``decreasing the coordinate $\epsilon$'', and replace 
$\mathscr{J}$ by $\mathscr{J} \wedge \mathscr{C}$.
Then the
coalgebra structure on $S^1$ can be replaced by a coalgebra structure on $\underline{S}^1$:
$$\underline{S}^1\wedge \mathscr{J}(n)\wedge \mathscr{C}(n)\r \underline{S}^1\wedge\dots\wedge \underline{S}^1 .$$
The action \rref{esbss1}
can then be made to ``not increase". The diagram \rref{esbs***} will be commutative up to all higher homotopies.
Since we use cofibrant replacement everywhere, this is sufficient. 
\end{proof}

Hence, we are able to use smashing with $\widetilde{\mathscr{J}}$ 
as a model for the suspension of operads in $\mathscr{S}$, and use smashing with $\widetilde{\mathscr{T}}$ 
as a model for the desuspension of operads. 

\vspace{3mm}

\subsection{Equivalence of Arone-Kankaanrantha $1$-operads}
\label{seq}
In this subsection, we construct another model $Q$ for the Arone-Kankaanrantha operad $\mathscr{J}$, which will 
use cubes instead of simplices. Therefore, it 
will be better suited toward working with the little cubes operads. 

Consider an $1$-operad $Q$ in $\mathscr{B}$ given by
$$Q(n)=\overline{A}_n/\overline{B}_n,$$
where
$$\overline{A}_n=\{(J_1,\dots,J_n)\mid J_i=[s_i,t_i],\;0\leq s_i<t_i\leq 1\},$$
$$\overline{B}_n=\{(J_1,\dots,J_n)\in\overline{A}_n\mid\text{Interior}(J_1\cap\dots\cap J_n)
=\emptyset\}.$$
Here $\overline{A}_n$, $\overline{B}_n$ are given the induced topology from $\R^{2n}$. Note that
$\overline{B}_n$ is a closed subset of $\overline{A}_n$.
Then $Q$ is a $1$-operad where composition is given by increasing linear homeomorphisms
$$\diagram I\rto^\cong & J_i.\enddiagram$$
It is easily checked that $Q(n)$ has the same homotopy type as $\mathscr{J}(n)$ where $\mathscr{J}$ is the
$1$-operad defined in the last section. However, we will need a stronger statement asserting 
the existence of equivalences compatible with structure maps, whose proof is more difficult:

\begin{proposition}
\label{pseq1}
There is an equivalence of $\mathscr{B}$-$1$-operads
$$Q\sim\mathscr{J}.$$
\end{proposition}

The proof of the Proposition will occupy the remainder of this subsection. It is purely topological 
in nature, and readers not interested in the 
proof may skip it without affecting their understanding of the rest of this paper. 

Put
$$A_n=\{(J_1,\dots,J_n)\in \overline{A}_n\mid J_1\cap\dots\cap J_n\neq\emptyset\},$$
$$B_n=A_n\cap \overline{B}_n.$$
Then the canonical continuous map
$$A_n/B_n\r \overline{A}_n/\overline{B}_n$$
is a bijection and is in fact a homeomorphism, as $A_n\subset \overline{A}_n$ is a closed subset.

To prove the Proposition, we shall introduce an ``intermediate" $1$-operad $Q^\prime$. We let
$$Q^{\prime}_{n}= A^{\prime}_{n}/B^{\prime}_{n}$$
where
$$
A^{\prime}_{n}=
A_n\cup C_n,
$$
$$C_n=\{([0,t_1],\dots,[0,t_n]\mid
0\leq t_i\leq 1,\; (t_1,\dots,t_n)\neq (0,\dots,0),\;(\exists i) t_i=0\},$$
$$B^{\prime}_{n}=B_n\cup C_n.$$
Again, $A^{\prime}_{n},B^{\prime}_{n}\subset \R^{2n}$ are given the induced topology.

It is easy to see that $C_n\subseteq B^{\prime}_{n}\subseteq A^{\prime}_{n}$ is a closed
subset (of course, $C_n\cap A_n=\emptyset$,) and $Q^\prime$ is a $1$-operad with structure defined
by the same formula as for $Q$. Furthermore, we have a canonical continuous map of $1$-operads
\beg{eseq1}{\iota:Q\r Q^\prime}
which is a bijection but not necessarily a homeomorphism.

\vspace{3mm}

\begin{lemma}
\label{lseq1}
$\iota$ is a weak equivalence.
\end{lemma}

\begin{proof}
We first remark that one easily checks that $A_n$, $A^{\prime}_{n}$ are contractible, and the inclusions
$B_n\subset A_n$, $B^{\prime}_{n}\subset A^{\prime}_{n}$ are cofibrations. Therefore, it suffices to prove that
\beg{eseq2}{B_n\subset B^{\prime}_{n} \;\text{is an equivalence.}
}
To this end, we define a map
$$p:B^{\prime}_{n}\r [0,1)$$
by putting
$$(J_1,\dots,J_n)\in B_n\mapsto J_1\cap\dots\cap J_n,$$
$$([0,t_1],\dots,[0,t_n])\in C_n\mapsto 0.$$
Clearly, $p$ is continuous. Note that \rref{eseq2} would follow from
\beg{eseq3}{\text{$p$ is a quasifibration.}
}
Thus, it suffices to prove \rref{eseq3}. We will use the criterion of May \cite{mayg}, Theorem 2.6. We remark that 
$\{0\}$, $(0,1)$ are clearly distinguished sets, as $p:B_n\r (0,1)$ is clearly a fiber bundle (hence, in fact,
a product).

We are therefore done if we can exhibit a continuous homotopy
$$\widetilde{\lambda}:B^{\prime}_{n}\r B_n$$
which covers multiplication by $\lambda\in [0,1]$ on $[0,1)$ such that $\widetilde{0}$ is an equivalence on fibers. 
We let
$$\widetilde{\lambda}|_{C_n}=Id,$$
and for 
$$J_1\cap\dots\cap J_n=\{n\}, \; J_i=[s_i,t_i],$$
we have
$$\widetilde{\lambda}(J_1,\dots,J_n)=
([\phi_{\lambda,u}s_1,\phi_{\lambda,u}t_1],
\dots,[\phi_{\lambda,u}s_n,\phi_{\lambda,u}t_n])
$$
where 
$$\phi_{\lambda,u}:I\r I$$
is a continuous map preserving $0$, $1$ such that
$$\phi_{\lambda,u}(u)=\lambda u$$
and $\phi_{\lambda,u}$ is linear on $[0,u]$ and $[u,1]$.
Clearly, 
$$\widetilde{?}:B^{\prime}_{n}\times [0,1]\r B^{\prime}_{n}$$
is continuous, and $\widetilde{\lambda}$ covers
$\lambda:[0,1]\r [0,1]$.
Thus, it remains to show that
\beg{eseq4}{\widetilde{0}:p^{-1}(u)\r C_n
}
is an equivalence. To this end, define a linear increasing homeomorphism
$$q_u:[u,1]\r [0,1]$$
and a linear {\em decreasing} homeomorphism
$$r_u:[0,u]\r[0,1].$$
Define
$$\Phi_u:p^{-1}(u)\r I^{2n}$$
by
$$\begin{array}{l}
\Phi_u([s_1,t_1],\dots,[s_n,t_n])=\\
(r_u(s_1),q_u(t_1),\dots,r_u(s_n),q_u(t_n)).
\end{array}$$
We then see that $\Phi_u$ is a homeomorphism onto the set $S$ of all points
$$(x_1,y_1,\dots,x_n,y_n)\in I^{2n}$$
where 
$$(x_i,y_i)\neq (0,0)$$
and
$$(\exists i,j) \; x_i=0, y_j=0.$$
Clearly, $S_0\subset S$ is a homotopy equivalence where
$$S_0=\{(x_1,y_1,\dots, x_n,y_n)\in S\mid x_i+y_i=1\}.$$
Then 
$$(x_1,y_1,\dots,x_n,y_n)\mapsto (x_1,\dots,x_n)$$
is a homeomorphism 
$$\diagram \Psi:S_0\rto^(.55){\cong}& R\enddiagram$$
where
$$R=\{(x_1,\dots,x_n)\in I^n\mid(\exists i,j) \; x_i=0,\; x_j=1\}.$$
Note that in the canonical decomposition of $I^n$ into cubes, $R$ is the union of open cubes in
$\partial I^n$ whose closures do not intersect $(0,\dots,0)$ and $(1,\dots, 1)$. Hence,
we have $R\cong S^{n-2}$. 

On the other hand, we have a homotopy equivalence
$$\diagram \Theta:C\rto^\simeq & R,\enddiagram$$
$$([0,t_1],\dots,[0,t_n])\mapsto (\frac{t_1}{\text{max}(t_i)},\dots, \frac{t_n}{\text{max}(t_i)}),$$
and we see that
$$\Theta\circ \widetilde{0}\circ\Phi^{-1}_{u}\circ\Psi^{-1}\simeq Id_R$$
(via a linear homotopy).
\end{proof}

\vspace{3mm}
\begin{corollary}
\label{cseq1}
(of the proof of Lemma \ref{lseq1}) The inclusion $C_n\r B^{\prime}_{n}$ is an equivalence.
\end{corollary}
\qed

\vspace{3mm}
\noindent
{\bf Proof of Proposition \ref{pseq1}:}
We just proved that the inclusion
\beg{eseqi}{\diagram Q\rto^\subset & Q^\prime\enddiagram}
is an equivalence of $1$-operads. Now note that we have another $1$-operad $\mathscr{J}^\prime$
with 
$$\mathscr{J}^\prime (n)= C_n$$
and the same composition formula. Thus, by Corollary \ref{cseq1}, the inclusion
\beg{eseqii}{\diagram\mathscr{J}^\prime \rto^\subset &Q^\prime\enddiagram
}
is an equivalence of $1$-operads. Finally, we have an obvious inclusion of $1$-operads
\beg{eseqiii}{\diagram\mathscr{J}\rto^\subset &\mathscr{J}^\prime \enddiagram
}
which on $\mathscr{J}(n)$ is
$$[t_1,\dots,t_n]\mapsto ([0,t_1],\dots, [0,t_n]).$$
Clearly, this is an equivalence, so by the equivalences \rref{eseqi}, \rref{eseqii}, \rref{eseqiii}, we are done.
\qed

\vspace{3mm}
\section{$S$-Lie algebras and their representations}
\label{slierep}

In~\cite{ching}, Ching constructed the Lie operad over $S$, using the Goodwillie derivatives of the identity
functor. In this section, 
we begin by constructing another version $\mathscr{L}$ 
of the Lie operad, which is equivalent in the derived category with Ching's operad 
$Ch$. The principle behind our construction is that of Koszul duality. In classical algebra, it is a well-known fact that 
the Lie operad and the commutative operads are Koszul dual to each other. Hence, we define the derived Lie operad 
$\mathscr{L}$ to be the Koszul dual of the (desuspension of the) commutative operad. 

More specifically, recall that
in algebra, $E_n$-operads are Koszul dual to 
themselves, up to desuspension by $n-1$ (see e. g. \cite{fresse}), and an analogous result holds over $S$ 
(see Proposition~\ref{psslie3} below). 

\subsection{The derived Lie operad}
\label{sslie}
Now consider the $k$-dimensional little cube operad $\mathscr{C}_{k+}$ in $\mathscr{B}$ (where the 
subscript $?_+$ indicates a disjoint base point), with 
$n$-th space
$$\mathscr{C}_k(n)_+.$$
We know by \cite{mayg} that
$$\mathscr{C}_{\infty+}=\lim_\r \mathscr{C}_{k+}$$
is an $E_\infty$-operad. Now note that there is a natural map of $1$-operads
\beg{essli+}{\mathscr{C}_{(k+1)+}\r Q\wedge(\mathscr{C}_{k+})
}
given by smashing the projection of little cubes to the first coordinate with the projection to the last $k$
coordinates (term-wise from the point of view of the $1$-operad). By Lemma \ref{l1}, we then have maps of
$1$-operads
\beg{esslie1}{\widetilde{\mathscr{T}}\wedge \mathscr{C}_{(k+1)+}\r \mathscr{C}_{k+}.
}
To simplify notation, we shall from now on omit the tilde from $\widetilde{\mathscr{T}}$, and write
simply $\mathscr{T}$. We shall, however, always mean the cofibrant replacement. Let
\beg{esslie2}{\mathscr{L}=\operatornamewithlimits{holim}_\leftarrow 
\mathscr{T}^k\wedge \mathscr{C}_{(k+1)+},
}
using the maps \rref{esslie1}. We shall call $\mathscr{L}$ the {\em derived Lie $1$-operad}.
We shall also consider $\mathscr{L}$ as an operad by applying the inclusion functor mentioned after
Definition \ref{d1}.
\vspace{3mm}

The following proposition shows that $\mathscr{L}$ is indeed a derived version of the Lie operad over $S$. 

\begin{proposition}
\label{psslie3}
There is an equivalence between the $1$-operad $\mathscr{L}$ and $\mathscr{J}\wedge Ch$, where
$Ch$ is the Ching operad \cite{ching}.
\end{proposition}

\begin{proof}
One knows by imitating the proof of \cite{hu} that the Koszul dual,
in the sense of Ching \cite{ching}, of $\mathscr{C}_{k+}$ is $\mathscr{T}^{k-1}\wedge \mathscr{C}_k$.
While Koszul duality, in general, does not commute with homotopy inverse limits, it follows by direct
calculation that it does so in this case of $\mathscr{L}$. Thus, the Koszul dual of $\mathscr{L}$ is
equivalent to $\mathscr{T}$, as is the Koszul dual of $\mathscr{J}\wedge Ch$. Applying Koszul duality
again gives the statement.
\end{proof}

\vspace{3mm}
\begin{proposition}
\label{psslie1}
We have a diagram of operads in $\mathscr{S}$:
\beg{epsslie*}{\diagram
\mathscr{L}\rto^u\dto_p & \mathscr{C}_{1+}\dto^q\\
\mathscr{E}\rto_i & \mathscr{C}_{\infty +}
\enddiagram}
which is commutative in the derived (i. e. Quillen homotopy) category.
\end{proposition}

\begin{proof}
It is easy to see explicitly that $\mathscr{L}(1)\simeq S^0$ (since the stabilization is an isomorphism on that
term), so $p$ is obtained as the counit of the adjunction between associative algebras and $1$-operads 
(see the comments after Definition \ref{d1}).

The top map follows directly from the definition \rref{esslie2}. 
Now let us investigate the composition map
\beg{esslie3}{\mathscr{T}^k\wedge\mathscr{C}_{(k+1)+}\r\mathscr{T}^{k-1}\wedge\mathscr{C}_{k+}\r
\dots\r \mathscr{C}_{1+}\r\dots\r\mathscr{C}_{(k+1)+}.
}
Smashing with $\mathscr{J}^k$ (again, we omit the tilde, but mean a cofibrant model), we have, in the
homotopy category of operads, a map
$$\mathscr{C}_{(k+1)+}\r Q^k\wedge\mathscr{C}_{(k+1)+}$$
which is, more or less by the definition of \rref{essli+}, $\epsilon\wedge Id$ where
$$\epsilon:S^0\r Q^k$$
is given as follows:
Choose a model $\widetilde{S}^0$ where
$$\widetilde{S}^0(n)=\Delta^{n}_+$$
with operad structure
$$
\begin{array}{l}
[t_1,\dots,t_k]\times [s_{1,1},\dots, s_{1,n_1}]\times\dots
\times [s_{k,1},\dots,s_{k,n_k}]\\
\mapsto [t_1s_{1,1},\dots,t_1s_{1,n_1},\dots, t_ks_{k,1},\dots, t_ks_{k,n_k}].
\end{array}
$$
A map of operads $\widetilde{S}^0\r \mathscr{J}$ is given by the projections
$$\Delta^{n}_+\r \Delta^n/\partial \Delta^n.$$
As before, this can be modeled by a map
$$\epsilon:\mathscr{D}_+\r Q$$
where
$$\mathscr{D}(n)_+=
\{(J_1,\dots,J_n)\mid
J_1,\dots,J_n\;\text{are closed subintervals of $I$}\}
$$
and $\epsilon$ is the projection. We conclude that \rref{esslie3} factors, up to homotopy, through
a map of operads
$$\diagram
\mathscr{T}^k\wedge \mathscr{C}_{(k+1)+}\dto_{\epsilon^k}\rto & \mathscr{C}_{(k+1)+}\\
\mathscr{J}^k\wedge \mathscr{T}^k\wedge \mathscr{C}_{(k+1)+}.\urto_\simeq &
\enddiagram
$$
Additionally, we claim there is a homotopy commutative diagram 
$$\diagram
\mathscr{J}^{k-1}\wedge \mathscr{T}^k\wedge \mathscr{C}_{(k+1)+}
\dto_{\epsilon\wedge Id}\rto 
& \mathscr{J}^{k-1}\wedge \mathscr{T}^{k-1}\wedge \mathscr{C}_{k+}
\rto^\simeq &\mathscr{C}_{k+} \dto^{\subset}\\
\mathscr{J}^k\wedge\mathscr{T}^k\wedge \mathscr{C}_{(k+1)+}\rrto_\simeq && \mathscr{C}_{(k+1)+}.
\enddiagram
$$
Thus, we have a homotopy commutative diagram of $\mathscr{S}$-operads
$$
\diagram
\mathscr{L}\dto_{\epsilon^k}&&&\\
\mathscr{J}^k\wedge\mathscr{L}\dto_\epsilon\rto &
\mathscr{J}^k\wedge\mathscr{T}^k\wedge\mathscr{C}_{(k+1)+}\rto^\simeq &\mathscr{C}_{(k+1)+}\dto &\\
\mathscr{J}^{k+1}\wedge \mathscr{L}\rto &
\mathscr{J}^{k+1}\wedge\mathscr{T}^{k+1}\wedge\mathscr{C}_{(k+2)+}\rto^\simeq &
\mathscr{C}_{(k+2)+}\rto &\mathscr{C}_{\infty +}
\enddiagram
$$
and hence taking the colimit, 
$$\mathscr{L}\r \mathscr{C}_{\infty +}$$
factors as
$$\diagram
\mathscr{L}\drto \dto_{\displaystyle\lim_\r \epsilon^k}&\\
(\displaystyle\lim_\r\mathscr{J}^k)\wedge\mathscr{L}\rto & \mathscr{C}_{\infty +}.
\enddiagram
$$
However, the lower left corner operad is equivalent to $\mathscr{E}$.
\end{proof}

\vspace{3mm}
\begin{proposition}
\label{psslie2}
$$\mathscr{L}(n)\simeq \bigvee_{(n-1)!} S^0.$$
\end{proposition}

\begin{proof}
We have a sequence of fibrations
$$\diagram
&\vdots\\
\displaystyle\bigvee_2 S^{k-1}\rto & \mathscr{C}_k(3)\dto\\
\displaystyle\bigvee_1 S^{k-1}\rto & \mathscr{C}_k(2)\dto\\
&\mathscr{C}_k(1)\simeq *.
\enddiagram
$$
The general fibration is
$$
\bigvee_{n-1} S^{k-1}\r \mathscr{C}_k(n)\r\mathscr{C}_k(n-1).
$$
Considering the corresponding Gysin cofibration
$$
\mathscr{C}_k(n)_+\r \mathscr{C}_k(n-1)_+\r \mathscr{T}_k(n),
$$
we see that $\mathscr{T}_k(n)$ has a based CW-decomposition with cells
corresponding to $(n-1)$ copies of the cells of the CW-decomposition of 
$\mathscr{C}_k(n-1)_+$, suspended by $k$.
Thus, stably, 
$\Sigma^{1-k}\mathscr{C}_k(n)_+$ has a finite CW-decomposition with 
$(n-1)!$ cells in dimension $0$, and other cells in dimension $\leq 1-k$.

Now investigate the homotopy limit of the sequence
\beg{ep2sslie1}{\dots\r
\Sigma^{(1-(k+1))(n-1)}\mathscr{C}_{k+1}(n)_+\r
\Sigma^{(1-k)(n-1)}\mathscr{C}_k(n)_+\r \dots\r \mathscr{C}_1(n)_+.
}
We see by obstruction theory that for each $k$, there exists a $K_k>>0$ such that we have a 
stable factorization
$$\diagram
\displaystyle\bigvee_{(n-1)!}S^0
\drdotted|>\tip\rto^{Id} &\displaystyle\bigvee_{(n-1)!}S^0\\
\Sigma^{1-K_k}\mathscr{C}_{K_k}(n)_+\udotted|>\tip\rto
&
\mathscr{C}_k(n)_+.\uto
\enddiagram
$$
This implies both that the top cells of $\mathscr{C}_k(n)_+$ stably split (which is well known)
and also that the homotopy limit \rref{ep2sslie1} factors through
$$\diagram
\dots \rto^(.4){Id} & \displaystyle\bigvee_{(n-1)!} S^0\rto^{Id} & \displaystyle\bigvee_{(n-1)!} S^0
\enddiagram
$$
which implies the statement.
\end{proof}

\vspace{3mm}
\subsection{Derived Lie algebra representations}\label{ssdlr}
By a {\em Lie algebra over $S$}, we shall mean an algebra $g$ over the $\mathscr{S}$-operad $\mathscr{L}$.
By a {\em $g$-representation}, we shall mean a module over the operad $\mathscr{L}$ and a Lie algebra $g$.
Over an ordinary ring from classical algebra, it is almost trivial to see that a representation of a Lie algebra $g$ is the same
thing as a left (or, alternately, right) module over its universal enveloping algebra $Ug$. 

Over $S$, this is still true, but it requires more discussion. The universal enveloping algebra functor is the
pushforward
$$U=u_\sharp$$
where 
$$u:\mathscr{L}\r\mathscr{C}_{1+}$$
is the canonical map (see \rref{esslie3}). However, an operad module over $\mathscr{C}_{1+}$ models a
{\em bimodule}, not a left or right module. Nevertheless, for a $\mathscr{C}_{1+}$-algebra $A$, a left
(or right) $(\mathscr{C}_{1+},A)$-module can be defined. Recall that connected components define an
equivalence of operads in spaces:
$$\mathscr{C}_1\r\Sigma$$
where $\Sigma$ is the associative operad,
$$\Sigma(n)=\Sigma_n.$$
Denote by $\mathscr{C}^{L}_{1}(n)$ the fiber of $\mathscr{C}_1(n)$ over the isotropy subgroup
$\Sigma_{n}^{L}$ of $n$ in $\Sigma_n$ (isomorphic to $\Sigma_{n-1}$). Then a {\em left
$(\mathscr{C}_{1+},A)$-module} has structure maps
$$\mathscr{C}^{L}_{1}(n)_+\wedge A^{n-1}\wedge X\r X$$
which satisfy associativity and equivariance with respect to $\Sigma^{L}_{n}$. 
Note that the operad structure map takes
$$\mathscr{C}^{L}_{1}(n)\times\mathscr{C}_1(k_1)\times\dots\times
\mathscr{C}_1(k_{n-1})\times \mathscr{C}^{L}_{1}(k_n)\r\mathscr{C}^{L}_{1}(k_1+\dots+k_n).$$
The definition of a {\em right
$(\mathscr{C}_{1+},A)$-module} is analogous when we replace ``$n$" by ``$1$" and write $R$ instead of $L$.

\vspace{3mm}
\begin{proposition}
\label{pdrl1}
Let $\widetilde{\mathscr{C}}_{1+}$ be a cofibrant model of $\mathscr{C}_{1+}$. (We already assume $\mathscr{L}$
to be cofibrant). Let $g$ be a cofibrant Lie algebra. Then the derived category of $g$-representations is canonically
equivalent to the derived category of left (or, alternately, right) $Ug$-modules.
\end{proposition} 

\begin{proof}
Consider first the monad 
$$D(g,X)=(Lg, D_1(g,X))$$
in the category of pairs $(g,X)$ of spectra (by which we mean a product of two copies of the category of spectra)
defining ``$\mathscr{L}$-algebra $g$ and $(\mathscr{L},g)$-module $X$". Now consider the monad
$$E(A,X)=(\widetilde{C}_{1+}A, E_1(A,X))$$
in pairs of spectra $(A,X)$ defining ``$\widetilde{\mathscr{C}}_{1+}$-algebra $A$ and left
$(\widetilde{\mathscr{C}}_{1+},A)$-module $X$". Then we have
$$D_1(g,X)=\bigvee_{n\geq 1} \mathscr{L}(n)\wedge_{\Sigma_{n-1}}g^{\wedge (n-1)}\wedge X
$$
and
$$E_1(A,X)=\bigvee_{n\geq 1} \widetilde{\mathscr{C}}_{1+}^{L}(n)\wedge_{\Sigma_{n-1}}A^{\wedge (n-1)}\wedge X.
$$
One checks from the definition of the map $\mathscr{C}_{(k+1)+}\r \mathscr{J}\wedge \mathscr{C}_{k+}$
that the map $\mathscr{L}\r\mathscr{C}_{1+}$
actually induces a map
\beg{eprdl2}{\diagram D_1(g,X)\rto^\sim & E_1(g,X)\enddiagram}
which is an equivalence. (This is plausible because of Proposition \ref{psslie2}; however, note that more is being claimed here,
namely certain diagram formed using the explicitly defined maps commutes on the nose.) From this, we obtain a map
of monads
\beg{eprdl1}{D\r E}
which is an equivalence on the second coordinate. Now the two functors between the categories of 
$g$-representations and left $Ug$-modules are pushforward, and pullback along \rref{eprdl1}, composed with the
unit map
$$g\r u^*Ug.$$
When $(g,X)$ is a free $D$-algebra, \rref{eprdl2} implies that the map on the second coordinate induced by 
\rref{eprdl1} is an equivalence. For a general cofibrant $D$-algebra $(g,X)$, we can use the
two-sided bar construction $B(D,D,(g,X))$ and the associated filtration spectral sequence in the standard
way. 

The proof for right modules is analogous.
\end{proof}

\vspace{3mm}
The example given in the following proposition is fundamental for the development of our theory. By an
{\em abelian} Lie algebra we shall mean an $S$-Lie algebra equivalent to an $S$-Lie algebra 
obtained by pullback from the map of $\mathscr{S}$-operads
$$\mathscr{L}\r \mathscr{E}.$$
(Perhaps, a more precise term than ``abelian" would be ``$E_\infty$", but in this context it is overloaded.)

\begin{proposition}
\label{prdl2}
Let $g$ be a cofibrant abelian Lie algebra and let $R$ be an $E_\infty$-algebra (i.e. a $\mathscr{C}_{\infty +}$-algebra). 
Then a map of spectra
$$\lambda:g\r R$$
canonically determines a $g$-representation with underlying spectrum $R$.
\end{proposition}

When $R$ is equivalent to $S$ as an $E_\infty$-algebra, we shall refer to such a representation as a
{\em character}, and denote it by $S^\lambda$. For a general $E_\infty$-algebra $R$, we will speak more generally 
of {\em $R$-characters}, and use the notation $R^\lambda$.

\vspace{3mm}
\noindent
{\bf Proof of Proposition \ref{prdl2}:} Consider the diagram \rref{epsslie*}. As usual, we obtain a morphism
\beg{eeprdl2}{u_\sharp p^*\r q^*i_\sharp
}
from its adjoint
$$p^*\r u^*q^*i_\sharp=p^*i^*i_\sharp,$$
which is $p^*$ composed with the unit of the adjunction $(i_\sharp, i^*).$ Now \rref{eeprdl2} for an abelian cofibrant
$S$-Lie algebra $g$ can be written as a map of $\mathscr{C}_{1+}$-algebras
\beg{eug}{Ug\r q^*C_{\infty +} g}
where $C_{\infty +}$ is the monad associated to $\mathscr{C}_{\infty +}$. 
Now an operad algebra is always an operad module over
the same algebra. Thus, a map of spectra $g\r R$ determines by adjunction a map of $\mathscr{C}_{\infty +}$-algebras
$$C_{\infty +} g\r R,$$
and hence a $(\mathscr{C}_{\infty +}, C_{\infty +} g)$-module structure on $R$. By adjunction, this then determines a
$(\mathscr{C}_{1+}, q^*C_{\infty +} g)$-module structure on $R$, hence, by \rref{eug}, a
$(C_{1+}, Ug)$-module structure and hence a left (or alternately right) $Ug$-module structure. Now use
Proposition \ref{pdrl1}.
\qed

\vspace{3mm}
\begin{proposition}(The projection formula)\label{pproj}
Returning to the diagram \rref{epsslie*}, on a cofibrant $\mathscr{E}$-algebra $X$, the canonical morphism
$$u_\sharp p^*X\r q^*i_\sharp X$$
(alternately, thinking of $X$ as an abelian Lie algebra, $UX\r C_{\infty +} X$), is an equivalence.
\end{proposition}

\begin{proof}
By a colimit argument, it suffices to consider the case when $X$ is a finite cell spectrum. By modeling the map of
interest as
\beg{eproj1}{B(C_{1+},L,X)\r C_{\infty +}X,
}
we see that \rref{eproj1} is, in the category of spectra, a wedge sum of maps of bounded below spectra, so
we may use homology. In homology with coefficients in $\Q$, \rref{eproj1} induces an isomorphism just by the ordinary
PBW theorem. With coefficients in $\Z/p$, the homology of $LY$ is $E_\infty$-Quillen homology of the abelian
commutative $\Z/p$-algebra $H_*(Y,\Z/p)$. This was calculated in \cite{hm, wr}. The answer is the free Lie algebra
on the Koszul dual to the Dyer-Lashof algebra. Essentially, they can be thought of as Steenrod operations without the 
admissibility relations. Using the Grothendieck spectral sequence as in \cite{hm1}, we see that \rref{eproj1}
induces an isomorphism in $\text{mod}\; p$ homology.
\end{proof}

\vspace{3mm}
\begin{proposition}
\label{ppbw}
Suppose $X$, $X_1,\dots,X_m$ are $\mathscr{L}$-algebras which are finite as spectra, and 
$$\phi_i:X_i\r X$$
are morphisms of $\mathscr{L}$-algebras such that the map of spectra
$$\bigvee_{i=1}^{m}\phi_i: \bigvee_{i=1}^{m}X_i\r X$$
is an equivalence of spectra. Then the map of spectra
$$\bigwedge_{i=1}^{m} UX_i \r UX$$
given by $U\phi_i$ and $\mathscr{C}_{1+}$-multiplication is an equivalence.
\end{proposition}

\begin{proof}
This proposition is proved by the same method as the previous one, using a calculation of homology.
\end{proof}

\vspace{3mm}
\subsection{Products}\label{ssproducts}
Let $\mathscr{C}$ be a cofibrant operad over $S$, and $R$ a $\mathscr{C}$-algebra. The first goal of this subsection 
is to define the \emph{operad tensor product} of modules over $(\mathscr{C}, R)$. 

Let $C$ be the associated monad of $\mathscr{C}$, and let $X, Y, Z$ be spectra. 
Recall that 
$$ CX = \bigvee_{n \geq 0} \mathscr{C}(n) \wedge_{\Sigma_n} X^{\wedge n}. $$
Define the functors
$$C_{(1)}(X,Y)=\bigvee_{n\geq 0} \mathscr{C}(n+1)\wedge_{\Sigma_n} X^{\wedge n}\wedge Y,$$
$$C_{(2)}(X, Y, Z)=\bigvee_{n\geq 0}\mathscr{C}(n+2)\wedge_{\Sigma_n} X^{\wedge n}\wedge Y\wedge Z.$$
Then define
$$\underline{C}_{(1)}(X, Y)=(CX,C_{(1)}(X, Y)) .$$
This is a monad in the category of pairs of spectra, defining ``$\mathscr{C}$-algebra $R$, $(\mathscr{C},R)$-module''. 
We also define
$$\underline{C}_{(1,1)}(X, Y, Z)=(CX,C_{(1)}(X, Y),C_{(1)}(X, Z)) . $$
This is a monad in the category of
triples of spectra, defining ``$\mathscr{C}$-algebra $R$, $(\mathscr{C},R)$-modules $M,N$''. There is also 
the functor
$$\underline{C}_{(2)}(X, Y, Z)=(CX,C_{(2)}(X, Y, Z)) .$$
This is now a (left $\underline{C}_{(1)}$, right  $\underline{C}_{(1,1)}$)-functor from triples of spectra to pairs of
spectra.

Now let $R$ be a $\mathscr{C}$-algebra, and let 
$M$ and $N$ be $(\mathscr{C}, R)$-modules. 
The {\em operad tensor product of $(\mathscr{C},R)$-modules $M$, $N$} is defined as
$$M\otimes_{(\mathscr{C},R)} N = B(\underline{C}_{(2)},\underline{C}_{(1,1)}, (R,M,N)).$$
It is a $(\mathscr{C}, B(C,C,R))$-module. Without further work, this operation does not have good point-set properties, 
but it works in the derived category.

In the derived category, the operad tensor product is commutative, since the two variables involved
play symmetrical roles. For example, for bimodules $M$, $N$ over an associative ring $A$, 
in the cofibrant case, their operadic tensor product using the associative operad has the homotopy type
$(M\wedge_A N)\vee (N\wedge_A M)$. However, the question of associativity and unitality is tricky.
In general, neither holds, although we may readily define
$$\left({\bigotimes_{i=1}^{n}}\right)_{(\mathscr{C},R)} M_i$$
for any $n\geq 0$ in the same way, and we have canonical comparison maps in the derived category
\beg{eprodocts+}{
\left({\bigotimes_{j=1}^{k}}\right)_{(\mathscr{C},R)}\left(
\left({\bigotimes_{i=1}^{n_j}}\right)_{(\mathscr{C},R)} M_{j,i}
\right)
\r 
\left({\bigotimes_{i,j}}\right)_{(\mathscr{C},R)} M_{i,j}.
}
For $\mathscr{C}=\mathscr{C}_{\infty+}$, the comparison maps \rref{eprodocts+} are equivalences 
(the first version of \cite{ekmm} was written using this fact), and it follows from
Koszul duality that the operad tensor product over $\mathscr{L}$ is associative and unital also. On the other hand,
it is easy to observe that the operad tensor product of an empty set of $(\mathscr{C},R)$-modules is always equivalent to
$R$. 
\vspace{3mm}

Now we will specialize to the case where $\mathscr{C} = \mathscr{L}$. We will use lower-case $g$ to denote 
a Lie algebra (instead of the general $R$). In this case, we will see that $g$ is {\em not} the unit for $\otimes_{(\mathscr{L},g)}$.
(Instead, the``trivial representation" over $g$ is the unit for the operadic tensor product).

\begin{lemma}
\label{lprod1}
For representations $M$, $N$ of a cofibrant $S$-Lie algebra $g$,
$$M\otimes_{(\mathscr{L},g)}N\sim M\wedge N$$
as spectra.
\end{lemma}

\begin{proof}
One checks that the natural composition map
\beg{elprod1}{\mathscr{L}(2)\wedge L_{(1)}(g,M)\wedge L_{(1)}(g,M)\r L_{(2)}(g,M,N)}
is an equivalence (again, this requires Proposition \ref{psslie2} and some care in matching terms). 
The key observation is that
$$L_{(1)}(g,M)\simeq \bigvee_{k\geq 1} M,$$
whereas
$$L_{(2)}(g,M,N)=\bigvee_{n\geq 2} (\bigvee_{n-1} M\wedge N)=
(\bigvee_{k\geq 1}M)\wedge (\bigvee_{\ell\geq 1} N)
$$
(as $\Sigma_{n-1}/\Sigma_{n-2}$ is a set of $n-1$ elements).

This establishes
the statement of the Lemma in the case of a free module over a free Lie algebra.
Now \rref{elprod1} is, by definition, a right $L_{(1,1)}$-functor, so we can use the usual bar construction argument for 
the general case.
\end{proof}

\vspace{3mm}
From now on, we shall denote the product ``$\otimes_{(\mathscr{L},g)}$'' also simply as ``$\wedge$".

\vspace{3mm}
\begin{theorem} 
\label{tchar}
(Products of characters) Let $g$ be a cofibrant abelian $S$-Lie algebra, let $R_1,R_2$ be cofibrant $E_\infty$-algebras, and let 
$$\phi_i:g\r R_i, \; i=1,2$$
be maps of spectra. Define
\beg{echarpsi}{\psi=\phi_1\wedge 1 + 1\wedge \phi_2:g\r R_1\wedge R_2.}
Then we have a canonical equivalence in the derived category
\beg{echar1}{(R_1)^{\phi_1}\otimes_{(\mathscr{L},g)} (R_2)^{\phi_2} \sim (R_1\wedge R_2)^\psi.
}
\end{theorem}

\vspace{3mm}
Note that in our category, a smash product of
$E_\infty$-algebras is canonically an $E_\infty$-algebra. Also, in \rref{echarpsi}, the sum refers to the sum in
the derived category of spectra, which is an additive category.

\vspace{3mm}



We will need a number of steps to prove the theorem. 
First, we introduce a $\otimes$-product of left $(C_{1+},R_i)$-modules, $i=1,2$, which coincides with
the {\em external} smash-product in the case $\mathscr{C}_{\infty+}$-algebras $R_i$. 

Define, for spectra $X$, $Y$,
$$C_{1+(1)}^{L}(X, Y)=\bigvee_{n\geq 0} \mathscr{C}^{L}_{1+}(n+1)\wedge_{\Sigma_n} X^n\wedge Y$$
where $\mathscr{C}^{L}_{1}(n+1)$ is the fiber of $\mathscr{C}_1(n+1)$ over the isotropy subgroup
of $n+1$ in $\Sigma_{n+1}$ (which is a copy of $\Sigma_n$). For spectra $X_1, X_2, Y, Z$, we also define
the functor
$$C^{L}_{1+(2)}(X_1,X_2,Y, Z)=
\bigvee_{k,\ell\geq 0} \mathscr{C}_{1+}^{LL}(k,\ell)\wedge_{\Sigma_k\times\Sigma_\ell}X_{1}^{\wedge k}
\wedge X_{2}^{\wedge \ell}\wedge Y \wedge Z
$$
where $\mathscr{C}_{1}^{LL}(k,\ell)\subseteq \mathscr{C}_{1}(k+\ell+2)$ is the fiber over the subgroup
of $\Sigma_{k+\ell+2}$ of permutations preserving the blocks $\{1,\dots,k\}$, $\{k+1,\dots,k+\ell\}$
and fixing $k+\ell+1$, $k+\ell+2$. This subgroup is a copy of $\Sigma_k \times \Sigma_{\ell}$.
We have a monad
in the category of pairs of spectra
$$\underline{C}_{1+(1)}^{L}(X, Y)=(C_{1+}X, C_{1+(1)}^{L}(X, Y))$$
defininig ``$\mathscr{C}_{1+}$-algebra $R$ and left $R$-module $M$".
We also have a monad in the category of quadruples of spectra
$$\underline{C}_{1+(1,1)}^L(X_1,X_2,Y_1,Y_2)=
(C_{1+}X_1,C_{1+}X_2,C_{1+}^{L}(X_1,Y_1),C_{1+}^{L}(X_2,Y_2))$$
defining ``$\mathscr{C}_{1+}$-algebra $R_i$ and left $R_i$-module $M_i$, $i=1,2$".

Recall the map $q: \mathscr{C}_{1+} \r 
\mathscr{C}_{\infty +}$ from diagram \rref{epsslie*}.

\begin{lemma}
\label{lchar1}
Let $R$ be a cofibrant $\mathscr{C}_{\infty+}$-algebra. Then the derived category of left (alternatively, right)
$q^*R$-modules is canonically equivalent to the derived category of $(\mathscr{C}_{\infty+},R)$-modules.
\end{lemma} 

\begin{proof}
We have a map induced by $q$:
$$C_{1+(1)}^{L}(R,M)\r C_{\infty+(1)}(R,M).$$
This gives a functor from $(\mathscr{C}_{\infty+},R)$-modules to $(\mathscr{C}_{1+}^{L},R)$-modules.
Next, we note that free modules are equivalent in both categories, since for $\mathscr{C}_{\infty+}$-algebras $R$,
the canonical map
\beg{echar+}{B(C_{1+(1)}(?,M), C_{1+},R)\r
B(C_{\infty+(1)}(?,M),C_{\infty+},R)
}
is an equivalence.
Denoting the target as $\mathfrak{G}$, then
$$M\mapsto B(\mathfrak{G},C_{1+(1)}(B(C,C,R),?),M)$$
is the desired inverse functor on the level of derived categories.
\end{proof}

We also have a functor in the category of pairs of spectra
$$C_{1+}^{\otimes}(X_1,X_2)=\bigvee_{k,\ell\geq 0} \mathscr{C}_{1+}^{\otimes}(k,\ell)
\wedge_{\Sigma_k\times\Sigma_\ell} X_{1}^{k}\wedge X_{2}^{\ell}$$
where $\mathscr{C}_{1}^{\otimes}(k,\ell)$ is the fiber of $\mathscr{C}_1(k+\ell)$ over the subgroup 
$\Sigma_{k} \times \Sigma_{\ell}$ of 
$\Sigma_{k+\ell}$ preserving the blocks $\{1,\dots,k\}$, $\{k+1,\dots,k+\ell\}$. Then
$C_{1+}^{\otimes}$ is a (left $C_{1+}$, right $(C_{1+},C_{1+})$)-functor. For $C_{1+}$-algebras 
$R_1, R_2$, 
$$B(C_{1+}^{\otimes},(C_{1+},C_{1+}),(R_1,R_2))$$
is a model for the external smash product of $C_{1+}$-algebras. We will denote it by $R_1\wedge_1 R_2$. 
(The underlying operation on spectra is, indeed, ``$\wedge$".)

For the external smash product of modules, define
$$\underline{C}_{1+(2)}^{L}=(C_{1+}^{\otimes}, C_{1+(2)}^{L}).$$
This is a (left $\underline{C}_{1+(1)}^{L}$, right $\underline{C}_{1+(1,1)}^{L}$)-functor, so for left modules 
$M_1$ over $R_1$ and $M_2$ over $R_2$, 
$$B(C_{1+(2)}^{L},\underline{C}_{1+(1,1)}^{L},(R_1,R_2,M_1,M_2))$$
is the exterior smash product $M_1\wedge_1 M_2$, as a left $R_1\wedge_1 R_2$-module.

Analogously, we can define external smash product of $\mathscr{C}_{\infty+}$-algebras and their modules. 
For spectra $X_1, X_2, Y_1, Y_2$, define
$$C_{\infty+}^{\otimes}(X_1,X_2)=\bigvee_{k,\ell\geq 0}\mathscr{C}_{\infty+}(k+\ell)
\wedge_{\Sigma_k\times\Sigma_\ell} X_{1}^{k}\wedge X_{2}^{\ell} ,$$
$$C_{\infty+(2)}^{\otimes}(X_1,X_2,Y_1,Y_2)=
\bigvee_{k,\ell\geq 0} \mathscr{C}_{\infty+}(k+\ell+2)\wedge_{\Sigma_k\times \Sigma_\ell}
X_{1}^{\wedge k}\wedge X_{2}^{\wedge \ell}\wedge Y_1\wedge Y_2 .$$
Then the external smash product of $\mathscr{C}_{\infty+}$-algebras $R_1$ and $R_2$ is 
$$B(C_{\infty+}^{\otimes},(C_{\infty+},C_{\infty+}), (R_1,R_2))$$
which we will, for the moment, denote by $R_1 \wedge_\infty R_2$. Then there is a morphism of 
$\mathscr{C}_{1+}$-algebras
$$R_1\wedge_1 R_2\r R_1\wedge_\infty R_2$$
which is an equivalence (we omit $q^*$ on the left hand side). Similarly, for modules, define
$$\begin{array}{l}\underline{C}_{\infty+(1,1)}(R_1,R_2,M_1,M_2)
=\\(C_{\infty+}R_1, C_{\infty+}R_2,
C_{\infty+(1)}(R_1,M_1),C_{\infty+(1)}(R_2,M_2))
\end{array}$$
and
$$M_1\wedge_\infty M_2=
B(C_{\infty+(2)}^{\otimes}, \underline{C}_{\infty+(1,1)}, (R_1,R_2,M_1,M_2))$$
is a smash product functor from $(\mathscr{C}_{\infty+}, R_i)$-modules $M_i$, $i=1,2$, to
$\mathscr{C}_{\infty+}$-$R_1\wedge_\infty R_2$-modules. Again, 
for $\mathscr{C}_{\infty+}$-algebras $R_i$, there is a canonical map of $\mathscr{C}_{1+}$-modules
from the $\mathscr{C}_{1+}$-smash product of left $R_i$-modules to the smash product
of $\mathscr{C}_{\infty+}$-modules, which is an equivalence.

\vspace{3mm}
\noindent
{\bf Proof of Theorem \ref{tchar}:}
Denoting by $L$ the monad associated with the operad
$\mathscr{L}$, analogously to the comparison of Lie algebra representations and left (alternately, right)
modules over the universal enveloping algebra, we obtain a natural map
$$Lg\r C_{1+}^{\otimes}(g,g). $$
From this, we get a map
$$B(L,L,g)\r B(C_{1+}^{\otimes}(g,g), L,g),$$
which can be interpreted as a map of $S$-Lie algebras
\beg{echari}{g\r Ug\wedge_1 Ug.
}
Using the adjunction, this gives a map of $\mathscr{C}_{1+}$-algebras (a ``non-rigid Hopf algebra structure")
\beg{echar*}{Ug
\r Ug\wedge_1 Ug.}
Also for modules, we get maps
$$L_{(1)}\r C_{1+(1)}^{L},$$
$$L_{(2)}\r C_{1+(2)}^{L}$$
and comparing the $2$-sided bar constructions of monads, one shows that the pullback of
a $\wedge_1$-product of left $Ug$-modules $M_1,M_2$ via \rref{echar*} is equivalent to 
$M_1\otimes_{(L,g)}M_2$.

Now when $g$ is abelian, \rref{echari} expands to
\beg{echar**}{\diagram
g\rto & Ug\wedge_1 Ug\rto^\sim & Cg\wedge Cg\rto^\sim &C(g\vee g).
\enddiagram
}
(The last equivalence can again be seen directly using our methods, or alternately
it follows from commutation of left adjoints with coproducts.) By inspection, the composition
\rref{echar**} is homotopic, as a map of spectra, to the composition
$$\diagram
g\rto^(.4){Id\times Id} & g \times g & g\vee g\lto_\sim \rto^(.4)\eta & C(g\vee g).
\enddiagram$$
Now \rref{echar**}, by adjunction, gives
$$Cg\r C(g\vee g).$$
The composition 
$$\diagram
Ug\rto^\sim & Cg\rto & C(g\vee g),
\enddiagram
$$
in the derived category of $\mathscr{C}_{1+}$-algebras, is homotopic to \rref{echar*} by uniqueness 
of adjoints ($U$ is a Quillen left adjoint). We have proved that the diagram
$$\diagram
Ug \rto \dto & Ug\wedge_1 Ug\dto\\
Cg\rto & Cg\wedge_\infty Cg
\enddiagram$$
is commutative up to homotopy in $\mathscr{C}_{1+}$-algebras. Composing with a smash product
of $\mathscr{C}_{\infty+}$-morphisms
$$Cg\r R_i,\; i=1,2,$$
which are classified by morphisms of spectra
$$g\r R_i$$
completes the proof of the Theorem.
\qed

\vspace{3mm}
\begin{corollary}
\label{cchar}
Let $g$ be an abelian $S$-Lie algebra and let $\lambda:g\r S$ be a character. Then 
$$S^\lambda \wedge S^{-\lambda}\sim S^0$$
in the derived category of $g$-representations. In particular, $S^\lambda$ is invertible and hence
strongly dualizable in this category with respect to this symmetric monoidal structure.
\end{corollary}
\qed

\vspace{3mm}
\section{The $S$-Lie algebra $gl_nS$ and its representations. The restricted category.}
\label{sgln}

In this section, we specialize to the general linear Lie algebra $gl_n$ over $S$, its subalgebras, and their representations, 
especially Verma modules. Although it is possible to obtain $gl_n$ over $S$ as an endomorphism spectrum in the 
category of $S$-modules, we need to obtain a ``matrix-theoretic" definition of $gl_n$, in order to carry out
several 
basic constructions, such as the weight-lattice grading on $gl_n$ (see Section 4.2 below). To this end, 
we make use of Elmendorf and Mandel's multiplicative infinite loop space machinery of multicategories~\cite{em}.

For standard basic references to Verma modules and other constructions from the representation theory of 
Lie algebras, see, e, g. \cite{knapp}. 

\subsection{Verma modules}
\label{ssverma}
We begin by constructing $gl_n$ as a ``matrix algebra" over $S$.
To construct $\mathscr{C}_{1+}$-algebras in $\mathscr{S}$, it suffices by \cite{em} to construct
multifunctors from the ``associative" operad $\Sigma$ to the multi-category $Perm$ of permutative 
categories. By the results of \cite{hkk}, it suffices to construct a {\em weak} multifunctor, i.e. up to coherence 
isomorphisms with appropriate coherence diagram (see also \cite{kl}). (Note: Technically, the target of
the Elmendorf-Mandel functor is symmetric spectra, but a version of it also exists which lands
in the category used in the present paper -  see \cite{scompare}.)

Let $fSet$ denotes the category of finite sets
and bijections.
As a multifunctor from $\Sigma$ to $Perm$, the $\mathscr{C}_{1+}$-algebra $gl_nfSet$ 
sends unique object of $\Sigma$ to the product of $n\times n$ copies of the category $fSet$. This 
target object is thought of 
as the permutative category of $n \times n$-matrices, where each entry is a finite set. On morphisms, 
$$gl_nfSet\times\dots\times gl_nfSet\r gl_nfSet$$
is given by ``multiplication of matrices".
 On the entries of the matrices, we use Cartesian product of sets for 
 ``multiplication", and disjoint union of sets for ``addition".

In fact, this construction can be generalized: Let 
$$T\subseteq\{1,\dots,n\}\times \{1,\dots,n\}$$
be a set of pairs which is transitive as a binary relation. Then 
$$g_TfSet=\prod_{(i,j)\in T} fSet$$
is a weak multifunctor from the associative operad $\Sigma$ to the multicategory $Perm$, where the 
target object is the category of $n \times n$-matrices of finite sets, where the entries for all positions $(i, j) \notin T$
are empty sets. 
Thus, applying
the Elmendorf-Mandell functor $\mathscr{K}$ \cite{em,hkk,kl}, we obtain a $\mathscr{C}_{1+}$-algebra, hence an
$\mathscr{L}$-algebra, $g_TS$. 

There are a number of examples of $\mathscr{L}$-algebras $g_TS$
which will be useful to us, for example the Borel subalgebra $b_+S$ and its opposite $b_-S$ of upper (resp. lower) triangular
matrices, and the corresponding nilpotent ``Lie subalgebras" $n_+S$ and $n_-S$, 
and also the Cartan algebra of diagonal matrices $h_nS$. We also recall that a parabolic subalgebra of $gl_n$ is 
a subalgebra containing the Borel subalgebra. Hence, a parabolic subalgebra $p$ 
of $gl_n$ is a subalgebra consisting of 
block upper (or lower) triangular matrices corresponding to a given partition $i_1 + \cdots + i_k = n$ of $n$.
For a parabolic subalgebra $p$ of $gl_n$, its Levi factor is the factor algebra consisting of its diagonal blocks. 
Hence, we can also obtain parabolic subalgebras $pS$ and their Levi factors in this manner. 

\vspace{3mm}
Now we area in position to define Verma modules 
over $S$. Consider the $\mathscr{C}_{1+}$-subalgebra $b_+fSet$ which consists of upper triangular 
$n\times n$ matrices in $fSet$ (embedded into $gl_nfSet$ by sending the remaining entries into
the empty set). In turn, $b_+fSet$ projects onto the algebra $h_nfSet$ of
{\em diagonal} $n\times n$ matrices in $fSet$ (forgetting the above-diagonal terms). Applying the Elmendorf-Mandell
realization functor $\mathscr{K}$, we obtain a diagram of $\mathscr{C}_{1+}$-algebras:
\beg{evermad}{\diagram
b_+S\rto^\alpha \dto_\pi  & gl_nS\\h_nS &
\enddiagram
}
We further note that the Borel $\mathscr{C}_{1+}$-algebra structure on $h_nS$ is the pullback of a
$\mathscr{C}_{\infty+}$-structure, as the $\Sigma$-structure on $h_nfSet$ comes from the \v{C}ech
resolution $E\Sigma$ by the results of Elmendorf-Mandell \cite{em} (note that $h_nfSet$ is the product
of $n$ copies of $fSet$ where the operations are done component-wise). Indeed, as a spectrum, 
$h_nS$ is just the wedge sum of $n$ copies of $S$. 

Now pull the diagram \rref{evermad} back to Lie algebras via $u^*$. By what we just observed, $h_nS$ is an abelian
Lie algebra, and hence by Proposition \ref{prdl2} above, an $n$-tuple of integers $\lambda=(k_1,\dots,k_n)$ (specifying
homotopy classes of maps $S\r S$) specifies a map $\lambda: h_n \rightarrow S$, 
and hence a representation $S^\lambda$ of $h_nS$ on $S$. Let 
$\widetilde{\pi^*S^\lambda}$ be a cofibrant replacement of $\pi^*S^\lambda$. Put
$$V_\lambda=\alpha_\sharp\widetilde{\pi^*S^\lambda}.$$
(Recall here that $\alpha_\sharp$ is the left adjoint to $\alpha^*$ on the level of derived categories.)
This is the {\em Verma module} over $gl_nS$ induced from the character $\lambda$. (Whenever considering this
on the nose, we will automatically assume cofibrant replacement has been performed, without
indicating it in the notation.) 

To see that this is an analogue of the Verma module from classical Lie theory, 
we recall that for classical (complex) Lie algebras, for a given $\lambda: h_n \rightarrow \mathbb{C}$, 
$\mathbb{C}_{\lambda}$ is the 1-dimensional representation of the Borel algebra $b_+$, where $h_n$ 
acts by $\lambda$, and the nilpotent subalgebra $n_+$ acts trivially. Then the Verma module is defined to be 
$$ V_{\lambda} = U(gl_n) \otimes_{U(b_{+})} {\mathbb C}_{\lambda}  $$
where $U$ is the universal enveloping algebra. 
In our context, the analogue of $\mathbb{C}_{\lambda}$, as a $b_{+}S$-representation, is $\pi^{\ast}S^{\lambda}$, 
and the pushforward $\alpha_{\sharp}$ performs the role of the $U(gl_n) \otimes_{U(b_{+})}?$. 

In accordance with conventions of
representation theory \cite{sussan}, however, when using numerical subscripts, we perform a {\em $\rho$-shift}:
We put
$$(\overline{k}_1,\dots,\overline{k}_n)=(k_1,\dots,k_n)+\rho,$$
where 
$$\rho=(\frac{n-1}{2},\frac{n-3}{2},\dots,\frac{1-n}{2})$$
is half the sum of all positive roots. When indexing by numbers, we then write
$$V_{(\overline{k}_1,\dots,\overline{k}_n)}=V_\lambda.$$
It is necessary to point out, however, that for $n$ even, $\rho$ is not an integral weight. Therefore, stritly 
speaking, Verma modules $V_{(\overline{k}_1,\dots,\overline{k}_n)}$ for $n$ even will exist only up 
to shifting the weight by $(a,\dots,a)$ where $a\in (1/2)+\Z$. 

\vspace{3mm}
\subsection{Variants of the construction. The graded category and the $p$-complete category.}\label{sres}
We shall also refer to the derived category of $gl_nS$-representations as defined so far as the {\em unrestricted}
category of representations. While we will need to use this category, and will prove some results about it,
at present a complete calculation of even a single non-zero $Ext$-group in this category is out of reach. Typically,
such a group is a homotopy group of the dual $DX$ of an infinite spectrum $X$, where we, perhaps, have some
control over the homology of $X$.

Because of that, we shall also consider some variants of the derived category of representations which are
more treatable. One tool we shall often use is {\em completion at $p$} where $p$ is
a prime number. By this, we mean Bousfield localization of 
$\mathscr{S}$ at the Moore spectrum $M\Z/p$. This localizations was originally defined by \cite{bousfield}, and 
constructed in \cite{ekmm}, Chapter VIII, for $\mathscr{S}$ and algebras in this category. 
As shown in \cite{ekmm}, Section VIII.2, localization of this type commutes with the 
forgetful functor from algebras over a cofibrant operad in $\mathscr{S}$ to $\mathscr{S}$, and hence all of our 
constructions, at least on the derived level, can be readily carried over to the $p$-complete category $\mathscr{S}_p$.
We will typically specialize our calculations to the $p$-complete category at $p>>k$, which means that $p$ is larger than
some constant multiple of $k$.

\vspace{3mm}
The other important variant of the category of $gl_nS$-representations is obtained as follows. The entire
construction of the $\mathscr{C}_{1+}$-algebra (and therefore $\mathscr{L}$-algebra) $gl_nS$, and the Verma modules, 
can be made $\Lambda$-graded, where $\Lambda$ is the weight lattice of $gl_n$. 
Recall that in classical Lie theory, $gl_n$ is graded by the roots. Here, 
the weight lattice is
\beg{eonehalf}{\Lambda = \Z^n.}
In this grading, diagonal matrices
have degree $0$, and the matrix $e_{i,j}$ (i. e. with the only nonzero entry being 1 in the $(i, j)$-th position for $i \neq j$) 
has degree 
$$(0,\dots,0,1,0,\dots,0,-1,0,\dots,0)$$ 
where the $1$ is in the $i$'th and
$-1$ is in the $j$'th position.
Matrix multiplication is additive with respect to this grading, giving the grading for the entire $gl_n$. 

In our context, this grading can be constructed for the input of the Elmendorf-Mandell infinite loop space 
machine, by which we constructed our $S$-Lie algebras. More specifically, we define the $\Z^n$-grading on the 
category of $n \times n$-dimensional matrices of finite sets, just the same 
as for ordinary matrices. 
Therefore, we can apply a $\Z^n$-graded version of the Elmendorf-Mandell functor (agreeing, in fact,
to sum only homogeneous terms in the same degree), creating a $\Z^n$-graded version of a $\mathscr{C}_{1+}$-algebra.
Pulling back along the map of operads $u: \mathscr{L} \rightarrow \mathscr{C}_{1+}$ from 
Proposition~\ref{psslie1} gives the $\Z^n$-graded 
$\mathscr{L}$-algebra. Obviously, all of the definitions (such as representations) and results 
established so far can be carried to the
$\Lambda$-graded context.  

\vspace{3mm}
\subsection{Function objects}\label{ssfun}
Function objects are generally obtained as right adjoints to versions of the smash product. Several different
flavors of such functors arise in our context. For an $S$-Lie algebra $g$ and a $g$-representation $V$,
we denote by $\mathscr{F}_g(V,?)$ the right adjoint to $V\wedge ?$, where $\wedge$ is the internal smash product in 
the category of $g$-representations. 
We will be typically interested in the case when $V$ is cofibrant, in which case this functor coincides with its
right derived functor. On the other hand, if $X$ is a spectrum, and $V$ is a $g$-representation, then the
ordinary smash product $X\wedge V$ (in the category of 
spectra) clearly has a canonical structure of a $g$-representation, where the $g$-action 
$$ \mathscr{L}(n) \wedge g^{\wedge n-1} \wedge X \wedge V \rightarrow X \wedge V $$
arises from the $(\mathscr{L}, g)$-module structure on $V$. 
We denote
the right adjoint to this functor by $F_g(V,?)$ from $g$-representations to spectra. Again, in the case
of $V$ cofibrant,
this is the same as the corresponding right derived functor. 

Two comments are in order. First, note that the functor
$F_g$ can be generalized to modules over general operad algebras, and also to left (resp. right) modules over a
$\mathscr{C}_{1+}$-algebra $A$. In this last case, we denote the function object by $F_{A}$ and observe that
the right derived funtors of $F_g$ and $F_{Ug}$ coincide. The other comment is that the functors $F_g$ and 
$\mathscr{F}_g$ are in fact related. The point is that there is a trivial representation functor from spectra
to $g$-representations, which has a right adjoint, which we may think of as ``$g$-invariants''. On the level of 
right derived functors, then, $F_g$ is the $g$-invariants of $\mathscr{F}_g$.

As discussed in the previous section, for all the Lie algebras we constructed and the categories of their representations, 
there are $\Lambda = \Z^n$ graded versions. We will denote the graded versions of these
functors by the superscript $0$, i.e. we write $\mathscr{F}^0_g$, $F^0_g$, etc. 

 \vspace{3mm}
 \vspace{3mm}
\subsection{Morphisms of Verma modules. Blocks}
The main results of this subsection are Thereoms \ref{tchb} and \ref{tblocks}.

We begin with morphisms of characters. From now on, 
we will omit the $S$ in the notation for $S$-Lie algebras.

\begin{theorem}
\label{tchb}
Let $\lambda,\mu:h_n \r S$ be morphisms in $\mathscr{S}$ (determining characters by
Proposition \ref{prdl2}). Then
\beg{etchb}{F^{0}_{h_n}(S^{\lambda},S^\mu)=\begin{array}{ll}
0 & \text{if $\lambda\neq \mu$}\\
DC_{\infty+}(\Sigma h_n) & \text{if $\lambda=\mu$}\end{array}
}
where $\Sigma$ denotes the suspension of a spectrum
and the characters are considered to be concentrated in the degree given by their weight.
\end{theorem}

\begin{proof}
Since $h_n$ is in degree $0$, $\lambda\neq \mu$ implies
$$F^0(Uh_n^{\wedge k}\wedge S^\lambda,S^\mu)=0$$
where $F^0$ denotes the function object (a spectrum) in the category of graded spectra. 
Consequently,
$$F^{0}_{Uh_n}(B(Uh_n, Uh_n,S^\lambda),S^\mu)=0.$$
Consider therefore the case $\lambda=\mu$. By the graded version of Corollary \ref{cchar}, it suffices
to consider the case $\lambda=\mu=0$. In this case, we can use the adjunction
$$F^{0}_{h_n}(S,S)=F_{C_{\infty+}h_n}(S,S)=
F_{C_{\infty+}h_n}(S,F(S,S))=F(S\wedge_{C_{\infty+}h_n}S,S).
$$
Now 
$$S\wedge_{C_{\infty+}h_n}S\sim C_{\infty+}\Sigma S,$$
as claimed.
\end{proof}

\noindent
{\bf Comment:}  The dual $DC_{\infty+}(\Sigma h_n)=F(C_{\infty+}\Sigma h_n,S)$ can be calculated. It is 
the product of spectra of the form
\beg{ecarls1}{D E\Sigma_{k+}\wedge_{\Sigma_k}S^{k\alpha}
}
where $\alpha$ is the sign representation of $\Sigma_k$. (Recall that $EG$ for a finite group $G$ is a
free $G$-CW complex which is non-equivariantly contractible.) Now using Spanier-Whitehead duality, \rref{ecarls1}
can be rewritten as
\beg{ecarls2}{F(E\Sigma_{k+}, S^{-k\alpha})^{\Sigma_k},
}
which, by Carlsson's theorem \cite{carlsson}, is the completion at the augmentation ideal of the Burnside ring of
the fixed point spectrum
\beg{ecarls3}{(S^{-k\alpha})^{\Sigma_k}.}
Now we have a cofibration sequence
\beg{ecarls4}{(\Sigma_k/A_k)_+\r S^0\r S^\alpha}
where $A_k$ is the alternating group, and its Spanier-Whitehead dual
\beg{ecarls5}{S^{-\alpha}\r S^0\r (\Sigma_k/A_k)_+.}
This means that $S^{-k\alpha}$ is the iterated fiber of the cube
\beg{ecarls6}{
\bigwedge_k(S^0\r (\Sigma_k/A_k)_+).
}
Since taking fixed points preserves iterated homotopy fibers, 
$(S^{-k\alpha})^{\Sigma_k}$ is obtained by applying $\Sigma_k$-fixed points to the corners of the cube 
\rref{ecarls6}, which are $S^0$, or wedges of copies of $(\Sigma_k/A_k)_+$, and their fixed point spectra
are
$$\bigvee_{G\subset \Sigma_k} BG_+$$
and wedges of copies of
$$\bigvee_{G\subset A_k} BG_+.$$

\vspace{3mm}

\begin{theorem}
\label{tblocks}
Suppose $a_i$, $b_i$, $i=1,\dots, n$ are two sequences of nonnegative integers.
Then, in the $p$-complete graded category,
$$F^{0}_{gl_n}(V_{(a_1,\dots,a_n)}, V_{(b_1,\dots,b_n)})\sim *$$
unless
\beg{eblock1}{(a_1,\dots, a_n)\geq (b_1,\dots,b_n)
}
(where $\geq$ denotes the ordering of weights, which is just the component-wise ordering)
and
\beg{eblock2}{\parbox{3.5in}{There exists a permutation $\sigma$ on $\{1,\dots,n\}$ such that
$(a_1,\dots,a_n)\equiv (b_{\sigma(1)},\dots,b_{\sigma(n)})\mod p$.}
}
\end{theorem}

\vspace{3mm}

\noindent
{\bf Comment:} 
We can show an analogous result with $F^0$ replaced by $F$ and condition \rref{eblock1} omitted, but
it is substantially more difficult, and there seems to be no benefit for our purposes in this paper.

\vspace{3mm}
\begin{proof}
First, by the definition of $V_{\lambda_1}$ and standard ``change of rings", we have
\beg{eblocks1}{
F^{0}_{gl_n}(V_{\lambda_1}, V_{\lambda_2})\sim F_{b_+}^{0}(S^{\lambda_1}, V_{\lambda_2}).
}
Now using further the fact that $S^{\lambda_1}$ is a pullback of an $h_n$-representation, 
the right hand side of \rref{eblocks1} can be further written as
\beg{eblocks2}{
F_{h_n}^{0}(S^{\lambda_1}, F_{b_+}^*(Uh_n, V_{\lambda_2})).
}
By the ``$*$'' superscript, we mean the graded object whose 
$\mu$-term is the $F^0$'s where the source is smashed with
$S^{-\mu}$. The functor $F^{\ast}_{b+}(Uh_n, -)$ is
the right adjoint to the forgetful functor from graded $b_+$-representations to graded $h_n$-representations.
Now again
\beg{eblocks3}{
 F_{b_+}^{*}(Uh_n, V_{\lambda_2})\sim F_{n_+}^*(S^{\lambda_1},V_{\lambda_2}).
}
This is the graded ``nilpotent stable homotopy" of $V_{\lambda_2}$. On the right hand side of \rref{eblocks3}, 
the significance of the ``$\lambda_1$"-decoration of $S$ is only for grading. 

Now all the weights of $n_+$ are positive. This means that in each given degree, the cosimplicial object calculating
the given term is, in fact, finite (in the sense that it has only finitely many non-trivial cosimplicial degrees, and
each degree is a finite spectrum).

This means that we can calculate with homology. Working in the $p$-complete category, specifically for
$H=H\Z/p$, we have
\beg{eblocks4}{
HF_{n_+}(S,V_{\lambda})\sim F_{Hn_+}(H,HV_{\lambda}).
}
Calculating the right hand side is pure ($E_\infty$-) algebra. It is, in fact, 
almost the same as doing the calculation in ordinary algebra in
characteristic $p$ \cite{Zforms}, with the exception that we have to include a discussion of higher
Dyer-Lashof operations. However, Dyer-Lashof operations in $\text{mod}\; p$
homology occur in weights which are multiples of $p$, and hence can be ignored for our purposes.

Let us calculate, then, in (classical) characteristic $p$. We will proceed by induction on $n$.
Consider the abelian graded Lie subalgebra of $n_+$ corresponding to the binary
relation
$$\{(1,j)\mid j=2,\dots,n\}.$$
Then we have a short exact sequence of Lie algebras
$$0\r a\r n_+\r n_{+}^{\prime}\r 0$$
where $n_{+}^{\prime}$ is the graded Lie algebra corresponding to the
transitive relation
$$\{(i,j)\mid 2\leq i<j\leq n\}.$$
Then we have a Hochschild-Serre spectral sequence
$$H^p(n_{+}^{\prime},H^q(a,V_{\lambda}))\Rightarrow
H^{p+q}(n_+, V_{\lambda}).$$
To consider the action of $a$ on $V_\lambda\sim Un_-$, we note that $e_{1,i}$ acts non-trivially
on $e_{i,1}$. Let $a_-$ be the graded Lie algebra corresponding to the transitive relation
$$\{(i,1)\mid 1<i\leq n\}.$$
Considering elements of $Un_-$ of the form
$$e_{2,1}^{k_2}\dots e_{n,1}^{k_n},$$
and considering the actions of $e_{2,1},\dots, e_{n,1}$ in order (i.e., using a sequence of
$n-1$ consecutive Hochschild-Serre spectral sequences), we see that
cocycles are only in the requisite weights by reducing to the $n=2$ case. The $n=2$ case is 
an easy direct calculation.

Thus, we have
\beg{eblocks+}{
H^*(a,V_\lambda)\cong H^*(a,Ua_-)\otimes Un_{-}^{\prime}
}
where $n_{-}^{\prime}$ is the opposite nilradical to $n_{+}^{\prime}$, consisting of
``below-diagonal elements". To compute the cohomology
of $n_{+}^{\prime}$ acting on \rref{eblocks+}, the weight shifting action of $n_{+}^{\prime}$ on
$H^*(a,Ua_-)$ can be neglected by a filtration spectral sequence. The statement then follows
from the induction hypothesis.
\end{proof}

\vspace{3mm}
Therefore, at least for Verma modules of regular weights
$$V_{(a_1,\dots,a_n)}, \; a_i\geq 0, \; i\neq j\Rightarrow a_i\neq a_j,$$
in the $p$-complete category with $p>> a_i$, there are no non-zero morphisms in the derived
category between Verma modules in different blocks of the BGG category $\mathcal{O}$ 
(i.e. $V_{(a_1,\dots,a_n)}$ and
$V_{(b_1,\dots,b_n)}$ where the sequence $(a_1,\dots,a_n)$ is not a permutation
of the sequence $(b_1,\dots,b_n)$ - see \cite{humphreys}).

\vspace{3mm}
\subsection{Some special finite $gl_k$-representations}
\label{sfin}

There is a natural (weak in the sense of \cite{kl}) 
action of $gl_kfSet$ on $(fSet)^k$. Rectifying and applying the $\mathscr{K}$-functor
of Elmendorf and Mandell \cite{em}, we obtain the ``standard representation" $W=W_k$ of $gl_kS$. 
This representation
is graded. It is useful to note that 
for any graded representation $U$, 
if $0\leq m<p$ and we are working in the $p$-complete category, we can construct the $m$'th {\em symmetric product}
\beg{efinsymm}{Sym^{m}U=(E\Sigma_{m+})\wedge_{\Sigma_m} U^{\wedge m}
}
and the {\em exterior product}
\beg{efinext}{\Lambda^mU=\Sigma^{-m}Sym^{m}(\Sigma U).
}
Note that the restriction $m<p$ is not necessary, but in this range the $\text{mod}\; p$ 
homology of $B\Sigma_m$ is $\Z/p$ in dimension $0$, and hence \rref{efinsymm}, \rref{efinext} are
finite (and of the expected dimension). 

In particular, for $k<p$, we have a $gl_kS$-representation
$$Det=\Lambda^k(W_k).$$
We will often use smash powers of the $Det$ representation.
As already remarked, the underlying graded spectrum of $Det_k=
Det$ is $S$ in the appropriate degree. ``Non-integral smash powers
of $Det$'' can also be constructed by the following trick: A representation on $S$ can
be considered as a morphism of $S$-Lie algebras of the 
given $S$-Lie algebra into the $\mathscr{C}_{1+}$-algebra $F(\widetilde{S},\widetilde{S})$. But the identity
inclusion
\beg{edetpower}{S\r F(\widetilde{S},\widetilde{S})}
is an equivalence, and hence $F(\widetilde{S},\widetilde{S})$ is, in the derived
category, canonically $q^*$ of the $E_\infty$-algebra $S$. Now for any $m$ prime to $p$, we can
compose \rref{edetpower} with
the character 
$$m:C_{\infty+}\widetilde{S}\r S$$
to create the ``$m$'th determinant power representation" which we will denote by 
$Det^{\wedge m}$. Placed into the appropriate degree, it is a special graded representation.

\vspace{3mm}
\subsection{Some useful pairs of adjoint functors}
\label{sadj}
Consider a transitive relations $T\subseteq T^\prime$ on $\{1,\dots,n\}$. Then we have an ``inclusion" morphism of graded Lie algebras
\beg{eadj1}{\kappa_{g_T,g_{T^\prime}}=\kappa_{T,T^\prime}:g_T\r g_{T^\prime}
}
Of course, we have the pullback functor $\kappa_{T,T^\prime}^{*}$ from graded $g_{T^\prime}$-representa-tions
to graded $g_T$-representations. This functor has a left adjoint $(\kappa_{T,T^\prime})_\sharp$
and a right adjoint $(\kappa_{T,T^\prime})_*$ (left and right Kan extension).
Certain compositions of these functors will be of major use to us. 
Specifically, we will often consider the case when 
\beg{eadj2}{g_T=gl_{k_1}\times\dots\times gl_{k_m}, \; g=g_{T^\prime}=gl_n
}
with $k_1+\dots+k_m=n$. Let $p_+$ resp. $p_-$ be the parabolic Lie subalgebra in $gl_n$ with reductive Levi
factor 
$g_T$ with respect to the positive (resp. negative) roots. 

\vspace{3mm}
\subsection{co-Verma modules}\label{sscoverma}
It turns out that to capture the analogue of derived Zuckermann functors, we will actually have to go even further
and investigate algebraic groups over $S$. On representations of algebraic groups, however, one naturally has
{\em induction}, which is an analogue of the functors $\kappa_*$. The functors $\kappa_\sharp$ do not in general have
analogues. Because of this, rather than Verma modules, it will be easier for us to work with generalized
{\em co-Verma modules}, which are modules of the form
$$(\kappa_{p_+,g})_*\pi^*W$$
where $\pi$ is the projection from the parabolic to its Levi factor, and $W$ is a representation of the Levi factor. 

Note: in classical representation theory, it is customary to swap $p_+$ for $p_-$ so that the Verma and co-Verma
modules are in the same parabolic BGG category
$\mathcal{O}_{\mathbf{p}}$. We do not bother with this convention here. A part of the reason is that
even in the graded sense, the graded pieces of our Verma modules are not truly finite spectra (since they
will be extended products), and therefore dualization is not as nicely behaved as one may hope. 
When we work in the category completed at a large prime $p$, however, in weights whose $\rho$-shifted 
coordinates are non-negative integers much smaller than $p$, graded morphisms behave well in the given range.

\vspace{3mm}
\section{Commutative Hopf algebras and Harish-Chandra pairs over $S$} \label{salgghc}

In this section, we will define the $S$-module version of Harish-Chandra pairs and their representations.
Currently, we have no construction of Lie groups in our theory. To remedy this, we use the notion of 
\emph{commutative Hopf algebras} instead, which can be thought as a weak version of algebraic groups. 

\vspace{3mm}
\subsection{Commutative Hopf algebras over $S$}\label{ssalgg}
A {\em 
commutative Hopf algebra $R$ over $S$} 
is a $\mathscr{C}_{1+}$-coalgebra in the category of commutative $S$-algebras, i.e. explicitly,
a choice of an operad $\mathscr{A}$ in $\mathscr{B}$ equivalent to $\mathscr{C}_{1+}$, and
structure maps in the category of commutative $S$-algebras
$$\epsilon:\mathscr{A}(0)\widetilde{\otimes} R\r S,$$
$$\psi:\mathscr{A}(n)\widetilde{\otimes}
 R\r \underbrace{R\wedge\dots\wedge R}_{\text{$n$ times}}$$
where $\widetilde{\otimes}$ is the based Kelly tensor product \cite{kelly} (which can be always
taken between a based simplicial set and an object of a symmetric monoidal category with simplicial
realization preserving the symmetric monoidal structure)
which satisfy the usual operad coalgebra relations. A {\em comodule algebra} 
(more generally, a $\mathscr{C}$-comodule algebra for a $\mathscr{B}$-operad
$\mathscr{C}$) over a commutative
$S$-Hopf algebra $R$ is a commutative $R$-algebra (resp. $\mathscr{C}$-algebra) 
$A$ together with structure morphisms 
\beg{ecomodule}{\theta:\mathscr{A}(n)^R\widetilde{\otimes} A\r A\wedge R\wedge\dots\wedge R,}
(where $\mathscr{A}(n)^R$ is as in Section \ref{ssdlr}),
compatible with $\psi$ and $\epsilon$ in the obvious sense. Obviously, $R$ is always
a comodule algebra over itself.
We will also be interested in
comodules in the category of spectra, which are spectra whose
structure is defined by the same formula \rref{ecomodule}, where $\widetilde{\otimes}
$ now denotes the Kelly product
in the category of spectra (which is, essentially, the smash product). For a fixed commutative
$S$-Hopf algebra $R$, denote by $R\dash Comod$ the category of $R$-comodules.

The smash product creates a commutative associative unital product in the category of 
$R$-comodules and also in the category of $R$-comodule algebras by
$$\diagram
\mathscr{A}(n)^R\widetilde{\otimes} V\wedge W\dto^{\Delta\wedge Id}\\
(\mathscr{A}(n)^R\wedge \mathscr{A}(n)^R)
\widetilde{\otimes} V\wedge W\dto^T\\
(\mathscr{A}(n)^R\widetilde{\otimes} V)\wedge (\mathscr{A}(n)^R\widetilde{\otimes} W)\dto^{
\theta\wedge\theta}\\
V\wedge R\wedge\dots\wedge R\wedge W\wedge R\wedge\dots\wedge R\dto^T\\
V\wedge W\wedge (R\wedge R)\wedge\dots\wedge (R\wedge R)\dto^{Id\wedge\phi\wedge\dots
\wedge \phi}\\
V\wedge W\wedge R\wedge\dots\wedge R
\enddiagram
$$
where $T$ denotes switches of factors and $\phi$ is the product in the category of commutative 
$S$-algebras. 
Using the Brown's representability theorem \cite{brown}, we find that in the
category of $R$-comodules, the functor $?\wedge V$ has
a right adjoint $\mathscr{F}_R(V,?)$. However, it is important to note that this right adjoint may not have the expected
properties, in particular the underlying spectrum may not be $F(V,?)$. This is because we do not have
conjugation as a part of our definition of commutative Hopf algebra (the reason of which, in turn, is that
we do not know how to construct examples with rigid conjugation - we will return to this point below). In the 
absence of conjugation, we do not expect the function object to behave as expected. To give a very simple
analogy, consider the category of complex representations of the commutative monoid $\N_0$ (think of it
multiplicatively, writing the generator as $t$). Then the unit of the tensor product is the ``trivial'' representation
$\C_1$ where $t$ acts by $1$. Now consider the representation $\C_0$ where $t$ acts by $0$. We actually
have $\mathscr{F}(\C_0,\C_1)=0$, because tensoring any representation with $\C_0$ makes $t$ act by $0$.

\vspace{3mm}
Note that $\mathscr{F}_R(V,V)$ is canonically a $\mathscr{C}_{1+}$-comodule algebra in the category
of $R$-comodules, while $F(V,V)$ is a $\mathscr{C}_{1+}$-algebra in $\mathscr{S}$. If
$\mathscr{U}$ is the forgetful functor from the category of $R$-comodules to $\mathscr{S}$,  there further is a 
canonical morphism of $\mathscr{C}_{1+}$-algebra
\beg{eforgett+}{\mathscr{U}\mathscr{F}_R(V,V)\r F(V,V),
}
but it is not an equivalence for general $R$.

\vspace{3mm}
By a {\em morphism of commutative $S$-Hopf algebras}, we shall mean
a morphism of $\mathscr{C}_{1+}$-coalgebras in the $\mathscr{C}_{\infty +}$-algebras. Similarly,
we define morphisms of comodules and comodule algebra by requiring compatibility with the
structure morphisms $\theta$. 

\vspace{3mm}
Our typical example of a commutative Hopf algebra is, for a transitive relation $T$ on $\{1,\dots,n\}$, a commutative
$S$-Hopf algebra
\beg{egtalg}{\mathscr{O}_{G_T}=det^{-1}C_\infty (g_T^\vee).
}
Here we use the notation $g^\vee=F(g,S)$, we suppress cofibrant replacement from the
notation, and $det^{-1}$ means
inverting the determinant $0$-homotopy class. Both cofibrant replacement and inverting a homotopy class
can be done in a way which does not spoil the $\mathscr{C}_{1+}$-coalgebra structure, using the methods
of \cite{ekmm}. This structure is
given by the diagram
$$\diagram
g_T^\vee\rto\dto_\eta & g_T^\vee\wedge g_T^\vee\dto^{\eta\wedge \eta}\\
C_\infty(g_T^\vee)\rdotted|>\tip &C_\infty(g_T^\vee)\wedge C_\infty(g_T^\vee),
\enddiagram$$
which follows from adjunction. 

We should mention, of course, that we have a canonical morphism $h\r (h^\vee)^\vee$ which, however, is not in
general an isomorphism in $\mathscr{S}$, especially when cofibrant objects are concerned. Similarly, while the
dual of an operad coalgebra (resp. comodule) is in general an algebra (resp. module), it is not true in general that
the dual of an algebra (resp. module) is a coalgebra (resp. comodule) over the same operad. This is one of the 
reasons we included the operad $\mathscr{A}$ in the definition: using standard rectification methods, we
can however change $\mathscr{A}$ (while preserving its homotopy type) in such a way that $g_T^\vee$ is
a coalgebra, and $g_T\r (g_T^\vee)^\vee$ is an equivalence of $\mathscr{A}$-coalgebras.

From the structure, 
$$V=\bigvee_{\text{$n$ copies}} S$$
is an $\mathscr{O}_{GL_nS}$-comodule, which we call the {\em standard representation}. Also, $S$ is 
canonically an $\mathscr{O}_{GL_nS}$-comodule, which we will call the {\em trivial representation}. As before, 
the smash product of comodules can be used to construct other comodules. We shall be especially interested
in the comodules
\beg{eoglext}{\Lambda^qV=\Sigma^{-q}E\Sigma_{q+}\wedge_{\Sigma_q}V^{\wedge q}.
}
In this paper, we will be interested in completing at a large prime $p$, by which we mean Bousfield localizing 
at the Moore spectrum $M\Z/p$ (\cite{ekmm}, Section VIII.1). We denote by $(\mathcal{O}_{GL_nS}\dash Comod)_p$
the full subcategory of $\mathcal{O}_{GL_nS}\dash Comod$ on $M\Z/p$-local
comodules, and by
\beg{emploc}{X\r X_p
}
the corresponding localization map of $(\mathcal{O}_{GL_nS}\dash Comod)_p$-comodules.
One sees easily that the underlying morphism of spectra of \rref{emploc} is also $M\Z/p$-localization.
Recall that the category of $M\Z/p$-local spectra has a smash product obtained by applying the smash product
and then $M\Z/p$-localizing. It is commutative, associative and unital up to homotopy, the unit being $S_p$.
We do not know how to rigidify this product, however. There is also a right adjoint, which is simply $F(?,?)$
restricted to $M\Z/p$-local spectra.
The category $(\mathcal{O}_{GL_nS}\dash Comod)_p$ has a smash product obtained by applying the ordinary
smash product and then $M\Z/p$-localizing. We will denote both this smash product and
the underlying smash product of $M\Z/p$-local spectra by $\wedge_p$.
The unit is the $M\Z/p$-localized trivial representation.

We do not know how to construct Bousfield localization
for commutative $S$-Hopf algebras. In fact, the reader should consult \cite{ekmm}, Chapter VIII, to see that even for
commutative algebras, Bousfield localization is not entirely what we expect. For example, one cannot 
$p$-complete the unit. 

This is the reason why we need results such as the following.

\vspace{3mm}
\begin{lemma}\label{lpcomplf}
Let $p>>n$ be a prime. Then $V_p$ is strongly dualizable
in the category of $M\Z/p$-local $\mathscr{O}_{GL_nS}$-comodules, and we have equivalences
of spectra
\beg{epcomplf1}{\diagram
\mathscr{U}\mathscr{F}_{\mathscr{O}_{GL_nS}}(V_p,S_p)\rto^(.6)\sim & F(V_p,S_p)
\enddiagram
}
\beg{epcomplf2}{\diagram
\mathscr{U}\mathscr{F}_{\mathscr{O}_{GL_nS}}(V_p,V_p)\rto^(.6)\sim & F(V_p,V_p).
\enddiagram
}
\end{lemma}

\begin{proof}
We first prove that
\beg{edetgln}{(\Lambda^n V)_p
}
is invertible in the derived category of $(\mathcal{O}_{GL_nS}\dash Comod)_p$.
In fact, it is true in general that if an ($M\Z/p$-local) $R$-comodule $L$ forgets
to $S$ (resp. $S_p$) and specifies a (homotopy) 
invertible class in $R$ (resp. $R_p$), then $L$ is invertible in the derived category
of $R$-comodules.
To see this, let $L^{-1}$ be the inverse of $L$ in the category of ($M\Z/p$-local) spectra.
We recall from \cite{lada}, Chapter V that 
an $R$-comodule structure on $M$ can be specified by a ``structure map''
\beg{ecomodstr}{M\r M\wedge R
}
and the vanishing of a series of obstructions
\beg{ecomodobstr}{S^k\wedge M\r M\wedge \underbrace{R\wedge\dots\wedge R}_{\text{$k+1$ times}}.
}
In the case of $L^{-1}$, we specify \rref{ecomodstr} as the $R$-inverse homotopy class to the
homotopy class associated with the comodule structure \rref{ecomodstr} for $L$. The homotopies
making the obstructions
\rref{ecomodobstr} vanish are then computed from smashing (point-wise) with the corresponding homotopies
for $L$, and noting that $S$ (resp. $S_p$) is also a comodule using the unit of $R$.

We shall now prove \rref{epcomplf1}; \rref{epcomplf2} is proved
analogously. We have a morphism of $\mathcal{O}_{GL_nS}$-comodules
$$V\wedge \Lambda^{n-1}V\r \Lambda^n V,$$
and hence
\beg{elpcomp1}{V\wedge \Lambda^{n-1}V\r (\Lambda^n V)_p.
}
Thus, we have a morphism of $\mathcal{O}_{GL_nS}$-comodules
\beg{elpunit}{V_p\wedge_p (\Lambda^{n-1}V)_p\wedge_p ((\Lambda^n V)_p)^{-1}\r S_p.
}
Similarly, transfer gives a morphism of $\mathcal{O}_{GL_nS}$-comodules
\beg{elpcomp2}{\Lambda^n V\r (V\wedge \Lambda^{n-1}V)_p,
}
which gives a morphism of $\mathcal{O}_{GL_nS}$-comodules
\beg{elpcounit}{S_p\r V_p\wedge_p (\Lambda^{n-1}V)_p\wedge_p ((\Lambda^n V)_p)^{-1}.
} 
One further sees that \rref{elpunit} and \rref{elpcounit} forget to the unit and counit of strong 
duality in $M\Z/p$-local spectra, and hence define a unit and counit of strong duality in 
the derived category of
$(\mathcal{O}_{GL_nS}\dash Comod)_p$ up to equivalence. Strong duality and \rref{epcomplf1}
follow.
\end{proof}

\vspace{3mm}
Let us make a few more remarks on the theory of commutative $S$-Hopf algebras. The next lemma connects 
commutative $S$-Hopf algebras to $S$-Lie algebras, 
analogous to the classical correspondance between Lie groups and Lie algebras. 

\vspace{3mm}
\begin{lemma}\label{lagstab}
Let $R$ be a commutative Hopf algebra 
over $S$. Then there is a canonical structure of a $(\mathscr{C}_{k+1})_+$-coalgebra
on the $k$'th bar construction $B^k(R)$, and a canonical structure of a $\mathscr{L}$-coalgebra where
$\mathscr{L}$ is the Lie operad) on $QR$. Here, $Q$ denotes topological Quillen homology. 
\end{lemma}

\begin{proof}
First of all, recall that topological Quillen homology can be defined by Dwyer-Kan stabilization:
$$QR=\operatornamewithlimits{holim}_\r \Sigma^{-k} B^k (R)$$
where by $\Sigma^{-k}$ we mean the $k$'th desuspension on the augmentation ideal (right adjoint
to the bar construction). Therefore, if we show that $B^k(R)$ has a canonical structure of
a $(\mathscr{C}_{k+1})_+$-coalgebra, $\Sigma^{-k}B^k R$ has a canonical structure of
a $\Sigma^{-k}(\mathscr{C}_{k+1})_+$-coalgebra, we know that $QR$ has a structure of
a $\mathscr{L}$-coalgebra by \rref{esslie1}.

The proof that $B^k(R)$ has a canonical structure of a $(\mathscr{C}_{k+1})_+$-coalgebra is analogous to
\cite{hkv}.
\end{proof}

\vspace{3mm}
We see that if $V$ is an $R$-comodule, then $V$ is also a right 
comodule over the co-Lie algebra $QV$. 

The next lemma shows that in
 particular, the commutative $S$-Hopf algebra $GL_nS$ corresponds to the $S$-Lie algebra $gl_n$. The 
same holds for the subalgebras of $gl_n$ that we have studied. 

\vspace{3mm}
\begin{lemma}\label{lquilgln}
One has
$$Q\mathscr{O}_{G_T}\sim g_T^\vee$$
as co-Lie algebras.
\end{lemma}

\begin{proof}
For a $\mathscr{C}_{1+}$-coalgebra $X$, one has
$$QC_{\infty}X\sim X$$
as an $S$-Lie coalgebra (recall the morphism of operads $\mathscr{L}\r \mathscr{C}_{1+}$). Inverting 
the determinant class does not affect the Quillen homology. (It disappears after one bar construction.)
\end{proof}

\vspace{3mm}

\vspace{3mm}
\subsection{Harish-Chandra pairs and their representations}\label{sshcc}

Harish-Chandra pairs appear in many areas of representation theory (for example algebraic, compact Lie, affine, super 
groups and algebras),
see \cite{knappv}. 
The basic idea is obvious - we want
a notion of a group representation which is simultaneously, and compatibly, a representation of a Lie algebra.
One must be mindful, however, of subtle details of the definition which change depending on the context.

In this paper, by a {\em pre-Harish-Chandra pair} $(R, g)$,  we mean simply a commutative $S$-Hopf algebra $R$,
and an $R$-equivariant $S$-Lie algebra $g$. More generally, for an operad $\mathscr{C}$ in $\mathscr{S}$,
and a commutative $S$-Hopf algebra $R$, an {\em $R$-equivariant $\mathscr{C}$-algebra $X$} 
is defined as an object of $\mathscr{S}$ which has both the structure of an $R$-comodule and
a $\mathscr{C}$-algebra, such that the following diagram commutes for $n\geq 2$:
\beg{eequivopalg}{\resizebox{0.95\hsize}{!}{%
$\diagram
\mathscr{A}(m+1)^R\wedge\mathscr{C}(n)\wedge X\wedge\dots\wedge X\rrto 
\dto^{T\circ(\Delta\wedge Id)}&&
\mathscr{A}(m+1)^R\wedge X\xto[3,0]\\
\mathscr{C}(n)\wedge\mathscr{A}(m+1)^R\wedge X\wedge \dots \wedge \mathscr{A}(m+1)^R\wedge X\dto &&\\
\mathscr{C}(n)\wedge X\wedge \underbrace{R\wedge\dots\wedge R}_{m}\wedge\dots
\wedge X\wedge \underbrace{R\wedge\dots\wedge R}_{m}\dto^{(Id\wedge \phi)} &&\\
\mathscr{C}(n)\wedge X\wedge\dots\wedge X\wedge R\wedge\dots\wedge R\rrto &&
X\wedge \underbrace{R\wedge\dots\wedge R}_{m}
\enddiagram$}%
}
(see Subsection \ref{ssdlr} for the meaning of $\mathscr{A}(n)^R$).
Here for simplicity, we denote all shuffles by $T$, all diagonals by $\Delta$ and all products by $\phi$.

A {\em morphism} $f$ from an $R_1$-equivariant $\mathscr{C}$-algebra $X_1$ to an
$R_2$-equi-variant $\mathscr{C}$-algebra $X_2$ is defined as a morphism of commutative $S$-Hopf
algebras
$$f_R:R_2\r R_1$$
and a morphism of $\mathscr{C}$-algebras 
$$f_X:X_1\r X_2$$
which satisfy the obvious commutative diagram. In particular, this defines morphisms of pre-Harish-Chandra
pairs. Notice that we put the contravariance 
into the commutative Hopf algebra coordinate. This is because we think of the Hopf algebras as 
coordinate rings of algebraic groups, in which morphisms would be ordinarily written contravariantly.

\vspace{3mm}
If $R$ is a commutative Hopf algebra over $S$ and $\mathscr{C}$ is an operad in $\mathscr{S}$, 
and $X$ is an $R$-equivariant $\mathscr{C}$-algebra, then we can define an $R$-equivariant
$(\mathscr{C},X)$-module $Y$ as an $R$-comodule which is also a $(\mathscr{C},X)$-module,
and the obvious analogue of diagram \rref{eequivopalg} where we replace in each entry,
one $X$ by $Y$, commutes. Morphisms are again defined in the obvious way. 

\vspace{3mm}



Let $R$ be a commutative $S$-Hopf algebra. Then the stabilization (see Lemma \ref{lagstab})
\beg{erqr}{R\r QR
}
is a morphism of $S$-Lie coalgebras. In our examples, $QR$ is generally strongly dualizable, but
our difficulty is that we do not know how to construct an $R$-equivariant structure on $QR^\vee$.
Therefore, we introduce another piece of data, namely morphisms of $S$-Lie algebras
\beg{erqr++}{QR^\vee\leftarrow q\r \rho
}
where $\rho$ is given a structure of an $R$-equivariant $S$-Lie algebra. (Of course, we think of the morphisms
of \rref{erqr++} be ``close to equivalences" in an appropriate sense, but it is not suitable to make this a part 
of the formal structure. The basic example, which we will need to generalize, is $R=\mathscr{O}_{GL_nS}$,
$q=gl_nS$, $\rho=\mathscr{F}_{\mathscr{O}_{GL_nS}}(V_p,V_p)$ of Lemma \ref{lpcomplf}.)

Given a commutative $S$-Hopf algebra $R$ and morphisms of $S$-Lie algebras \rref{erqr++} where 
$\rho$ is given an additional structure of an $R$-equivariant $S$-Lie algebra, a 
{\em Harish-Chandra pair} $(R,g,\gamma,\rho, q)$ (briefly $(R,g)$) relative to the data \rref{erqr++}
is a pre-Harish-Chandra pair $(R,g)$ together with a morphism of $R$-equivariant $S$-Lie algebras
$$\gamma:\rho\r g.$$
Note that for the given data \rref{erqr++}, there is a ``universal" example of a Harish-Chandra pair, namely
$(R,\rho,Id,\rho, q)$.

A morphism of Harish-Chandra pairs
$$(R_1,g_1,\gamma_1,\rho_1,q_1)\r (R_2,g_2,\gamma_2,\rho_2,q_2)$$
consists of a morphism of commutative $S$-Hopf algebras
$$\diagram
R_2\rto^{f_R} & R_1,
\enddiagram
$$
a morphism of $S$-Lie algebras
$$\diagram
q_1\rto^{f_q} & q_2,
\enddiagram
$$
a morphism of $R$-equivariant $S$-Lie algebras
$$\diagram
\rho_1\rto^{f_\rho}&\rho_2
\enddiagram
$$
and a diagram of $S$-Lie coalgebras
$$\diagram
QR_1^\vee\rto^{Qf_R^\vee} & QR_2^\vee\\
q_1\uto\rto^{f_q}\dto & q_2\dto\uto\\
\rho_1\rto^{f_\rho} & \rho_2
\enddiagram
$$
and a diagram of $S$-Lie algebras
$$\diagram
\rho_1\rto^{f_\rho}\dto &\rho_2\dto\\
g_1\rto^{f_g}&g_2
\enddiagram
$$
equivariant with respect to the diagram of commutative $S$-Hopf algebras
$$\diagram
R_1 & R_2\lto_{f_R}\\
R_1\uto^{Id} &R_2\uto_{Id}\lto_{f_R}.
\enddiagram
$$
A {\em representation} of (or {\em module} over) a Harish-Chandra pair $(R,g,\gamma,\rho,q)$
is an $R$-equivariant $(\mathscr{L},g)$-module whose underlying $(\mathscr{L},q)$-module
structure coincides with the $(\mathscr{L},q)$-module structure coming from the 
$(\mathscr{L},QR)$-comodule structure arising from $R$-equivariance. (Recall again that
comodules over operad coalgebras dualize to modules over operad algebras, but not vice versa.)

Modules over a Harish-Chandra pair form a full subcategory
of the category of modules over the underlying pre-Harish-Chandra pair.

\vspace{3mm}
For any commutative $S$-Hopf algebra $R$ and any $R$-comodule $W$, $\mathscr{F}_R(W,W)$ is, by adjunction, 
a $R$-equivariant associative algebra (and hence $R$-equivariant $S$-Lie algebra). 
We will be especially interested in the case 
when $R=\mathscr{O}_{GL_nS}$, and $W=V$ is the standard representation. Then by Lemma 
\ref{lpcomplf}, at least up to the eyes of $M\Z/p$, we can think of $\mathscr{F}_{\mathcal{O}_{GL_nS}}(V_p,V_p)$ as
an $\mathscr{O}_{GL_nS}$-model of $gl_nS$, which we already know is its Lie algebra by Quillen 
cohomology (cf. Lemma \ref{lquilgln}).

For our purposes, though, we will need to consider a somewhat more general example. Concretely, 
we will have 
\beg{eordpart1}{n=k_1+\dots+k_m,}
(which we will refer to as an {\em ordered partition $\mathbf{k}=(k_1,\dots,k_m)$ of $n$})
and
$$R=R_\mathbf{k}=\mathscr{O}_{(GL_{k_1} \times\dots\times GL_{k_m})(S)}.$$
We will also have some
$$i_1+\dots +i_s=m,$$
and
\beg{eordpart2}{l_j=k_{i_1+\dots +i_{j-1}+1}+\dots+k_{i_1+\dots +i_j}.}
(We will refer to the ordered partition \rref{eordpart1} as
a {\em refinement} of the ordered partition $\mathbf{l}=(l_1,\dots, l_s)$ of
\rref{eordpart2}.)
Let $W$ be the standard $GL_nS$-representation. Then as an $R$-comodule, $W$ has an increasing filtration
$\Phi$ where $\Phi^jW$ is the standard $GL_{\ell_1+\dots +\ell_j}S$-representation. We may define
a category of filtered $R$-comodules in the obvious way, and obtain the $R$-equivariant associative algebra
(hence, $S$-Lie algebra)
\beg{eparabphi}{\mathbf{p}=\mathbf{p}_\mathbf{l}=\mathscr{F}^\Phi_R(W_p,W_p)}
where the superscript means the function object in the filtered $R$-comodule category. The notation $\mathbf{p}$ 
stands for {\em parabolic}, as these are examples of parabolic Lie subalgebras of $gl_nS$.

\vspace{3mm}
\begin{lemma}\label{llparab}
With the above notation, we have an equivalence of 
$R$-equivariant associative algebras
\beg{ellparab}{\mathscr{F}^\Phi_R(W_p,W_p)\r \mathscr{F}^\Phi(W_p,W_p)
}
where the right hand side denotes the analogous construction in the filtered category of spectra.
\end{lemma}

\begin{proof}
One sees that, by definition, in the category of $R$-comodules,
$$\mathscr{F}^\Phi_R(W_p,W_p)\sim \bigvee_{j_1\leq j_2} \mathscr{F}_R((W_{k_{j_1}})_p, (W_{k_{j_2}})_p)
$$
where $W_{k_j}$ are the pushforwards of the basic comodules over $\mathcal{O}_{GL_{k_j}S}$ to $R$.
\end{proof}

In fact, all the morphisms of Harish-Chandra pairs we will consider will be of the form
\beg{ehskappa}{\kappa:(R_\mathbf{k_1},\mathbf{p}_\mathbf{l_1})\r(R_\mathbf{k_2},\mathbf{p}_\mathbf{l_2})}
where
$\mathbf{k_1}$ is a refinement of $\mathbf{k_2}$ and $\mathbf{l_1}$ is a refinement of $\mathbf{l_2}$.

A variant of this construction is if we replace the filtration $\Phi$ by {\em grading}, which means taking
\beg{eee79}{\ell=\ell_{\mathbf{l}}=\mathscr{F}_R(V,V)}
where $V$ is the standard representation of $R_{\mathbf{l}}$ (product of the standard representations
of the $GL_{\ell_j}S$'s). We will also consider the morphism
\beg{ehspi}{\pi:(R_{\mathbf{k_1}},\mathbf{p}_\mathbf{l})\r (R_{\mathbf{k_1}},\ell_\mathbf{l}).}
We refer to $\ell_\mathbf{l}$ as the {\em Levi factor} of the corresponding parabolic $\mathbf{p}_\mathbf{l}$.

\vspace{3mm}
As already mentioned, our basic example of a Harish-Chandra pair is 
$$R=\mathscr{O}_{(GL_{s_1}\times\dots\times GL_{s_m})(S)},\;
q=(gl_{s_1}\times\dots\times gl_{s_m})(S),\;
\rho=\ell$$
(see \rref{eee79}). Our basic examples of $(R,g)$-modules are of the form
$$(\Lambda^{t_1}V_{(1)}\wedge\dots\wedge \Lambda^{t_m}V_{(m)})_p$$
where $V_{(i)}$ is the standard representation of $\mathscr{O}_{GL_{s_i}S}$.

\vspace{3mm}
There are forgetful functors from 
$$U_{R,g}:(R,g)\dash Mod\r R\dash Comod,$$ 
$$U_R:R\dash Comod\r\mathscr{S},$$
$$V_{R,g}:(R,g)\dash Mod\r g\dash Mod,$$
$$V_g:g\dash Mod\r \mathscr{S}.$$
All of these functors have right adjoints, and all are comonadic. 
The comonad $\mathscr{H}=\mathscr{H}_{(R,g)}$ corresponding
to the composition 
$$U=U_R\circ U_{R,g} =V_g\circ V_{R,g}$$
is called the {\em Hecke comonad} corresponding to the Harish-Chandra pair $(R,g)$.

For a morphism of Harish-Chandra pairs 
\beg{efff}{f:(R_1,g_1)\r (R_2,g_2)}
we have a forgetful functor (i.e. pullback)
\beg{ehsres10}{f^*:(R_2,g_2)\dash Mod\r (R_1,g_1)\dash Mod.
}

\begin{proposition}\label{pindr}
The functor \rref{ehsres10} preserves equivalences, and there is a functor
\beg{eindfuntor}{f_*:(R_1,g_1)\dash Mod\r (R_2,g_2)\dash Mod}
which is its right adjoint on the level of derived categories. Additionally, if $R_1=R_2$,
we have a commutative diagram
up to equivalence:
\beg{eindfunctor1}{\diagram
(R,g_1)\dash Mod\rto^{f_*}\dto^{V_{(R,g_1)}} &
(R,g_2)\dash Mod\dto^{V_{(R,g_2)}}\\
g_1\dash Mod\rto^{\phi_*} & g_2\dash Mod
\enddiagram 
}
where $\phi:g_1\r g_2$ is the underlying morphism of $S$-Lie algebras.
\end{proposition}

\begin{proof}
The functor $f_*$ is constructed by the $2$-sided cobar construction of Hecke comonads
$$f_*(?)=Cobar(?,\mathscr{H}_{(R_1,g_1)},\mathscr{H}_{(R_2,g_2)}).$$
The commutativity of diagram \rref{eindfunctor1} is essentially due to the fact that 
$U_R$ takes the smash of $R$-comodules to the smash product of spectra.
\end{proof}

\vspace{3mm}
Suppose we are given an $(R_{\mathbf{k}},\ell)$-module $W$ where $\ell$ is the Levi factor of a parabolic 
$\mathbf{p}_\mathbf{l}$.
(Note that the standard representation of $\ell$, and extended powers of its suspensions, are examples
of such representations $W$.)
Suppose the ordered partition $\mathbf{l}$ is a refinement of another ordered partition $\mathbf{m}$. Then
\beg{egenverma}{V_{\mathbf{l},\mathbf{m},W}=\kappa_*\pi^*W
}
where $\kappa$ is the morphism of Harish-Chandra pairs \rref{ehskappa} with $\mathbf{l_1}=\mathbf{l}$,
$\mathbf{l_2}=\mathbf{m}$, $\mathbf{k_1}=\mathbf{k_2}=\mathbf{k}$, and $\pi$ is the morphism
of Harish-Chandra pairs \rref{ehspi} where $\mathbf{k_1}=\mathbf{k}$. We refer to the representation
\rref{egenverma} of the Harish-Chandra pair $(R_{\mathbf{k}},\mathbf{p}_{\mathbf{m}})$ as a
{\em generalized co-Verma module}.

\vspace{3mm}
\noindent
{\bf Example:} Let us consider the case $n=2$, $R_1=\mathscr{O}_H$, $R_2=\mathscr{O}_{GL_2}$, 
$\mathbf{k}=(1,1)$, $\mathbf{l}=(2)$, $H$ stands for diagonal matrices. Consider the morphism
$$\kappa:(R_1,\mathbf{p}_\mathbf{k})\r (R_2,\mathbf{p}_\mathbf{l}).$$
Let further
$V^\lambda$ be the $(R_1,\mathbf{p}_\mathbf{l})$-co-Verma module on a $\mathscr{O}_H$-character $S^\lambda$
where $\lambda$ is a dominant integral weight.
We wish to compute 
\beg{eindv}{\kappa_*(V^\lambda).
}
As a warm-up, let us consider the case of algeraic groups over $\C$. We can, of course, interpret this
as a case of our setup, replacing $\C$ by the commutative (Eilenberg-MacLane)
$S$-algebra $H\C$. However, let us first work 
completely classically, i.e. with commutative $\C$-algebras in the classical sense. Then the morphism 
of Hecke 
comonads is interpreted as a map of commutative graded Hopf algebras
\beg{ehopfmap1}{\mathscr{O}_{GL_2}\r\mathscr{O}_{B}\otimes (Un_-)^\vee.
}
(As elsewhere in the paper, we set $B=B_+$, the subgroup of lower triangular matrices.)

Thus, we must study the Bruhat decomposition
\beg{ebruh}{B\times N_-\r GL_2.
}
On matrices, this is
$$
\left(\begin{array}{cc}
a_{11} & a_{12}\\
a_{21} & a_{22}
\end{array}
\right)=
\left(\begin{array}{cc}
b_{11} & 0\\
b_{21} & b_{22}
\end{array}
\right)\cdot
\left(\begin{array}{cc}
1 & x\\
0 & 1
\end{array}
\right)=
\left(\begin{array}{cc}
b_{11} & xb_{11}\\
b_{21} & xb_{21}+b_{22}
\end{array}
\right).
$$
Thus, the $\mathscr{O}_?$ of \rref{ebruh} can be written as
\beg{ehopfmap2}{
\begin{array}{lll}
a_{11} & \mapsto  & b_{11}\\
a_{12} & \mapsto & xb_{11}\\
a_{21} & \mapsto & b_{21}\\
a_{22} &\mapsto & xb_{21} + b_{22}.
\end{array}
}
Now in \rref{ehopfmap1}, $Un_-$ should be thought of as a divided power algebra on generators 
\beg{ehopfmap3}{\gamma_n =\frac{x^n}{n!}}
(although, of course, over $\C$, this is just $\C[x]$), so we can write \rref{ehopfmap1} as
\beg{ehopfmap4}{det^{-1}\C[a_{11},a_{12},a_{21},a_{22}]
\r (b_{11}b_{22})^{-1}\C[b_{11},b_{21},b_{22}]\otimes \C\{\gamma_{0},\gamma_{1},\gamma_2,\dots\}
}
given by \rref{ehopfmap2}. We see in any case from \rref{ehopfmap2} that in the sub-$\C$-module of the
target generated by monomials not containing $b_{21}$ (which corresponds to co-Verma modules),
the possible powers of $x$ occurring in a monomial which has $b_{11}^\ell$ are 
$$x^{0},\dots,x^{\ell}.$$
This corresponds to the fact that $\kappa_*(V^\lambda)$ is the irreducible $GL_2$-repre-sentation of
weight $\lambda$.

To interpret this example over $S$, we note that basically one can argue the same way.
One difficulty is that
$$Un_-^\vee=F(C_\infty S, S)$$
which has been calculated by Carlsson's theorem \cite{carlsson}, is in general more complicated than over
$\C$. The other difficulty is the denominator in \rref{ehopfmap3}, which will cause disruptions at large weights. 

Because of this, in the present paper, we restrict to a `large prime' setting. Of course, the `small prime' case is 
in principle more interesting, and we will return to it in future work. In the present example,
we see that if we work in the $p$-completed category of spectra
(localized at $M\Z/p$) and that the $\rho$-shifted components of the weight $\lambda$ of our 
co-Verma module are non-negative integers $<<p$, then the argument goes through and 
we see that $\kappa_*(V^\lambda)$ is a finite $GL_2$-representation of dimension $\ell+1$ where the
difference between the first and second coordinate of $\lambda$ is $\ell$.

\vspace{5mm}
\subsection{Localization}\label{ssapprox}
When dealing with Harish-Chandra pairs of the form $(R_\mathbf{k},\mathbf{p}_\mathbf{l})$ and morphisms
of the form \rref{ehskappa} (resp. \rref{ehspi}), we have a variant of all the constructions described
so far in this section where everything is graded by $gl_n$-weights. From now to the end of the present paper,
we will work in this graded context. Additionally, we will work in the categories of 
$(R_\mathbf{k},\mathbf{p}_\mathbf{l})$-modules which have a fnite set of highest weights 
(generalized Verma co-modules
of the form \rref{egenverma} are an example, provided that $W$ has a highest weight). Recall that 
classically in representation theory, a {\em highest weight
module} is a module generated by a vector $v$ which is annihilated by all positive root spaces in $g$. We have a 
weaker condition in mind here: we simply require that there be a finite set of weights $\{\lambda_1,\dots,\lambda_n\}$
such that the module is concentrated in weights $\lambda_i+\mu$ where $\mu$ runs through linear combinations 
with coefficients in $\N_0$ of
negative roots.

Additionally
still, we will Bousfield-localize the full subcategory of $(R_\mathbf{k},\mathbf{p}_\mathbf{l})$-modules with highest weight
at the Moore spectrum $M\Z/p$ (see \cite{bousfield} and \cite{ekmm}, Chapter VIII, for the Bousfield localization functor). 
We will also always assume $k_i<<p$. We will denote this category by
$$(R_\mathbf{k},\mathbf{p}_\mathbf{l})\dash hwMod.$$
But sometimes we wish to restrict attention even further, to full subcategories of objects which arise 
by taking fibrations and homotopy limits of co-Verma modules coming from characters $S^\lambda$, where
$\lambda$ is bounded in an appropriate sense.

This can be addressed through the concept of {\em localization}. Suppose we have a Harish-Chandra pair
$(R_\mathbf{k},\mathbf{p}_\mathbf{l})$, and a set 
$$\mathfrak{E}\subseteq Obj((R_\mathbf{k},\mathbf{p}_\mathbf{l})\dash hwMod).$$
We say that an object $X$ of $(R_\mathbf{k},\mathbf{p}_\mathbf{l})\dash hwMod$ 
is {\em $\mathfrak{E}$-local} if for every object
$Q$ of $(R_\mathbf{k},\mathbf{p}_\mathbf{l})\dash hwMod$ which satisfies
$$(R_\mathbf{k},\mathbf{p}_\mathbf{l})\dash hwMod_*(Q,E)=0 \;\text{for all $E\in \mathfrak{E}$},$$
we have
$$(R_\mathbf{k},\mathbf{p}_\mathbf{l})\dash hwMod_*(Q,X)=0.$$
(Here $?_*$ denotes the graded morphism group, i.e. allowing arbitrary $\Z$-suspensions of the 
objects involved.) The full subcategory of $(R_\mathbf{k},\mathbf{p}_\mathbf{l})\dash hwMod$ on $\mathfrak{E}$-local
objects will be denoted by $\mathfrak{E}\dash (R_\mathbf{k},\mathbf{p}_\mathbf{l})\dash hwMod$.

\vspace{3mm}

\begin{proposition}\label{papp1}
Under the assumptions of the beginning of this subsection, the inclusion
$$\mathfrak{E}\dash(R_\mathbf{k},\mathbf{p}_\mathbf{l})\dash hwMod\subseteq 
(R_\mathbf{k},\mathbf{p}_\mathbf{l})\dash hwMod$$
has a left adjoint $L_\mathfrak{E}$, called {\em localization}. Additionally, given a morphism
of Harish-Chandra pairs
$$f:(R_1,g_1)\r (R_2,g_2)$$
satisfying the same assumptions as those of Proposition \ref{pindr},
let $E\subseteq (R_1,g_1)\dash hwMod$ and let
$$f_*(\mathfrak{E})=\{f_*(E)\mid E\in \mathfrak{E}\}.$$
Then $f_*$ restricts to a functor
$$f_{*,\mathfrak{E}}:
\mathfrak{E}\dash(R_1,g_1)\dash hwMod\r f_*\mathfrak{E}\dash(R_2,g_2)\dash hwMod,$$
which is right adjoint to the functor
$$f^{*,\mathfrak{E}}:
f_*\mathfrak{E}\dash(R_2,g_2)\dash hwMod\r \mathfrak{E}\dash(R_1,g_1)\dash hwMod$$
given by
$$f^{*,\mathfrak{E}}(X)=L_\mathfrak{E}f^*(X).$$
\end{proposition}

\begin{proof}
Formal from the definitions.
\end{proof}

\vspace{3mm}
\vspace{3mm}
\begin{lemma}\label{llocal1}
If $X\in \mathfrak{E}$ then $X$ is $\mathfrak{E}$-local, i.e. $L_\mathfrak{E}(X)\cong X$.
\end{lemma}

\begin{proof}
Formal.
\end{proof}

\vspace{3mm}
In all the cases we will be interested in this paper from now on, we will work under the assumptions
in the beginning of this subsection. We will consider the Borel subgroup $B$ of $
{GL_n}$
of lower triangular matrices, with $\mathscr{O}_{B}$-equivariant
parabolic Lie algebra $\mathbf{b}$ (the model constructed in \rref{eparabphi}). 
We shall work in the $M\Z/p$-localized category of Harish-Chandra pair representations.

At this point, a comment must be made on the $\rho$-shift. It was noted in Subsection \ref{ssverma}
that for $gl_n$, the components of the $\rho$-shift may not be integral, so we may need to add a half-integral multiple of
the determinant weight to obtain integral components. It will be advantageous for us to do this once and 
for all. Let, therefore,
$$\rho^\prime=(n-1,n-2,\dots,0).$$
We shall, from now on, shift by $\rho^\prime$ instead of $\rho$.

The set $\mathfrak{E}_0\subset (\mathscr{O}_H,\mathbf{b})\dash hwMod$ is the set of all
modules 
of $\rho^\prime$-shifted weights
whose coordinates are in $\{0,\dots,k-1\}$ where $k<<p$. All the classes of objects of
$(R,g)\dash hwMod$ we will localize at will then be of the form $\mathfrak{E}=f_*(\mathfrak{E}_0)$ where
$\mathfrak{E}_0$ is as defined above, and $f:(\mathscr{O}_H,\mathbf{b}) \r(R,g)$ is (the coefficients of)
an inclusion of Harish-Chandra pairs.

\vspace{3mm}
In fact, using the same notation, we shall be interested in the specific case where $R=\mathscr{O}_{H}$, 
$g=gl_n$. We then define the category $\mathcal{O}_{n,k}$ (where $k$ is as in the last paragraph)
as the category of 
$(\mathscr{O}_{H},gl_n)$-modules which can be filtered so that the associated graded terms
are $\Z$-suspensions of $f_* S^\lambda$ with $S^\lambda\in \mathfrak{E}_0$ (i.e. co-Verma-modules).
We also denote $\mathfrak{E}=f_*(\mathfrak{E}_0)$.

More generally, we will consider the case where $R=\mathscr{O}_{L}$ where $L$ corresponds to 
the Levi factor of a parabolic $\mathbf{p}\supseteq \mathbf{b}$ 
(always interpreted in the sense of \rref{eparabphi}), which is a parabolic
Lie subalgebra of $gl_n$, and $g=gl_n$. In this case, we define the category 
$\mathcal{O}_{\mathbf{p},n,k}$ as the category of 
$(\mathscr{O}_{L},gl_n)$-modules which can be filtered so the associated graded terms
are $\Z$-suspensions of $f_* S^\lambda$ with $S^\lambda\in \mathfrak{E}_0$ (i.e. generalized 
co-Verma-modules). In this case, we denote $\mathfrak{E}=
\mathfrak{E}_{\mathbf{p}}=f_*(\mathfrak{E}_0)$.

\vspace{3mm}
\begin{theorem}\label{thsduality}
Under the standing assumptions, let 
$$f=\mathscr{O}_\kappa:\mathscr{O}_{L_1}\r\mathscr{O}_{L_0}$$
where
$\kappa:\mathbf{p}_0\r \mathbf{p}_1$ is an inclusion of parabolics in $g$, and $\ell_1, \ell_0$ are
the corresponding Levi factors. Then we have a functor
\beg{eopkappa1}{\kappa^*=L_{\mathfrak{E}_{\mathbf{p}_0}}Res_f:\mathcal{O}_{\mathbf{p}_1,n,k}\r 
\mathcal{O}_{\mathbf{p}_0,n,k}}
whose left derived functor is left adjoint to a right derived functor of
\beg{eopkappa2}{\kappa_*=Ind_f:\mathcal{O}_{\mathbf{p}_0,n,k}\r \mathcal{O}_{\mathbf{p}_1,n,k}.}
Additionally, the left derived functor of \rref{eopkappa1} also has a left adjoint denoted by 
\beg{eopkappa3}{\kappa_\sharp,} 
and we have, on the level of derived functors, 
\beg{eopkappa4}{\kappa_\sharp=\kappa_*[2(dim(\mathbf{p}_1)-dim(\mathbf{p}_0))]}
where 
the functors \rref{eopkappa2} and \rref{eopkappa3} are referred to as {\em right (resp. left) 
Zuckermann functor}.
\end{theorem}

\begin{proof}
To prove the adjunction between \rref{eopkappa1} and \rref{eopkappa2}, we first show that
the functors are well-defined. In case of \rref{eopkappa2}, we may as well think of $f$ as a morphism
of Harish-Chandra pairs 
\beg{eop1op1}{(\mathscr{O}_{L_0},\mathbf{p}_1)\r(\mathscr{O}_{L_1},\mathbf{p}_1)}
(since the objects both in the source and target of \rref{eopkappa2} are obtained from this context by
applying $\kappa_*$ from $\mathbf{p}_1$ to $gl_n$. Now computing the $\kappa_*$ of \rref{eop1op1} on
a co-Verma module is an extension of the Example at the end of last section: In the range of weights specified,
we obtain a $\mathscr{O}_{L_1}$-comodule which is a wedge of $m$ spheres where $m$ is finite and
equal to the dimension of the Levi factor representation we obtain when applying an analogous construction 
over $\C$. (Over $S$, we shall also refer to this as a {\em finite representation}.)
This proves that \rref{eopkappa2} is well defined. 

To treat \rref{eopkappa1}, once again, we may work with the morphism of Harish-Chandra pairs \rref{eop1op1}
instead, but this time, we shall also consider the morphism
\beg{eop1op2}{g:(\mathscr{O}_{L_0},\mathbf{p}_0)\r (\mathscr{O}_{L_0},\mathbf{p}_1).
}
To resolve $W^\prime=f^*(W)$ where $W$ is a finite representation in terms of co-Verma modules,
consider first 
\beg{eop1op3}{g_*g^*(W^\prime).
}
By the projection formula, \rref{eop1op3} is equivalent to the $(\mathscr{O}_{L_0},\mathbf{p}_1)$-module
\beg{eop1op4}{g_*g^*(S)\wedge W^\prime
}
where $S$ is the trivial comodule. But now letting $Cg$ denote the cone on an $S$-Lie algebra $g$ in
the category of $S$-Lie algebras, we may also consider the morphism of Harish-Chandra pairs
\beg{eop1op5}{h:(\mathscr{O}_{L_0},C\mathbf{p}_0)\r (\mathscr{O}_{L_0},C\mathbf{p}_1),
}
and in particular we have a canonical morphism of $(\mathscr{O}_{L_0},\mathbf{p}_1)$-modules from
\rref{eop1op4} to
\beg{eop1op5a}{h_*h^*(S)\wedge W^{\prime}.
}
Moreover, the composite morphism 
$$W^\prime\r h_*h^*(S)\wedge W^{\prime}$$
is an equivalence (the {\em dual Koszul resolution}). Filtering by the dimensions in the suspended cone 
coordinate, we see that the dual Koszul resolution has a decreasing filtration by co-Verma modules and moreover there is
a morphism 
\beg{eop1op6}{h_*h^* (S)\wedge W^\prime\r Q}
where $Q$ consists of finitely many filtered pieces in the specified weight range, 
and the homotopy fiber $F$ of \rref{eop1op6} is concentrated in lower weights. ($F$ comes from the fact that
$C_\infty S^{-1}$, completed at $p$, has an ``exterior algebra'' part similar as over $\C$, and then additional 
``derived'' parts coming from extended $p$'th powers.)

In any case, this discussion implies that \rref{eop1op6} is localization at $\mathfrak{E}$ by Lemma \rref{llocal1}.

\vspace{3mm}
To prove the adjunction between \rref{eopkappa3} and \rref{eopkappa1}, and \rref{eopkappa4}, we use
formal duality arguments together with calculations which go through because of the finiteness and
large prime hypotheses. 

Specifically, to manufacture a left adjoint out of a right adjoint, we need an invertible object $\omega$ and a 
morphism
\beg{edualizer}{\tau:\kappa_*\omega\r S.
}
In our case, we have
\beg{edualizer1}{\omega = S[2(dim(\mathbf{p}_1)-dim(\mathbf{p}_0)]
}
and the appropriate morphism \rref{edualizer} is constructed from the rational case using the large
prime hypothesis. Next, one proves the projection formula
\beg{edualizer2}{\diagram
\varrho:X\wedge \kappa_*\omega \rto^\sim & \kappa_*(\kappa^*X\wedge \omega).
\enddiagram
}
One then sets
\beg{edualizer3}{
\kappa_\sharp(X)=\kappa_*(X\wedge\omega),
}
and the counit of adjunction is constructed as
\beg{edualizer4}{
\diagram
\kappa_\sharp\kappa^*X\rto^(.4){Id} &
\kappa_*(\kappa^*X\wedge \omega)\rto^(.55){\rho^{-1}} &
X\wedge
\kappa_*\omega\rto^(.6)\tau & X.
\enddiagram
}
Validity of the triangle identities up to isomorphism is a calculation, again, mimicking the rational case
using the large prime hypothesis. 

One thing to note is that in our present setting, the objects $\omega$ and $S$, while they are comodules,
are not objects of the categories $\mathcal{O}_{\mathbf{p},n,k}$. However, smashing with them is an endofunctor
on $\mathcal{O}_{\mathbf{p},n,k}$, which is precisely what we need.
\end{proof}

\vspace{5mm}

\section{Khovanov cube and diagram relations}\label{skhovanov}

At this point, we start considering a smooth projection of an oriented link. We assume (just as \cite{sussan})
that the crossings
are at most double, that they are transverse and that the plane is given a system of Cartesian coordinates 
where in the neighborhood of each crossing, the two crossing strands are in the first and third resp. second and
fourth quadrant when the origin is shifted to the point of the crossing. Furthermore, we assume that the
strands are oriented upward (i.e. from negative to positive in the direction of the $y$-coordinate). 
We also assume that when the projection is tangent to any horizontal line (i.e. line parallel to the $x$ axis), 
then the critical point is non-degenerate (i.e. it is a graph of a function with non-zero second derivative).\vspace{3mm}

\subsection{The basic setup}\label{sssetup}

As in Sussan \cite{sussan}, we will assign to each oriented link projection as above a stable homotopy type 
(more precisely a finite spectrum). More generally, we will assign mathematical entities to oriented tangles projected
into
$$\R\times [a_0,a_1],\; a_0<a_1$$
with crossings as above
where the only points with $y$-coordinate $a_0$ or $a_1$ are the ends, and the ends meet the horizontal lines
transversely. Let
$$n_{a_i}=\sum n(s)
$$
where the sum is over the strands ending on the line $y=a_i$, and $n(s)=1$ resp. $n(s)=k-1$ for
an upward resp. downward oriented strand. Consider the graded  $S$-Lie algebras $g_i=gl_{n_{a_i}}$,
$i=0,1$. Consider the category 
$$\mathcal{O}_{n,k}$$
introduced after Lemma \ref{llocal1}. 
Now consider the parabolic $\mathbf{p}_i\subseteq gl_{n_{a_i}}$ whose Levi factor consists of block sums of matrices
with $1\times 1$ block for every upward strand, and a $(k-1)\times (k-1)$ block for every downward strand. Using 
the machinery of Subsection \ref{ssapprox}, we will consider the parabolic BGG categories
$$\mathcal{O}_{\mathbf{p}_i,i},\; i=0,1$$
with respect to the parabolic $\mathbf{p}_i$ inside the Lie algebra $gl_{n_{a_i}}$. (Throughout, we are working completed
at a large prime.) 

\vspace{3mm}
The ``Khovanov cube'' associated with a tangle projection as described will be an $\mathscr{S}$-enriched functor
\beg{ekhovanov1}{\mathfrak{K}:\mathcal{O}_{\mathbf{p}_0,0}\r\mathcal{O}_{\mathbf{p}_1,1}.
}
The idea is that if there are no input or output strands, the functor would be
$$\mathfrak{K}:\mathscr{S}\r\mathscr{S},$$
which is equivalent to specifying its value on $S$: This is our $sl_k$-Khovanov stable homotopy type. Similarly
as in Sussan \cite{sussan}, we will prove its invariance with respect to the relevant flavor of Reidemeister moves, 
thereby showing that it is a link invariant.

\vspace{3mm}

\subsection{Equivalences of categories}\label{sscaston}

Let $\mathbf{p}\subseteq gl_n$ be a standard parabolic with Levi factor $\ell$ and let 
$$\mathcal{O}_{\mathbf{p},n,k}$$
be the parabolic BGG category of $\mathcal{O}_{n,k}$ with respect to $\mathbf{p}$. Now consider the parabolic 
$\mathbf{q}\subseteq gl_{k+n}$
with Levi factor $gl_k\oplus \ell$. Let $\overline{\mathbf{p}}$ 
also denote the parabolic $\mathbf{b}+\mathbf{p}$ in $gl_{k+n}$ 
(which has Levi factor $h_k\oplus \ell$) and let
$\mathbf{g}_{k,n}$ denote the parabolic in $gl_{k+n}$ with Levi factor $gl_k\oplus gl_n$. 
Denote by 
$$\kappa:\overline{\mathbf{p}}\r\mathbf{q}$$
the inclusion. 

We have a 
canonical projection of Harish-Chandra pairs
$$\pi: (\mathscr{O}_{H_k\times L}, \mathbf{g}_{k,n})\r (\mathscr{O}_{H_k\times L},gl_k\oplus gl_n).$$ 
Denote by $\Delta$ the power of the determinant comodule of $\mathscr{O}_{GL_{k}}$
with weight equal to the difference of the $(\rho^\prime)$'s of $gl_{n+k}$ and $gl_k$.
(Note: all these Lie
algebras are associated with reflexive transitive relations, and hence have versions over $S$.) Let $V$ denote
the $(\mathscr{O}_{H_k},gl_k)$-co-Verma module on the character with $\rho^\prime$-shifted
weight $(k-1,k-2,\dots,0)$. Let 
$$f:(\mathscr{O}_{H_k\times L}, \mathbf{g}_{k,n})\r (\mathscr{O}_{H_k\times L}, gl_{k+n})$$
be the canonical ``inclusion'' of Harish-Chandra pairs. 
Consider
the functor
\beg{eoplusverma}{\psi:f_*\circ
\pi^*\circ ((\Delta\wedge V)\underline{\wedge} ?):\mathcal{O}_{\mathbf{p},n,k}\r 
\mathcal{O}_{\overline{\mathbf{p}},k+n,k}
}
where $\underline{\wedge}$ denotes the ``external smash product'' which from representations of the
Harish-Chandra pairs $(\mathscr{O}_{H_k},gl_k)$ and $(\mathscr{O}_L,gl_n)$
produces a representation of $(\mathscr{O}_{H_k\times L}, gl_{k}\oplus gl_{n})$.

Now denote
\beg{ecastonphi}{\phi=\kappa_*\circ \psi: \mathcal{O}_{\mathbf{p},n,k}\r\mathcal{O}_{
\mathbf{q},k+n,k}.}

\vspace{3mm}
\begin{lemma}\label{lcaston}
The functor $\phi$ induces an equivalence of derived categories.
\end{lemma}

\begin{proof}
Note that by our large prime hypothesis, the functor $\phi$ defines a bijective correspondence of the
co-Verma modules of $\mathbf{p}$ and of $\mathbf{q}$ in the given weight range. In the given weight range, 
the functor further induces an isomorphism of $Hom$-sets over $\Z_p$ (the $p$-adic numbers), and over $S_p$ they
will just be tensored with the stable $p$-stems.
\end{proof}

\vspace{3mm}
\noindent
{\bf Comment:} The equivalence of derived categories $\phi$ is actually dual to the equivalence used by
Sussan \cite{sussan} in the sense that we use right instead of left adjoints. The reason for this variation is
that we have developed the right adjoint functor $f_*$ to $f^*$ for a morphism of Harish-Chandra pairs, so using 
the right adjoints throughout is easier. The constructions and observations which follow actually do not depend on any
special properties of this equivalence of categories.

\vspace{3mm}
Now using this, following Sussan \cite{sussan}, for a sequence of numbers $n_1,\dots, n_m$, $1\leq n_i\leq k$,
letting $n=\sum n_i$ and letting $n^\prime$ be the sum of all the $n_i$'s which are not equal to $k$,
and letting $\mathbf{p}$ be the parabolic in $gl_n$ with Levi factor
\beg{ecastonlevi}{gl_{n_1}\oplus\dots\oplus gl_{n_m}}
and letting $\mathbf{p}\prime$ be the parabolic in $gl_{n^\prime}$ with Levi factor
\rref{ecastonlevi} modified by omitting all $gl_k$ summands, we shall construct an equivalence of 
categories 
\beg{ecastonlevi1}{\mathcal{O}_{\mathbf{p},n,k}\r\mathcal{O}_{\mathbf{p}^\prime,n^\prime,k}.
}
The equivalence \rref{ecastonlevi1} is constructed in \cite{sussan} by combining sums of equivalences of 
derived categories
of the form
\beg{ecastonrho}{\mathcal{O}_{\mathbf{p}_1,j+k,k}\r \mathcal{O}_{\mathbf{p}_2,k+j,k}
}
where $\mathbf{p}_1$ resp. $\mathbf{p}_2$ is the parabolic with Levi factor $gl_j\oplus gl_k$ resp. $gl_k\oplus gl_j$, followed
by the inverse of the equivalence of Lemma \ref{lcaston}.
We follow the same approach. The equivalence \rref{ecastonrho} is constructed by considering the parabolic
$\mathbf{p}$ with Levi factor $gl_{j}\oplus gl_{k-j}\oplus gl_{j}$ and composing the $\kappa^*$ with
respect to $\mathbf{p}_1$, $\mathbf{p}$ with the $\kappa_\sharp$ with respect to $\mathbf{p}$, 
$\mathbf{p}_2$. Again, \rref{ecastonrho}
is then a derived equivalence by the large prime hypothesis.

\vspace{3mm}
\subsection{Defining the Khovanov cube}\label{sskhovanovcube}

In what follows, we shall use $\kappa$ generically for the type of ``inclusion'' of Harish-Chandra pairs
which occurs in Theorem \ref{thsduality}, when there is no danger of confusion. When more than one such
morphism is in sight, we shall be more specific about the notation.

Now for a projection of a tangle $T$ as above, assume we have assigned a functor
\rref{ekhovanov1} to $T$. Now assume that $n(s)$ and $n(s+1)$ in the sum defining $n_{a_1}$ are
$1$ and $k-1$ (in either order) and assume a tangle projection $T^\prime$ is obtained by
replacing $a_1$ with $a_1+\epsilon$ and joining the $s$'th and $(s+1)$'st strands at $a_1$. Then
the functor \rref{ekhovanov1} for the tangle $T^\prime$ is obtained by taking
the functor \rref{ekhovanov1} for the tangle $T$, then applying the
derived left Zuckermann from $\mathbf{p}_1$ to $\mathbf{p}_1^\prime$ (where $\mathbf{p}_1^\prime$ is the corresponding 
parabolic with respect to $T^\prime$), followed by the equivalence \rref{ecastonlevi1}.

\vspace{3mm}
Symmetrically, if 
$n(s)$ and $n(s+1)$ in the sum defining $n_{a_0}$ are
$1$ and $k-1$ (in either order) and assume a tangle projection $T^\prime$ is obtained by
replacing $a_0$ with $a_0-\epsilon$ and joining the $s$'th and $(s+1)$'st strands at $a_0$. Then
the functor \rref{ekhovanov1} for the tangle $T^\prime$ is obtained by applying the
inverse of the equivalence \rref{ecastonlevi1}, followed by the $\kappa^*$ functor
with respect to $\mathbf{p}_0$ and $\mathbf{p}_0^\prime$, followed by the functor \rref{ekhovanov1} for $T$.

\vspace{3mm}
Now assume that $n(s)=n(s+1)=1$ in the sum defining, say, $n_{a_1}$ for an oriented tangle projection $T$ 
as above and suppose that a tangle projection $T^\prime$ is obtained by replacing $a_1$ with $a_1+\epsilon$,
and adding a crossing between the $s$'th and $(s+1)$'st strand. Consider the parabolic $\mathbf{p}_1^\prime$ whose Levi factor
is generated by the Levi factor of $\mathbf{p}_1$ and a copy of $gl_2$ on the $s$'th and $(s+1)$'st strands. 

If the $s$'th strand at $a_1$ of $T$ crosses above the $(s+1)$'st strand, then assign to $T^\prime$ the
functor \rref{ekhovanov1} obtained from the functor \rref{ekhovanov1} for $T$ followed by the
homotopy cofiber of the counit of adjunction
\beg{ekhovcounit}{\kappa^*\kappa_*\r Id}
where $\kappa$ is with respect to the parabolic subalgebras $\mathbf{p}_1$, $\mathbf{p}_1^\prime$.
 
If the $(s+1)$'st strand of $T$ at $a_1$ crosses above the $s$'th strand, then assign to $T^\prime$ the
functor \rref{ekhovanov1} obtained from the functor \rref{ekhovanov1} for $T$ followed by the homotopy cofiber
of the unit of adjunction
\beg{ekhovunit}{Id\r\kappa^*\kappa_\sharp}
where again, $\kappa$ is with respect to $\mathbf{p}_1$ and $\mathbf{p}_1^\prime$.

\vspace{3mm}
Now completely analogously 
as in Sussan \cite{sussan}, the Reidemeister relations follow from Diagram relations 1-5 of \cite{sussan}. The 
remainder of this paper 
consists of proving these relations. We will give each of the diagrams, taken from \cite{sussan}, as it is proven below.

\vspace{3mm}

\subsection{Diagram relation 2} \label{ssr2}

The statement we need is contained in the following theorem.

\begin{theorem}\label{tr2}
Let $\mathbf{p}\subset \mathbf{q}\subseteq gl_n$ be  parabolic subalgebras where $\mathbf{p}$ has Levi factor 
$\ell=\ell_1\oplus gl_1\oplus gl_1\oplus \ell_2$ and $\mathbf{q}$ 
has Levi factor $m=\ell_1\oplus gl_2\oplus \ell_2$. Then
the functor
$$\kappa^*:\mathcal{O}_{\mathbf{q},n,k}\r\mathcal{O}_{\mathbf{p},n,k}$$
$$\kappa_*,\kappa_\sharp:\mathcal{O}_{\mathbf{p},n,k}\r\mathcal{O}_{\mathbf{q},n,k}$$
satisfy
\beg{tr2e1}{R\kappa_*\cong L\kappa_\sharp[-2],
}
\beg{tr2e2}{L\kappa_\sharp\kappa^*\cong Id\vee Id[2].
}
\end{theorem}

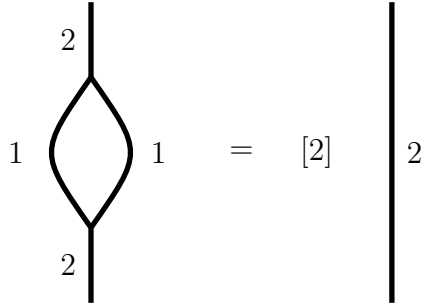
\begin{figure}[H]
\label{relation2}
\begin{tikzpicture}[line width=2pt]
\draw (-2,2) -- (-2,1);
\node at (-2.3, 1.5) {2};
\draw (-2,1) .. controls (-2.7,0) .. (-2,-1);
\node at (-3, 0) {1};
\draw (-2,1) .. controls (-1.3,0) .. (-2,-1);
\node at (-1.1, 0) {1};
\draw (-2,-1) -- (-2,-2);
\node at (-2.3, -1.5) {2};

\node at (0, 0) {=};
\node at (1, 0) {[2]};

\draw (2,2) -- (2,-2);
\node at (2.3, 0) {2};
\end{tikzpicture}
\caption{Diagram Relation 2}
\end{figure}

This statement is represented graphically in Figure 1. Here, the figure is read from top to 
bottom. At the top, the single edge labeled by $2$ corresponds to $\mathcal{O}_{\mathbf{q}, n, k}$, 
as $\mathbf{q}$ has the Levi factor of $gl_2$. Moving 
down, we get two edges labeled by $1$, which correspond to $\mathcal{O}_{\mathbf{p}, n, k}$, 
with the Levi factor of $gl_1 \oplus gl_1$ in $\mathbf{p}$. (The Levi factors $l_1$ adnd $l_2$ are not 
shown in the figure.)
Splitting a strand into two 
represents the functor 
$\kappa^*$, and combining two strands represents the functor $\kappa_{\sharp}$, which goes back to
$\mathcal{O}_{\mathbf{q}, n, k}$.
On the right hand side, the edge that goes straight through represents the identity functor on 
$\mathcal{O}_{\mathbf{q}, n, k}$, and $[2]$ denotes the 
wedge of 2 copies of this identitiy functor, suspended by $0$ and $2$.

\begin{proof}
In fact, \rref{tr2e1} is a special case of Theorem \ref{thsduality}. For \rref{tr2e2}, we can think of
this is an example of Verdier duality in the case of Kan extensions. We need to find an invertible object
$\omega$ in $D\mathcal{O}_{\mathbf{p},n,k}$ and a morphism in $\mathcal{O}_{\mathbf{q},n,k}$ of the form
\beg{etr2e1}{t:S\r \kappa_\sharp\omega^{-1}.
}
The duality morphism is then the morphism in the derived category induced by
\beg{etr2e2}{
\diagram
\kappa_*(X\wedge \omega)\dto^{t\wedge Id}\\
\kappa_\sharp\omega^{-1}\wedge\kappa_*(X\wedge \omega)\dto^\varrho\\
\kappa_\sharp(\omega^{-1}\wedge\kappa^*\kappa_*(X\wedge \omega))
\dto^{\kappa_\sharp(Id\wedge\epsilon)}\\
\kappa_\sharp(\omega^{-1}\wedge X\wedge \omega)\dto^\simeq\\
\kappa_\sharp(X)
\enddiagram
}
where $\varrho$ is the projection formula equivalence
\beg{etr2e3}{(\kappa_\sharp X)\wedge Y\simeq \kappa_\sharp(X\wedge \kappa^* Y).
}
To prove a duality, we must choose $\omega$ and \rref{etr2e1} so that \rref{etr2e2} induces
an isomorphism in the derived category. In the present case, $S$ is the ``trivial representation''
and $\omega=S[2]$. The construction of \rref{etr2e1} is an explicit calculation,
as is \rref{etr2e2} in the case of parabolic co-Verma modules. This proves \rref{tr2e1}.

To prove \rref{tr2e2}, we apply the projection formula
$$X\wedge \kappa_\sharp \omega^{-1}\simeq \kappa_\sharp (\kappa^* X
\wedge \omega^{-1}),$$
and calculate more precisely
$$\kappa_\sharp \omega^{-1}\simeq S\vee S[-2].$$
\end{proof}

\vspace{3mm}
\subsection{Diagram relation 1}\label{ssr1}

The statement for Sussan's Diagram relation 1 in this case is the following theorem.

\begin{theorem}\label{tr1}
Let $\mathbf{p}\subset \mathbf{q}\subseteq gl_n$ be  parabolic subalgebras, where $\mathbf{p}$ has Levi factor 
$\ell=\ell_1\oplus gl_i\oplus gl_j\oplus \ell_2$ where
$\{i,j\}=\{1,k-1\}$ and $\mathbf{q}$ has Levi factor $m=\ell_1\oplus gl_k\oplus \ell_2$. Then
the functors
$$\kappa^*:\mathcal{O}_{\mathbf{q},n,k}\r\mathcal{O}_{\mathbf{p},n,k}$$
$$\kappa_*,\kappa_\sharp:\mathcal{O}_{\mathbf{p},n,k}\r\mathcal{O}_{\mathbf{q},n,k}$$
satisfy
\beg{tr1e1}{R\kappa_*\cong L\kappa_\sharp[-2k+2],
}
\beg{tr1e1}{L\kappa_\sharp\kappa^*\cong Id\vee Id[2]\vee\dots\vee Id[2k-2].
}
\end{theorem}
\vspace{2mm}

Here, $Id$ is the identity functor on $\mathcal{O}_{\mathbf{q}, n, k}$. 
Graphically, the relation is represented by Figure 2. Once more, on the right hand side of the figure, $[k]$ denotes
the wedge of $k$ copies of the identity functor, with suspensions by $0, 2, \ldots, 2k-2$. On the left and side, 
the node splitting the edge labeled by $k$ to two edges, labeled by $k-1$ and $1$, represents $\kappa^{\ast}$, 
and the node combining two edges into one represents $\kappa_{\sharp}$.

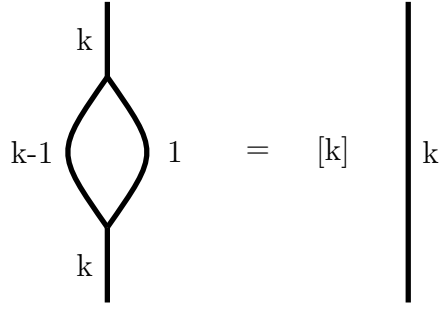
\begin{figure}[H]
\label{relation1}
\begin{tikzpicture}[line width=2pt]
\draw (-2,2) -- (-2,1);
\node at (-2.3, 1.5) {k};
\draw (-2,1) .. controls (-2.7,0) .. (-2,-1);
\node at (-3, 0) {k-1};
\draw (-2,1) .. controls (-1.3,0) .. (-2,-1);
\node at (-1.1, 0) {1};
\draw (-2,-1) -- (-2,-2);
\node at (-2.3, -1.5) {k};

\node at (0, 0) {=};
\node at (1, 0) {[k]};

\draw (2,2) -- (2,-2);
\node at (2.3, 0) {k};
\end{tikzpicture}
\caption{Diagram Relation 1}
\end{figure}

\begin{proof}
Analogous to the proof of Theorem \ref{tr2} with the exception that
$$\omega=S[2k-2],$$
and one calculates
$$\kappa_\sharp \omega^{-1}=S\vee S[-2]\vee\dots\vee S[-2k+2].$$
\end{proof}

\vspace{3mm}
\subsection{Diagram relation 3}\label{ssr3}

Sussan's Diagram relation 3 is represented graphcially by Figure 3. 

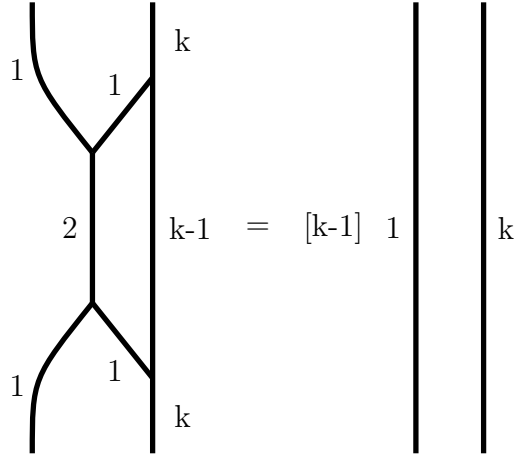
\begin{figure}[H]
\label{relation3}

\begin{tikzpicture}[line width=2pt]
\draw (-1.2, 3) -- (-1.2, -3);
\node at (-0.7, 0) {k-1};
\node at (-0.8, 2.5) {k};
\node at (-0.8, -2.5) {k};
\draw (-1.2, 2) -- (-2, 1);
\node at (-1.7, 1.9) {1};
\draw (-2, 1) -- (-2, -1);
\node at (-2.3, 0) {2};
\draw (-2, -1) -- (-1.2, -2);
\node at (-1.7, -1.9) {1};
\draw (-2.8, 3) .. controls(-2.8, 2) .. (-2, 1);
\node at (-3, 2.1) {1};
\draw (-2, -1) .. controls(-2.8, -2) .. (-2.8, -3);
\node at (-3, -2.1) {1};

\node at (0.2, 0) {=};
\node at (1.2, 0) {[k-1]};

\draw (2.3, 3) -- (2.3, -3);
\node at (2, 0) {1};
\draw (3.2, 3) -- (3.2, -3);
\node at (3.5, 0) {k};
\end{tikzpicture}
\caption{Diagram Relation 3}
\end{figure}

The corresponding statement is the following theorem.

\begin{theorem}\label{tr3}
Let 
\beg{tr3e1}{\mathbf{p}\supset \mathbf{q} \subset \mathbf{p}^\prime} 
be parabolic subalgebras in $gl_n$ with 
Levi factors $\ell_1\oplus gl_1\oplus gl_k\oplus \ell_2$, 
$\ell_1\oplus gl_1\oplus gl_1\oplus gl_{k-1}\oplus \ell_2$,
$\ell_1\oplus gl_2\oplus gl_{k-1}\oplus \ell_2$, respectively. 
Denote the first inclusion $\mathbf{q} \subset \mathbf{p}$ in \rref{tr3e1} by $\kappa$ and the
 second inclusion $\mathbf{q} \subset \mathbf{p}^{\prime}$ in
\rref{tr3e1} by $\nu$. Then 
\beg{tr3e2}{L\kappa_\sharp\nu^*L\nu_{\sharp}\kappa^*\cong
Id[2]\vee Id[4]\vee\dots\vee Id[2k-2].
}
\end{theorem}
\vspace{2mm}

Here, $Id$ denotes the identity functor on $\mathcal{O}_{\mathbf{p}, n, k}$, which is shown on the right 
hand side of the figure by the two edges labelled by $1$ and $k$, and $[k-1]$ denotes the wedge of $k-1$ copies
of this identity functor, with the appropriate suspensions. Each node on the left hand side of the figure represents 
a functor, which goes from the category represented just above the node to the category represented just below the 
node. 

\begin{proof}
Using Theorem \ref{tr1}, we have a unit of adjunction
$$Id[0]\vee\dots\vee Id[2k]\simeq 
L\kappa_\sharp\kappa^*\r
L\kappa_\sharp\nu^*L\nu_{\sharp}\kappa^*.$$
Restrict this morphism to 
$$Id[2]\vee Id[4]\vee\dots\vee Id[2k-2].$$
One verifies \cite{sussan} that the left derived functor of the morphism obtained is an equivalence.
\end{proof}

\vspace{3mm}
\subsection{Diagram relation 4} \label{ssr4}
Consider the diagram of parabolic subalgebras in $gl_n$:
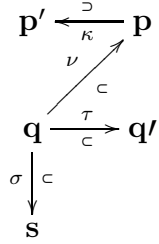
$$\diagram
\mathbf{p}^\prime & \mathbf{p}\lto_\supset^{\kappa} \\
\mathbf{q}\urto_\subset^{\nu}\rto_\subset^{\tau}\dto^\subset_{\sigma} &\mathbf{q}
\prime\\
\mathbf{s}&
\enddiagram
$$
which on Levi factors is
$$
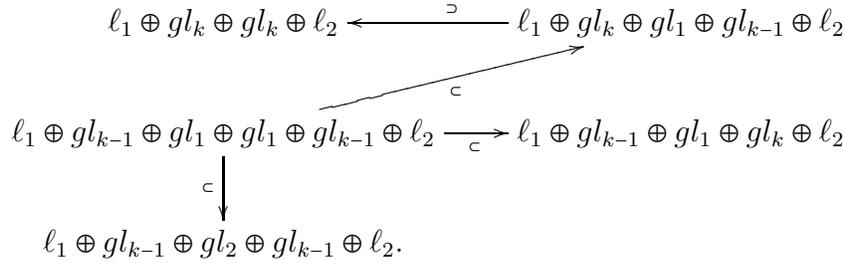
\diagram
\ell_1\oplus gl_k\oplus gl_k\oplus \ell_2 & \ell_1 \oplus gl_k\oplus gl_1\oplus gl_{k-1}\oplus \ell_2\lto_\supset \\
\ell_1\oplus gl_{k-1}\oplus gl_1\oplus gl_1\oplus gl_{k-1}\oplus \ell_2
\urto_\subset\rto_(.55)\subset\dto_\subset &\ell_1\oplus gl_{k-1}\oplus gl_1\oplus gl_k\oplus \ell_2\\
\ell_1\oplus gl_{k-1}\oplus gl_2\oplus gl_{k-1}\oplus \ell_2.&
\enddiagram
$$

Sussan's Diagram relation 4 is given by the following theorem. 

\begin{theorem}\label{tr4}
There is an isomorphism in the derived category
\beg{tr4e1}{
L\nu_\sharp \sigma^*L\sigma_\sharp \tau^*
L\tau_\sharp \sigma^*L\sigma_\sharp \nu^*
\cong 
Id[2k] \vee\kappa^*L\kappa_\sharp[4]\vee\dots\vee 
\kappa^*L\kappa_\sharp[2k-2].
}
\end{theorem}

Here, $Id$ is the identity functor on $\mathcal{O}_{\mathbf{p}, n, k}$. Graphically, the Diagram 
relation is represented by Figure 4. 

\begin{figure}[H]
\label{relation4}
\begin{tikzpicture}[line width= 2pt]

\draw(-1.2, 3) -- (-1.2, -3);
\node at (-0.9, 0) {k};
\node at (-0.8, 1.5) {k-1};
\node at (-0.8, -1.5) {k-1};
\draw(-4, 3) -- (-4, -3);
\node at (-4.2, 2.5) {k};
\node at (-4.2, -2.5) {k};
\node at (-4.4, 0) {k-1};
\draw(-3, 0.8) -- (-1.2, 0.5);
\node at (-2.2, 0.9) {1};
\draw(-3, -0.8) -- (-1.2, -0.5);
\node at (-2.2, -0.9) {1};
\draw (-3, 0.8) .. controls (-3.4, 0) .. (-3, -0.8);
\node at (-3.1, 0) {1};
\draw (-3, 1.7) -- (-3, 0.8);
\node at (-3.2, 1.2) {2};
\draw (-3, -0.8) -- (-3, -1.7);
\node at (-3.2, -1.2) {2};
\draw (-4, 2.2) -- (-3, 1.7);
\node at (-3.5, 2.2) {1};
\draw (-4, -2.2) -- (-3, -1.7);
\node at (-3.5, -2.2) {1};
\draw (-2, 3) .. controls (-2, 2.2) .. (-3, 1.7);
\node at (-2.35, 2.3) {1};
\draw (-2, -3) .. controls (-2, -2.2) .. (-3, -1.7);
\node at (-2.35, -2.3) {1};

\node at (0,0) {=};

\draw (1.2, 3) -- (1.2, -3);
\node at (0.9, 0) {k};
\draw (1.9, 3) -- (1.9, -3);
\node at (1.7, 0) {1};
\draw(2.6, 3) -- (2.6, -3);
\node at (3, 0)  {k-1};

\node at (4.2, 0) {+ \ [k-2]};

\draw (5.4, 3) -- (5.4, -3);
\node at (5.6, 0) {k};
\draw(6, 3) -- (6.4, 1.7);
\node at (6.05, 2.2) {1};
\draw (6.8, 3) -- (6.4, 1.7);
\node at (6.95, 2.2) {k-1};
\draw (6.4, 1.7) -- (6.4, -1.7);
\node at (6.6, 0) {k};
\draw(6, -3) --(6.4, -1.7);
\node at (6.05, -2.2) {1};
\draw(6.8, -3) -- (6.4, -1.7);
\node at (6.95, -2.2) {k-1};

\end{tikzpicture}
\caption{Diagram Relation 4}
\end{figure}
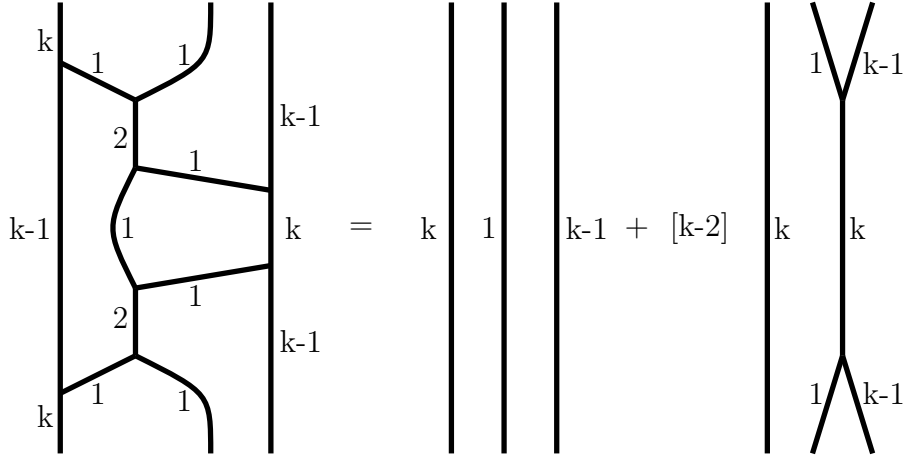

In the right hand side of 
the figure, the identity functor on $\mathcal{O}_{\mathbf{p}, n, k}$ is represented by the 
three edges going straight through (labelled by $k$, $1$ and $k-1$). The part of the right hand side after the plus sign 
denotes the wedge of $k-2$ copies of $\kappa^{\ast}L\kappa_{\sharp}$, each with the appropriate suspension. 

\begin{proof}
We have the unit of adjunction
\beg{etr4e1}{\diagram Id\dto^{\eta}\\
R\nu_* \sigma^*R\sigma_* \tau^*
L\tau_\sharp \sigma^*L\sigma_\sharp \nu^*\dto^\cong\\
L\nu_\sharp \sigma^*L\sigma_\sharp \tau^*
L\tau_\sharp \sigma^*L\sigma_\sharp \nu^*[-2k]
\enddiagram
}
where the bottom arrow follows from Theorems \ref{tr1}, \ref{tr2}. Similarly, we have morphisms
\beg{etr4e2}{\diagram
\kappa^*L\kappa_\sharp[0]\vee\dots\vee 
\kappa^*L\kappa_\sharp[2k-2]\dto^\cong\\
L\nu_\sharp  \kappa^*L\kappa_\sharp\nu^*\dto^{L\nu_\sharp \eta
\kappa^*L\kappa_\sharp\eta\nu^*
}\\
L\nu_\sharp \sigma^*L\sigma_\sharp \tau^*
L\tau_\sharp \sigma^*L\sigma_\sharp \nu^*.
\enddiagram
}
(Note that the left derived pushforwards and pullbacks associated with $\kappa$ commute with
the left derived pushforwards and pullbacks associated with $\nu$.)

Now restricting \rref{etr4e1} and \rref{etr4e2} to the summands which occur 
on the right hand side of \rref{tr4e1} gives the required isomorphism in the derived category. The fact that it is
an isomorphism is verified on co-Verma modules as in \cite{sussan}.
\end{proof}

\vspace{3mm}
\subsection{Diagram relation 5}\label{ssr5}

First consider a diagram of inclusions of parabolic subalgebras of $gl_n$ of the form
$$\diagram
\mathbf{p}_1\rto^\subset_{\nu_1} & \mathbf{q}\\
\mathbf{p}\uto^\subset_{\kappa_1}\rto^\subset_{\kappa_2} & 
\mathbf{p}_2\uto^\subset_{\nu_2}
\enddiagram
$$
with corresponding inclusions of Levi factors
$$\diagram
\ell_1\oplus gl_2\oplus gl_1\oplus\ell_2\rto^\subset & \ell_1\oplus gl_3\oplus \ell_2\\
\ell_1\oplus gl_1\oplus gl_1\oplus gl_1\oplus \ell_2\uto^\subset\rto^\subset 
& \ell_1\oplus gl_1\oplus gl_2\oplus \ell_2.\uto^\subset
\enddiagram
$$

\begin{lemma}\label{lr5}
In the derived category, we have an isomorphism
\beg{lr5e1}{
L(\kappa_1)_\sharp \kappa_2^*L(\kappa_2)_\sharp\kappa_1^*
\cong Id[2]\vee \nu_1^*L(\nu_1)_\sharp.
}
\end{lemma}

\begin{proof}
The morphism from the left hand side of \rref{lr5e1} to the first summand on the right hand side is the composition
\beg{elr5e0}{\diagram
L(\kappa_1)_\sharp \kappa_2^*(L\kappa_2)_\sharp\kappa_1^*[-2]\dto^\cong\\
L(\kappa_1)_\sharp \kappa_2^*(R\kappa_2)_*\kappa_1^*\dto^\epsilon\\
Id
\enddiagram}
where the top equivalence is by Theorem \ref{tr2} and the bottom arrow is the counit of adjunction.

To construct a morphism from the left hand side of \rref{lr5e1} to the second summand of the right hand side,
we proceed in several steps. First, we have a counit of adjunction
\beg{elr5e1}{L(\kappa_1)_\sharp \kappa_1^*\r Id,
}
so by composition, we obtain a morphism
\beg{elr5e2}{
L(\nu_2)_\sharp L(\kappa_2)_\sharp \kappa_1^*=
L(\nu_1)_\sharp L(\kappa_1)_\sharp \kappa_1^*\r
L(\nu_1)_\sharp.
}
Again, by composition, we obtain a morphism
\beg{elr5e3}{L(\kappa_2)_\sharp \kappa_1^*\r
\nu_2^*L(\nu_2)_\sharp L(\kappa_2)_\sharp \kappa_1^*\r
\nu_2^*L(\nu_1)_\sharp
}
where the first morphism is given by a unit of adjunction. By composition, again, we obtain a morphism
\beg{elr5e4}{
\kappa_2^*L(\kappa_2)_\sharp \kappa_1^*\r
\kappa_2^*\nu_2^*L(\nu_1)_\sharp
=\kappa_1^*\nu_1^*L(\nu_1)_\sharp.
}
By another composition, we then obtain a morphism
\beg{elr5e5}{
L(\kappa_1)_\sharp\kappa_2^*L(\kappa_2)_\sharp \kappa_1^*\r
L(\kappa_1)_\sharp\kappa_1^*\nu_1^*L(\nu_1)_\sharp\r
\nu_1^*L(\nu_1)_\sharp
}
which is a morphism from the left hand side of \rref{lr5e1} to the second summand of the 
right hand side, as needed. Summing with \rref{elr5e0}, 
we obtain a morphism from the left hand side to the right hand side
of \rref{lr5e1}, which is checked to be an isomorphism in the derived category by 
calculation on co-Verma modules.
\end{proof}

\vspace{3mm}
Sussan's Diagram relation 5 is represented graphically by Figure 5.

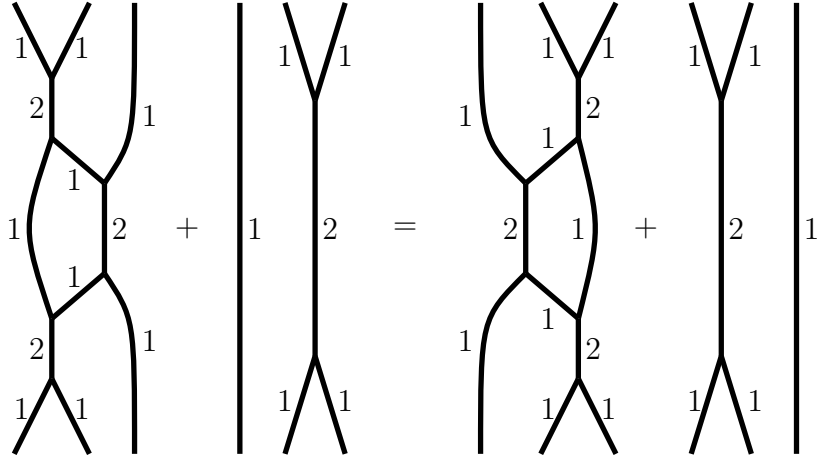
\begin{figure}[H]
\label{relation5}
\begin{tikzpicture}[line width=2pt]

\draw (-3.6, 3) .. controls (-3.6, 1.2) .. (-4, 0.6);
\node at (-3.4, 1.5) {1};
\draw (-4, 0.6) -- (-4, -0.6);
\node at (-3.8, 0) {2};
\draw (-3.6, -3) .. controls (-3.6, -1.2) .. (-4, -0.6);
\node at (-3.4, -1.5) {1};
\draw (-4, 0.6) -- (-4.7, 1.2);
\node at (-4.4, 0.65) {1};
\draw (-4, -0.6) -- (-4.7, -1.2);
\node at (-4.4, -0.65) {1};
\draw (-4.7, 1.2) .. controls (-5.1, 0) .. (-4.7, -1.2);
\node at (-5.2, 0) {1};
\draw (-4.7, 1.2) -- (-4.7, 2);
\node at (-4.9, 1.6) {2};
\draw (-4.7, -1.2) -- (-4.7, -2);
\node at (-4.9, -1.6) {2};
\draw (-4.2, 3) -- (-4.7, 2);
\node at (-4.3, 2.4) {1};
\draw (-5.2, 3) -- (-4.7, 2);
\node at (-5.1, 2.4) {1};
\draw (-4.2, -3) -- (-4.7, -2);
\node at (-4.3, -2.4) {1};
\draw (-5.2, -3) -- (-4.7, -2);
\node at (-5.1, -2.4) {1};

\node at (-2.9, 0) {+};

\draw (-2.2, 3) -- (-2.2, -3);
\node at (-2, 0) {1};
\draw (-1.6, 3) -- (-1.2, 1.7);
\node at (-1.6, 2.3) {1};
\draw (-0.8, 3) -- (-1.2, 1.7);
\node at (-0.8, 2.3) {1};
\draw (-1.2, 1.7) -- (-1.2, -1.7);
\node at (-1, 0) {2};
\draw (-1.6, -3) -- (-1.2, -1.7);
\node at (-1.6, -2.3) {1};
\draw (-0.8, -3) -- (-1.2, -1.7);
\node at (-0.8, -2.3) {1};

\node at (0, 0) {=};

\draw (1, 3) .. controls (1, 1.2) .. (1.6, 0.6);
\node at (0.8, 1.5) {1};
\draw (1.6, 0.6) -- (1.6, -0.6);
\node at (1.4, 0) {2};
\draw (1, -3) .. controls (1, -1.2) .. (1.6, -0.6);
\node at (0.8, -1.5) {1};
\draw (1.6, 0.6) -- (2.3, 1.2);
\node at (1.9, 1.2) {1};
\draw (1.6, -0.6) -- (2.3, -1.2);
\node at (1.9, -1.2) {1};
\draw (2.3, 1.2) .. controls (2.6, 0) .. (2.3, -1.2);
\node at (2.3, 0) {1};
\draw (2.3, 1.2) -- (2.3, 2);
\node at (2.5, 1.6) {2};
\draw (2.3, -1.2) -- (2.3, -2);
\node at (2.5, -1.6) {2};
\draw (1.8, 3) -- (2.3, 2);
\node at (1.9, 2.4) {1};
\draw (2.8, 3) -- (2.3, 2);
\node at (2.7, 2.4) {1};
\draw (1.8, -3) -- (2.3, -2);
\node at (1.9, -2.4) {1};
\draw (2.8, -3) -- (2.3, -2);
\node at (2.7, -2.4) {1};

\node at (3.2, 0) {+};

\draw (4.2, 1.7) -- (4.2, -1.7);
\node at (4.4, 0) {2};
\draw (4.6, 3) -- (4.2, 1.7);
\node at (4.65, 2.3) {1};
\draw(3.8, 3) -- (4.2, 1.7);
\node at (3.85, 2.3) {1};
\draw (4.6, -3) -- (4.2, -1.7);
\node at (4.65, -2.3) {1};
\draw(3.8, -3) -- (4.2, -1.7);
\node at (3.85, -2.3) {1};
\draw (5.2, 3) --(5.2, -3);
\node at (5.4, 0) {1};

\end{tikzpicture}
\caption{Diagram Relation 5}
\end{figure}

 It is expressed by the following theorem.

\begin{theorem}\label{tr5}
We have an isomorphism in the derived category
\beg{tr5e1}{\begin{array}{l}
\kappa_1^*L(\kappa_1)_\sharp\kappa_2^*L(\kappa_2)_\sharp
\kappa_1^*L(\kappa_1)_\sharp \vee
\kappa_2^*L(\kappa_2)_\sharp[2]\cong\\
\kappa_2^*L(\kappa_2)_\sharp\kappa_1^*L(\kappa_1)_\sharp
\kappa_2^*L(\kappa_2)_\sharp \vee
\kappa_1^*L(\kappa_1)_\sharp[2].
\end{array}
}
\end{theorem}
\vspace{2mm}

Note that these compositions of functors begin and end in $\mathcal{O}_{\mathbf{p}, n, k}$, as 
$\mathbf{p}$ has the Levi factor $gl_1 \oplus gl_1 \oplus gl_1$. In the graphic representation, it corresponds to the 
fact that all parts of the figure start and end with three strands, each labeled by $1$.  

\begin{proof}
By Lemma \ref{lr5}, we have isomorphisms in the derived category
$$\begin{array}{l}
\kappa_1^*L(\kappa_1)_\sharp\kappa_2^*L(\kappa_2)_\sharp
\kappa_1^*L(\kappa_1)_\sharp \vee
\kappa_2^*L(\kappa_2)_\sharp[2]\cong\\
\kappa_1^*\nu_1^*L(\nu_1)_\sharp L(\kappa_1)_\sharp \vee
\kappa_1^*L(\kappa_1)_\sharp[2]\vee
\kappa_2^*L(\kappa_2)_\sharp[2]\cong\\
\kappa_2^*\nu_2^*L(\nu_2)_\sharp L(\kappa_2)_\sharp \vee
\kappa_1^*L(\kappa_1)_\sharp[2]\vee
\kappa_2^*L(\kappa_2)_\sharp[2]\cong\\
\kappa_2^*L(\kappa_2)_\sharp\kappa_1^*L(\kappa_1)_\sharp
\kappa_2^*L(\kappa_2)_\sharp \vee
\kappa_1^*L(\kappa_1)_\sharp[2]
\end{array}
$$
as claimed.
\end{proof}

\end{document}